\documentclass[a4paper]{amsart}
\usepackage{amsmath,amsthm,amssymb}
\usepackage[T1]{fontenc}
\usepackage[all]{xy}
\usepackage{url}
\usepackage{amsmath,amsthm,amssymb}
\usepackage{scalerel,stackengine}
\usepackage{verbatim}
\usepackage{mathrsfs}
\usepackage{graphics}
\usepackage{color}
\usepackage{todonotes}
\usepackage{multirow}
\usepackage{subfigure}
\usepackage{hyperref}

\newtheorem{theorem}{Theorem}[section]
\newtheorem{proposition}[theorem]{Proposition}
\newtheorem{lemma}[theorem]{Lemma}
\newtheorem{corollary}[theorem]{Corollary}
\newtheorem*{conjecture}{Conjecture}

\theoremstyle{remark}
\newtheorem{remark}[theorem]{Remark}

\newtheorem{example}{Example}

\numberwithin{equation}{section}

\newcommand{\R}{{\mathbb{R}}}

\newcommand{\Q}{{\mathbb{Q}}}
\newcommand{\C}{{\mathbb{C}}}
\newcommand{\Z}{{\mathbb{Z}}}
\newcommand{\N}{{\mathbb{N}}}
\newcommand{\T}{{\mathbb{T}}}

\newcommand{\xbm}{(X,\mathcal{B},\mu)}

\newcommand{\Int}{\operatorname{Int}}

\newcommand{\SL}{\operatorname{SL}}

\let\oldmarginpar\marginpar
\renewcommand\marginpar[1]{\-\oldmarginpar[\raggedleft\footnotesize #1]%
{\raggedright\footnotesize #1}}

\stackMath
\newcommand\reallywidehat[1]{%
\savestack{\tmpbox}{\stretchto{%
  \scaleto{%
    \scalerel*[\widthof{\ensuremath{#1}}]{\kern-.6pt\bigwedge\kern-.6pt}%
    {\rule[-\textheight/2]{1ex}{\textheight}}%WIDTH-LIMITED BIG WEDGE
  }{\textheight}%
}{0.5ex}}%
\stackon[1pt]{#1}{\tmpbox}%
}
\begin{document}

\title[On ergodicity of foliations on covers of half-translation surfaces]{On ergodicity of foliations on $\Z^d$-covers of half-translation surfaces
 %with vanishing Lyapunov exponents
 and some applications to periodic systems of Eaton lenses}

\author[K.\ Fr\k{a}czek]{Krzysztof Fr\k{a}czek}
\address{Faculty of Mathematics and Computer Science, Nicolaus
Copernicus University, ul. Chopina 12/18, 87-100 Toru\'n, Poland}
\email{fraczek@mat.umk.pl}
\thanks{The first author was partially supported by the Narodowe Centrum Nauki Grant
2014/13/B/ST1/03153.}

\author[M.\ Schmoll]{Martin Schmoll}\address{Department of Mathematical Sciences, Clemson University, Clemson SC, 29634, USA }\email{schmoll@clemson.edu}
\thanks{The second author was partially supported by Simons Collaboration Grant 318898.}

%\address{} \email{}
%\date{\today}
%
%\keywords{}

\maketitle
\begin{abstract}
We consider the geodesic flow defined by periodic Eaton lens patterns in the plane and discover ergodic ones among those.
The ergodicity result on Eaton lenses is derived from a result for quadratic differentials on the plane that are pull backs
of quadratic differentials on tori.  Ergodicity itself is concluded for $\Z^d$-covers of quadratic differentials on
compact surfaces with vanishing Lyapunov exponents.
%The general result uses the structure Theorem for orbit closures of Eskin, Mirzakhani and Mohammedi.
\end{abstract}

\section{Introduction}
\subsection{Periodic Eaton lens distributions in the plane}
An \emph{Eaton lens} is a circular lens on the plane $\R^2$
which acts as a perfect retroreflector, i.e.\ so
that each ray of light after passing through the Eaton lens is
directed back toward its source,  see Figure~\ref{eaton}.
\begin{figure}[!htb]
 \centering
\includegraphics[width=0.8\textwidth]{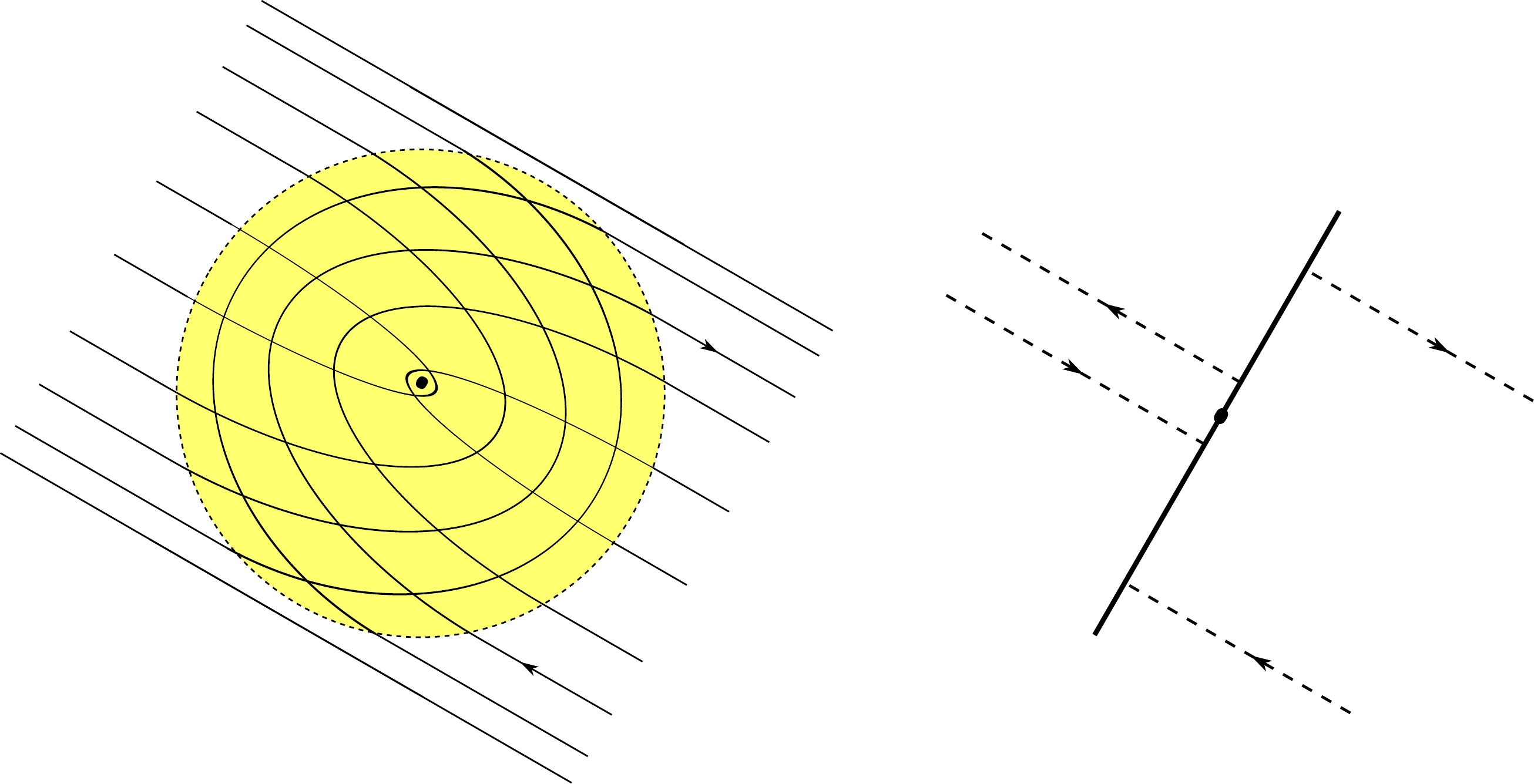}
 \caption{Light rays passing through an Eaton lens and its flat counterpart}
 \label{eaton}
\end{figure}
More precisely, if an Eaton lens is of radius $R>0$ then the refractive index (RI for short) inside lens depends only on the distance from the center $r$
and is given by the formula $n(x,y)=n(r)=\sqrt{{2R}/{r}-1}$. The refractive index $n(x,y)$ is constant and equals $1$ outside the lens.

In this paper we consider dynamics of light rays in periodic Eaton lens distributions in the plane $\R^2 \cong \C$.
As a simple example take a lattice $\Lambda \subset \R^2$ and consider an Eaton lens of radius $R>0$ centered at each
lattice point of $\Lambda$. This configuration of lenses will be denoted by $L(\Lambda,R)$

Let us call an Eaton lens distribution, say $\mathcal{L}$,
in $\R^2$ admissible, if no pair of lenses intersects.
%Let us call a lattice $R$-{\em admissible}, if (open) balls of radius
%$R$ centered at the points of $\Lambda$ do not intersect.
For every admissible Eaton lens configuration $\mathcal{L}$  the dynamics of the light rays can
be considered as a geodesic flow $(\mathfrak{g}^\mathcal{L}_t)_{t\in\R}$ on the unit tangent bundle of $\R^2$ with lens centers removed, see
Section~\ref{system_of_Eaton_lenses} for details. The Riemannian metric inducing the flow is
given by $g_{(x,y)}=  n (x,y)\cdot (dx\otimes dx +dy\otimes dy )$, where $n(x,y)$ is the refractive index at point $(x,y)$.

Since each Eaton lens in $\mathcal{L}$ acts as a perfect retroreflector, for any given slope $\theta \in \R/\pi\Z$ there is an invariant set $\mathscr{P}_{\mathcal{L},\theta}$ in the unit tangent bundle, such that all trajectories on  $\mathscr{P}_{\mathcal{L},\theta}$ have direction $\theta$ or $\theta+\pi$ outside the lenses. The restriction of the geodesic flow $(\mathfrak{g}^\mathcal{L}_t)_{t\in\R}$ to  $\mathscr{P}_{\mathcal{L},\theta}$ will be
denoted by $(\mathfrak{g}^{\mathcal{L},\theta}_t)_{t\in\R}$. Moreover,
$(\mathfrak{g}^{\mathcal{L},\theta}_t)_{t\in\R}$ possesses a natural invariant infinite measure $\mu_{\mathcal{L},\theta}$ equivalent to the Lebesgue measure on $\mathscr{P}_{\mathcal{L},\theta}$,  see
Section~\ref{system_of_Eaton_lenses} for details.
With respect to this setting we consider measure-theoretic questions.

In \cite{Fr-S} for example the authors have shown, that simple periodic Eaton lens configurations, for example  $L(\Lambda,R)$, have the opposite behavior of {\em ergodicity}. More precisely, a light ray in an Eaton lens
configuration is called {\em trapped}, if the ray never leaves a strip parallel to a line in $\R^2$.
The trapping phenomenon observed in \cite{Fr-S} was extended in \cite{Fr-Shi-Ul} to the following result:

\begin{theorem}
If $L(\Lambda,R)$ is an admissible configuration then
for a.e.\ direction $\theta\in \R/\pi\Z$
there exist constants $C=C(\Lambda,R,\theta)>0$
and $v =v(\Lambda,R,\theta) \in \R/\pi\Z$, such that every orbit in $\mathscr{P}_{L(\Lambda,R),\theta}$
is trapped in an infinite band of width $C>0$ in direction $v$.
\end{theorem}
Knieper and Glasmachers \cite{Gl-Kn,Gl-Kn1} have trapping results for geodesic flows on Riemannian planes.
Among other things Theorem~2.4 in \cite{Gl-Kn1} says, that for all Riemann metrics on the plane
that are pull backs of Riemann metrics on a torus with vanishing topological entropy,
the geodesics are trapped. Nevertheless the trapping phenomena obtained in \cite{Gl-Kn,Gl-Kn1} and \cite{Fr-S,Fr-Shi-Ul} have different flavors. The former is transient whereas the latter is recurrent.

Let us further mention that Artigiani describes a set of exceptional triples $(\Lambda,R,\theta)$
for which the flow $(\mathfrak{g}^{L(\Lambda,R),\theta}_t)_{t\in\R}$ is ergodic in \cite{Art}.

In this paper we investigate ergodicity and trapping for
more complicated periodic Eaton lens distributions.
In fact, given  a lattice $\Lambda \subset \C$
let us denote a $\Lambda$-periodic distribution of $k$ Eaton lenses with center
$c_i \in \C$ and radius $r_i \geq 0$ for $i=1,\ldots, k$ by $ L(\Lambda, c_1, \ldots, c_k, r_1, \ldots, r_k )$.
Of course, we will only consider admissible configurations.
If the list of Eaton lenses has centrally symmetric pairs, we write $\pm c_i$ for their centers and
list their common radius only once. We adopt the convention that if the radius of a lens is zero then this lens disappears.

For a random choice of admissible parameters in this family of configurations in Section\ref{sec:typtrap} we prove trapping.
\begin{theorem}\label{thm:typtrap}
For every lattice $\Lambda\subset\C$, every vector of centers $\overline{c}\in \C^k$ and almost every $\overline{r}\in\R^k_{>0}$ such that $L(\Lambda,\overline{c},\overline{r})$ is admissible the geodesic flow on $\mathscr{P}_{L(\Lambda,\overline{c},\overline{r}),\theta}$ is trapped for a.e.\ $\theta\in \R/\pi\Z$.
\end{theorem}

\subsection*{An admissible ergodic Eaton lens configuration in the plane}
As a consequence we have that the set of parameters $(\Lambda,\overline{c},\overline{r},\theta)$ for which  $(\mathfrak{g}^{L(\Lambda,\overline{c},\overline{r}),\theta}_t)_{t\in\R}$ is ergodic is very rare.
Despite this, in this paper, we find exceptional one-dimensional ergodic sets (piecewise smooth curves) of parameters such that a random choice inside such a curve provides an ergodic behavior of light rays.
%There are ergodic admissible Eaton lens foliations. Unlike trapping configurations those are rare and,  in a sense we will precise later, isolated.
In fact the configurations we found are curves
\[ \theta \longmapsto L(\Lambda_{\theta}, c_1(\theta), \ldots, c_k(\theta), r_1(\theta), \ldots, r_k(\theta) )\]
parameterized with the angle $\theta\in\R/\pi\Z$. We should stress that results of \cite{Fr-Shi-Ul} essentially show, that ergodic curves do not exists when $k=1$.

The simplest curve, described below
is a loop defined for every angle $\theta\in[0,\pi]$.
To start we take the function $l(\theta):= 2-\cot \theta(1-  \cot \theta)$ and consider
the curve of lattices
$$\Lambda_{\theta}= \Z (0,4) \oplus \Z(4,2)\  \text{ for }\  \theta  \bmod \pi \in [-\pi/4, \pi/4]$$
 continued by
$$\Lambda_{\theta} = \Z(0,4) \oplus \Z (2l(\theta), 2) \text{ for }\  \theta  \bmod \pi \in [\pi/4, 3\pi/4]. $$
Both families of lattices agree on the respective boundaries of
their defining intervals and so we obtain a continuous loop of lattices since  $\Lambda_{\pi}=\Lambda_0$.
Next define the curve $\theta\mapsto \gamma_W(\theta)$ of admissible Eaton lens configurations
%is Whitehead equivalent to the Wollmilchsau foliation and
for every $\theta \in \R/\pi \Z$ as follows:
\begin{equation*}
 \gamma_W(\theta) = \begin{cases}
   L\left(\Lambda_{\theta}, (0,0), \pm (1,1+\tan \theta), 2\sin\theta, \cos \theta \right)
  &  \text{if } \theta  \bmod \pi \in [0, \pi/4]\\
  L(\Lambda_{\theta}, (0,0), \pm (\cot \theta,2), l(\theta)\sin \theta, \cos \theta)
  & \text{if } \theta  \bmod \pi \in [\pi/4, \pi/2] \\
  L(\Lambda_{\theta}, (0,0), \pm (-\cot \theta , 2), l(\theta)\sin \theta, -\cos \theta)
  &  \text{if } \theta  \bmod \pi \in [\pi/2, 3\pi/4] \\
  L\left(\Lambda_{\theta}, (0,0), \pm (-1,1+\tan \theta), 2\sin\theta, -\cos \theta \right)
  &  \text{if } \theta  \bmod \pi \in [3\pi/4, \pi]
  \end{cases}
\end{equation*}
We want to assume, that two Eaton lens configurations in the plane are the same, if they differ
by a translation. After all, that is equivalent to a translation of the origin, preserving dynamical properties. Then
the curve of Eaton lens distribution closes, since $\gamma_W(0)=\gamma_W(\pi)+(0,2)$.
The admissibility of all Eaton lens configurations in the image of $\gamma_W$
is shown in Section~\ref{subsec:adm}. To give a geometric outline of the lens configurations
we add a cartoon showing the configurations at representative angles
in the interval $[0,\pi/4]$ (Figure~\ref{german_animal_1}) and $[\pi/4,\pi/2]$ (Figure~\ref{german_animal_3}).
 \begin{figure}[!htb]
 \centering
 \includegraphics[scale=0.40]{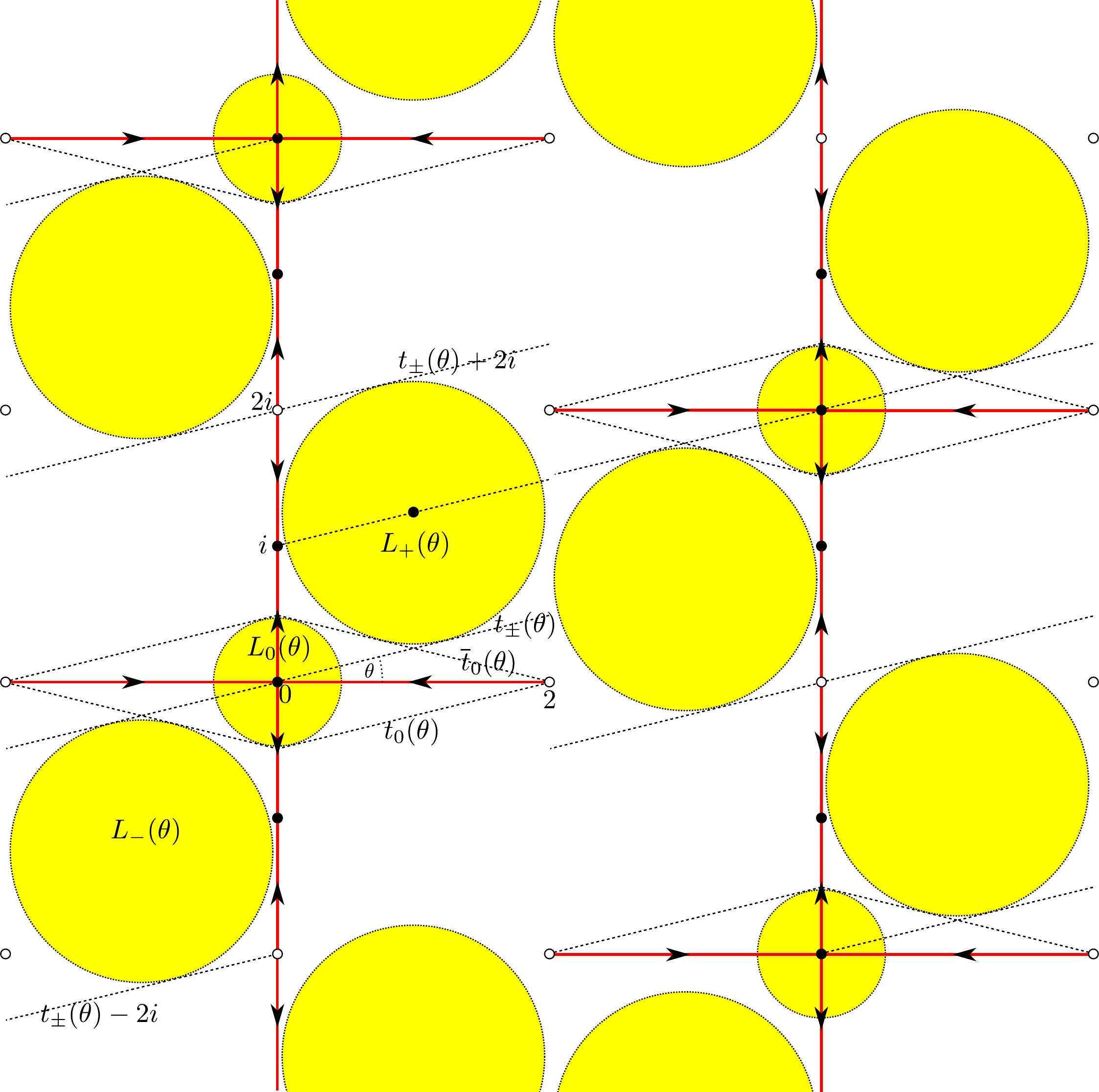}
 \caption{Ergodic curve for angles $|\theta | \leq \pi/4$ and  $|\theta - \pi| \leq \pi/4$}
 \label{german_animal_1}
\end{figure}
\begin{figure}[!htb]
 \centering
 \includegraphics[scale=0.40]{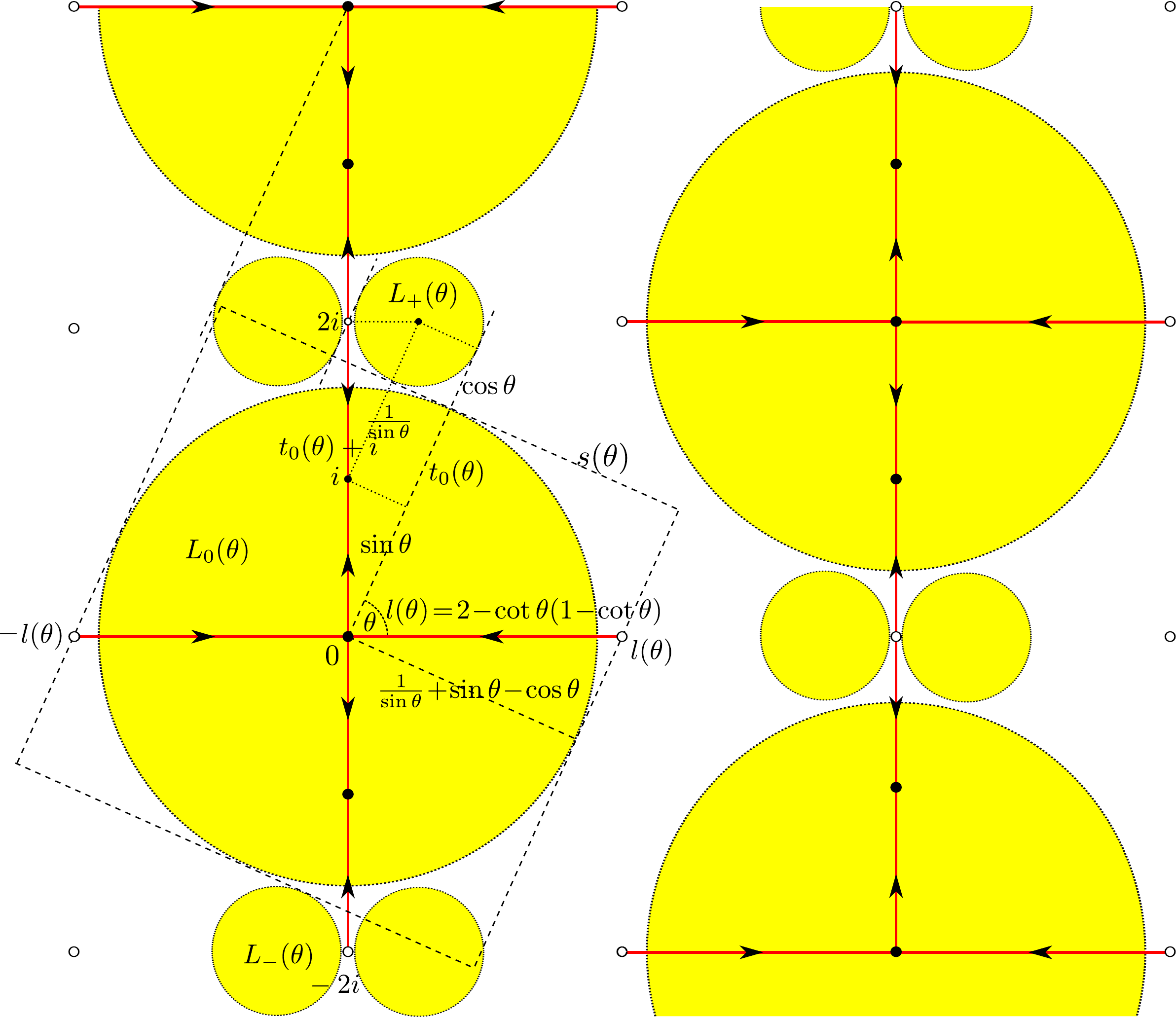}
 \caption{Ergodic curve for the angles $|\theta \pm \pi/2| \leq \pi/4 $ }
 \label{german_animal_3}
\end{figure}
\begin{theorem}\label{thm:erglens}
For almost every $\theta\in\R/\pi\Z$ the geodesic flow $(\mathfrak{g}^{\gamma_W(\theta),\theta}_t)_{t\in\R}$ is ergodic.
\end{theorem}
We devote part of the paper showing several curves of ergodic Eaton lens configurations in the plane, see Figures~\ref{Eaton_C_6(3,1)}, \ref{Eaton_C_6(3,2)} and \ref{Eaton_C_3(2,1)}.
For some of those curves we describe admissible Eaton lens configurations only for an interval
of slopes in $\R/ \pi \Z$.

\subsection*{Reduction to quadratic differentials and cyclic pillow case covers}
The dynamical results for periodic Eaton lens distributions in the plane rely on the equivalence
of the Eaton dynamics in a fixed direction, say $\theta$, to the (dynamics on a) direction foliation
$\mathcal{F}_{\theta}(q)$ of a quadratic differential $q$ on the plane. Starting from a (slit-fold) quadratic differential, the connection is made by replacing a {\em slit-fold}, as shown in Figure \ref{slit-fold}, by an Eaton lens. For a given direction the dynamical equivalence of a slit-fold and an Eaton lens is motivated by
Figure~\ref{eaton}. This equivalence is described in detail in Section~\ref{system_of_Eaton_lenses}.
\begin{figure}[!htb]
 \centering
\includegraphics[scale=0.50]{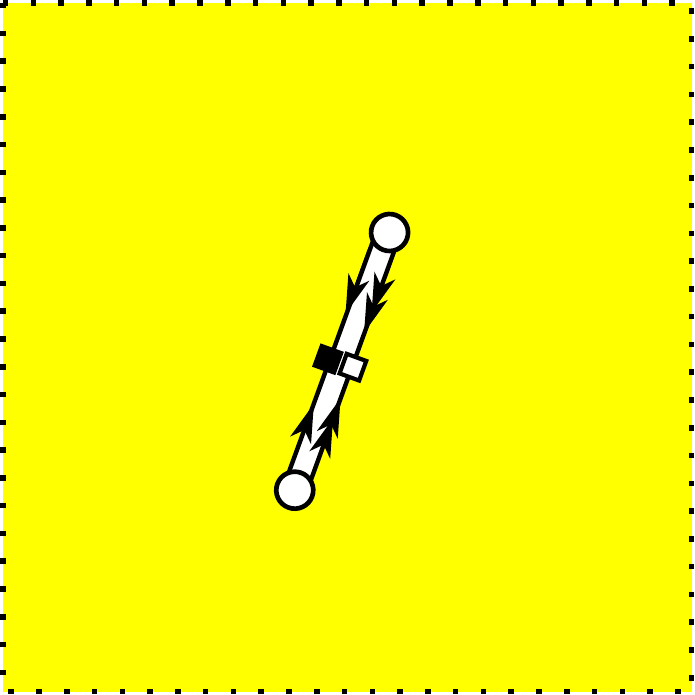}
 \caption{A slit-fold}
 \label{slit-fold}
\end{figure}
We distinguish two objects, a {\em flat lens} is a two-dimensional replacement of an Eaton lens
perpendicular to the light direction, that does not change the future and the past of the light in the
complement of the Eaton lens that is replaced, see Figure \ref{eaton}.
A {\em slit-fold} on the other hand is a flat lens in the language of quadratic differentials.
In fact a slit-fold is constructed by removing a line segment, say $[a,b]$ with $a,b\in \C$, from the plane (or any flat surface), then a closure is taken so that the removed segment is replaced by two parallel and disjoint segments. Then for each segment one identifies those pairs of points, that have equal distance from the segments center point.
Once this is done we obtain a slit-fold that we denote by $\rangle a, b \langle$ on the given surface, see Figure \ref{slit-fold}.
The single slit-fold $\rangle a, b \langle$ defines a quadratic differential on the plane with two singular points located on the (doubled) centers of the segment and a zero at its (identified) endpoints.  Alternatively that quadratic differential on the plane is obtained as quotient of the abelian differential defined by gluing two copies of the slit plane $\C \backslash [ a, b ]$ crosswise along its strands.
Then a quotient is taken with respect to the sheet exchange map that lifts the rotation by $\pi$ around
the center point of $[a,b]$. By adding slit-folds we can construct a variety of quadratic differentials on any flat surface.

For fixed $k \in \N$ the set $\mathcal{S}_k$ of quadratic differentials made of $k$ disjoint slit folds is a subset of $\mathcal{Q}((-1)^{2k}, 2^{k})$,  the vector space of genus one quadratic differentials that have $2k$ singular points and $k$ cone points of order $2$.
Disjoint means, the cone points of different slit-folds do not fall together.
We will use the superset $\overline{\mathcal{S}}_k \supset \mathcal{S}_k$ of the quadratic differentials that
are made of exactly $k$ slit-folds, including the ones with merged cone points.
\begin{figure}[!htb]
 \centering
\includegraphics[width=0.9\textwidth]{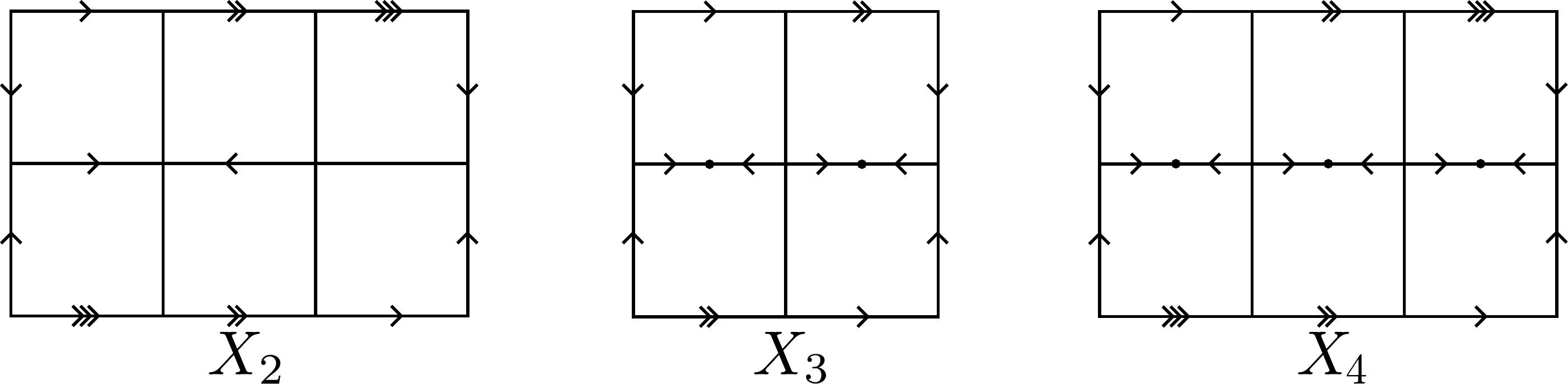}
 \caption{Torus quadratic differentials $X_2$, $X_3$ and $X_4$.}
 \label{3animals}
\end{figure}
Let us consider three special quadratic surfaces $X_2\in \overline{\mathcal{S}}_2$, $X_3\in \overline{\mathcal{S}}_3$ and  $X_4\in \overline{\mathcal{S}}_4$
drawn on Figure~\ref{3animals}.
\begin{theorem}\label{thm:ergodicity:intr}
Let $X=X_k$ for $k=2,3,4$ and denote by $\widetilde{X}$ its universal cover (quadratic differential on the plane).
Then for almost every $\theta \in \R/\pi \Z$ the foliation in the direction $\theta$ on $\widetilde{X}$ is ergodic.
\end{theorem}
Those ergodic foliations on the plane can be converted into ergodic curves of admissible Eaton lens distributions.

The ergodicity of universal covers of quadratic surfaces in $S_k$ on the other hand is rather exceptional.
If $X\in S_k$ satisfies a separation condition on slit folds (which is an open condition) then  the foliation in the direction $\theta$ on $\widetilde{X}$ is trapped for a.e.\ $\theta\in \R/\pi\Z$,
see Corollary~\ref{cor:trap} for details.

The following more general ergodicity result supplies the key to the proof of Theorem~\ref{thm:ergodicity:intr} and Theorem~\ref{thm:erglens}.
\begin{theorem}\label{thm:mainquadr:intr}
Let $(X,q)$ be a quadratic differential on a compact, connected surface such that all Lyapunov exponents of the Kontsevich-Zorich cocycle of $(X,q)$ are zero. Then for every connected $\Z^d$-cover $(\widetilde{X},\widetilde{q})$, with $d\leq 2g$, almost every directional foliation on  $(\widetilde{X},\widetilde{q})$ is ergodic.
\end{theorem}
This result is in fact a consequence of the more general Theorem~\ref{thm:ergV} that
provides a criterion on ergodicity for translation flows on $\Z^d$-covers of compact translation surfaces.
We would like to mention that a similar result was obtained independently by Avila, Delecroix, Hubert and Matheus but it was never published (communicated by Pascal Hubert). Some related research was also recently done by Hooper
who studied ergodicity of directional flows on translation surfaces with infinite area, see e.g.\ \cite{Hoop}.

\section{Ergodic slit-fold configurations on planes by cyclic pillowcase covers.}
In this section we outline the strategy to construct the ergodic quadratic differentials on the plane assuming
the validity of Theorem \ref{thm:mainquadr:intr}.
%Thus the following ergodicity result for pullbacks of quadratic differentials along $\Z^d$ covers applies.
Theorem \ref{thm:mainquadr:intr} reduces the problem of ergodicity from cyclic quadratic differentials in the plane to quadratic differentials $(\mathcal{T}, q)$ on the torus $\mathcal{T}$ with {\em zero Lyapunov exponents}. A recent criterion of Grivaux and Hubert \cite{Gr-Hu} implies
that a cyclic cover of the {\em pillowcase} has zero Lyapunov exponents, if it is
branched at (exactly) three points. Now it turns out that there is a only a short list of those branched cyclic covers
$\mathcal{T} \rightarrow \mathcal{P}$.
\begin{figure}[!htb]
 \centering
\includegraphics[scale=0.50]{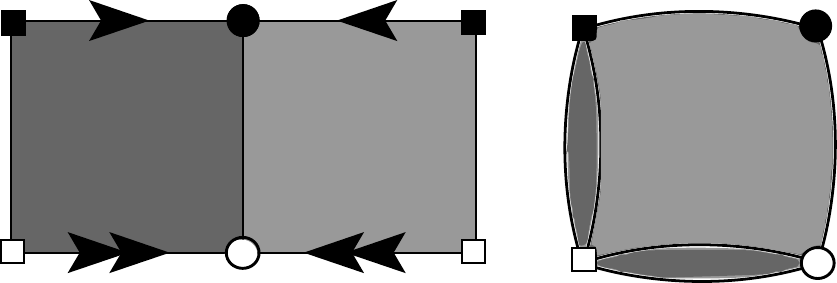}
 \caption{The pillowcase quadratic differential in polygonal representation}
 \label{pillowcase}
\end{figure}
Recall, the {\em pillowcase} $\mathcal{P}$ is a quadratic differential $q_{\mathcal{P}}$ on the sphere $S^2$.
To characterize it, consider the quadratic differential $dz^2$ on $\R^2 \cong \C$. It is invariant under translations
and the central reflection  $-id:\R^2\to\R^2$. Thus it descends to the torus $\T^2:=\R^2/\Z^2$
defining a quadratic differential invariant under the hyperelliptic involution $\varphi: \T^2 \rightarrow \T^2$
induced by the  central reflection of $\R^2$. So it further descends to a quadratic differential $q_{\mathcal{P}}$ on the quotient sphere $S^2=\T^2/ \varphi$.  The pillowcase the pair $\mathcal{P}=(S^2, q_{\mathcal{P}})$, see Figure \ref{pillowcase}.
%The singularities of $q_{\mathcal{P}}$ are the images of the Weierstrass points, in particular they are cone points of total angle $\pi$.
Putting the result from \cite{Gr-Hu} on cyclic pillowcase covers and Theorem  \ref{thm:mainquadr:intr} together one has:
\begin{corollary}\label{cor:pillow:intr}
Let $\pi : X \rightarrow \mathcal{P}$ be a finite cyclic cover branched over  three of the singular points of $\mathcal{P}$ and let $q = \pi^{\ast}q_{\mathcal{P}}$ be the pull back quadratic differential to $X$. If
$(\widetilde{X},\widetilde{q}) \rightarrow (X,q)$ is a connected $\Z^d$-cover with $d \leq 2g$, then almost every directional foliation on
$(\widetilde{X},\widetilde{q})$ is ergodic.
\end{corollary}
We further present a list of relevant pillow-case covers:
\begin{proposition}\label{prop:pillow:intr} Up to the action of $\text{SL}_2(\Z)$ on covers and up to isomorphy,
there are three cyclic covers $(\T, q) \rightarrow \mathcal{P}$
that are branched over exactly three cone points of $q_{\mathcal{P}}$.
The degree of each such cover is $3$, $4$ or $6$.
%There is a degree $3$ cover and a degree $4$ cover of this kind.  Further there is one nontrivial
%$\text{SL}_2(\Z)$ orbit of degree $6$ covers.
\end{proposition}
Figure \ref{three_animal} shows polygonal one strip representations of one cyclic pillowcase cover in each degree.
We note that the quadratic differential on the degree $3$ cover has the {\em Ornithorynque} (see \cite{Fo-Ma} for the description of the surface) as its orientation cover and the quadratic differential on the degree $4$ cover has the {\em Eierlegende Wollmilchsau} (see also \cite{Fo-Ma}) as its orientation cover.
\begin{figure}[!htb]
 \centering
 %\resizebox{9cm}{!}{\input three_animals.pdf}
 \includegraphics[scale=0.60]{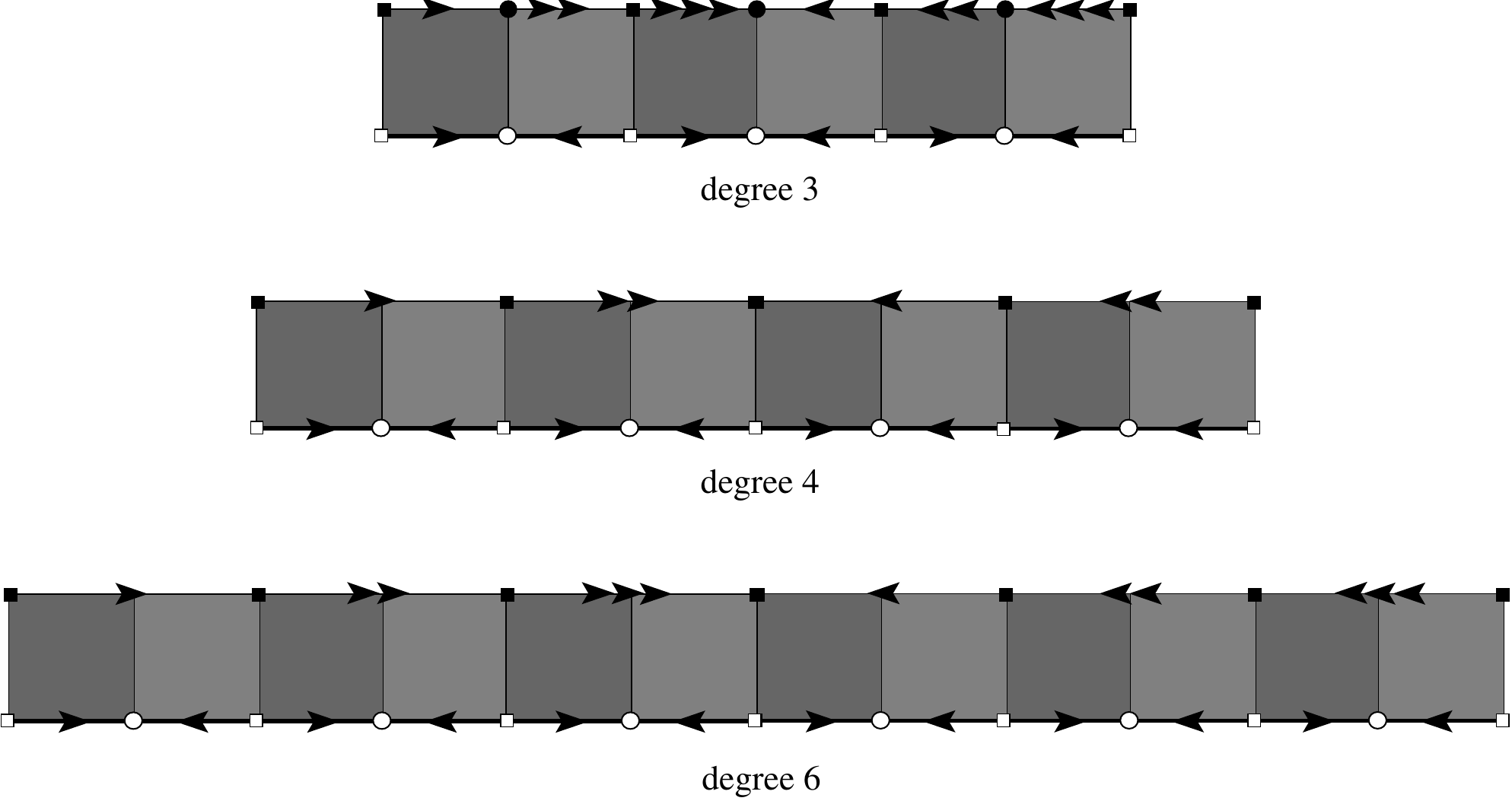}
 \caption{Torus differentials with zero Lyapunov exponents}
 \label{three_animal}
\end{figure}
There are particular questions regarding the conversion of a quadratic differential to an
admissible Eaton lens distribution in the plane. In order to convert the torus differentials
from Proposition \ref{prop:pillow:intr} to Eaton lens distributions one needs a cover
that is a slit-fold differential in the plane. We do this below
for the Eaton curve presented in the introduction.
The construction of some other curves need more
sophisticated geometric arguments which can be found in Appendix~\ref{app:foldsskel}.

\subsection*{Eaton differentials and skeletons}
For a fixed direction the (long term) Eaton lens dynamics on the plane or a torus is equivalent to the dynamics on a particular slit-fold, so we call a quadratic differential  that is given by a union of
slit-folds a {\em pre-Eaton differential}.
The radius of an Eaton lens replacing a slit-fold depends on the angle between the light ray and the slit-fold, a light direction needs to be specified for such a replacement. Recall that a configuration of Eaton lenses is {\em admissible}, if no pair of Eaton lenses intersects.
A pre-Eaton differential $q$  is called an {\em Eaton differential}, if there is a nonempty open interval $I \subset \R$ such that for every (light) direction
$\theta \in I \bmod \pi$ the direction foliation $\mathcal{F}_{\theta}(q)$ is measure equivalent to
the geodesic flow of an admissible Eaton lens configuration, whose lens centers and radii depend
continuously on $\theta \in I$. We further call an Eaton differential {\em maximal}, if
$I \rightarrow \R/\pi \Z$, $ x \mapsto x \bmod \pi$ is onto.
Finally let us call a (pre-)Eaton differential {\em ergodic}, if its direction foliations are ergodic
in almost every direction. Note, that a pre-Eaton differential must be located on a torus, or a plane, since it has no singular points besides the ones of its slit-folds. So it is enough to present a pre-Eaton differential by a union of slit-folds, that we will call {\em skeleton}. Below we introduce and use geometric as well as algebraic presentations of skeletons.
\begin{figure}[!htb]
 \centering
 \includegraphics[scale=0.50]{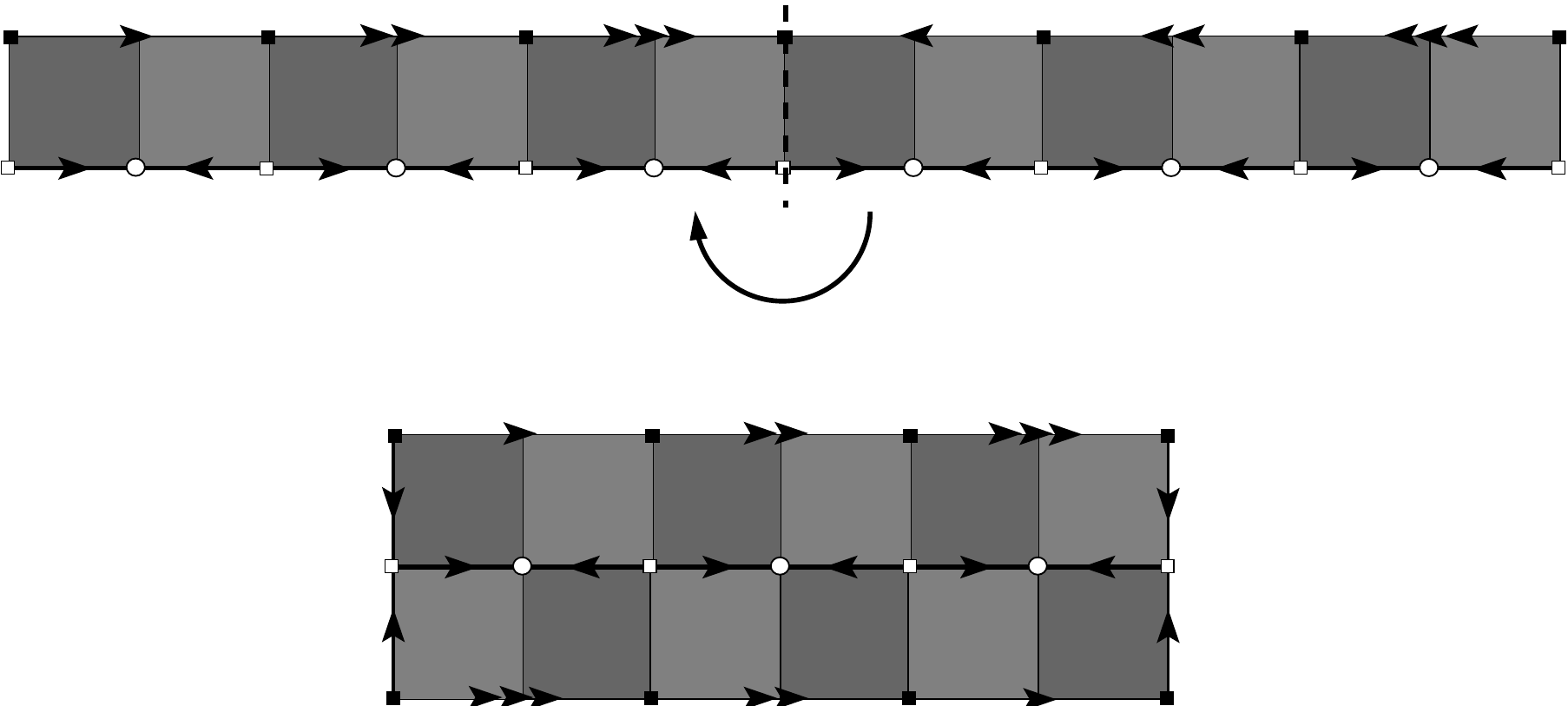}
 \caption{Cutting and turning polygonal pieces gives a pure slit-fold representation modulo absolute homology}
 \label{cut_turn}
\end{figure}
\begin{proof}[Proof of Theorem~\ref{thm:ergodicity:intr}]
Pre-Eaton differentials are obtained from all three torus differentials in Figure \ref{three_animal},
by first cutting vertically through their center and then rotating one of the halfes underneath the other
as in Figure \ref{cut_turn}. Up to rescaling the resulting pre-Eaton differentials are $X_2$ (from the degree $3$ cover), $X_3$ (from the degree $4$ cover) and $X_4$ (from the degree $6$ cover) as shown in Figure~\ref{3animals}. It follows, that $X_2$, $X_3$ and $X_4$ are cyclic covers of the pillowcase and branched over exactly three singularities of $\mathcal{P}$.
Passing to their universal covers we obtain three pre-Eaton differentials $\widetilde{X}_2$, $\widetilde{X}_3$, $\widetilde{X}_4$ on the plane. In view of Corollary~\ref{cor:pillow:intr} almost every directional foliation for every such differential is ergodic.
\end{proof}
%This polygonal representation has absolute homology generators of the torus located
%on the edges of the respective polygon.
Below we call the quadratic differential $\widetilde{X}_3$ on the complex plane obtained
from the degree $4$ pillowcase cover $X_3$ the \emph{Wollmilchsau differential}, see Figure~\ref{flat_german}.
\begin{figure}[!htb]
 \centering
 \includegraphics[width=1\textwidth]{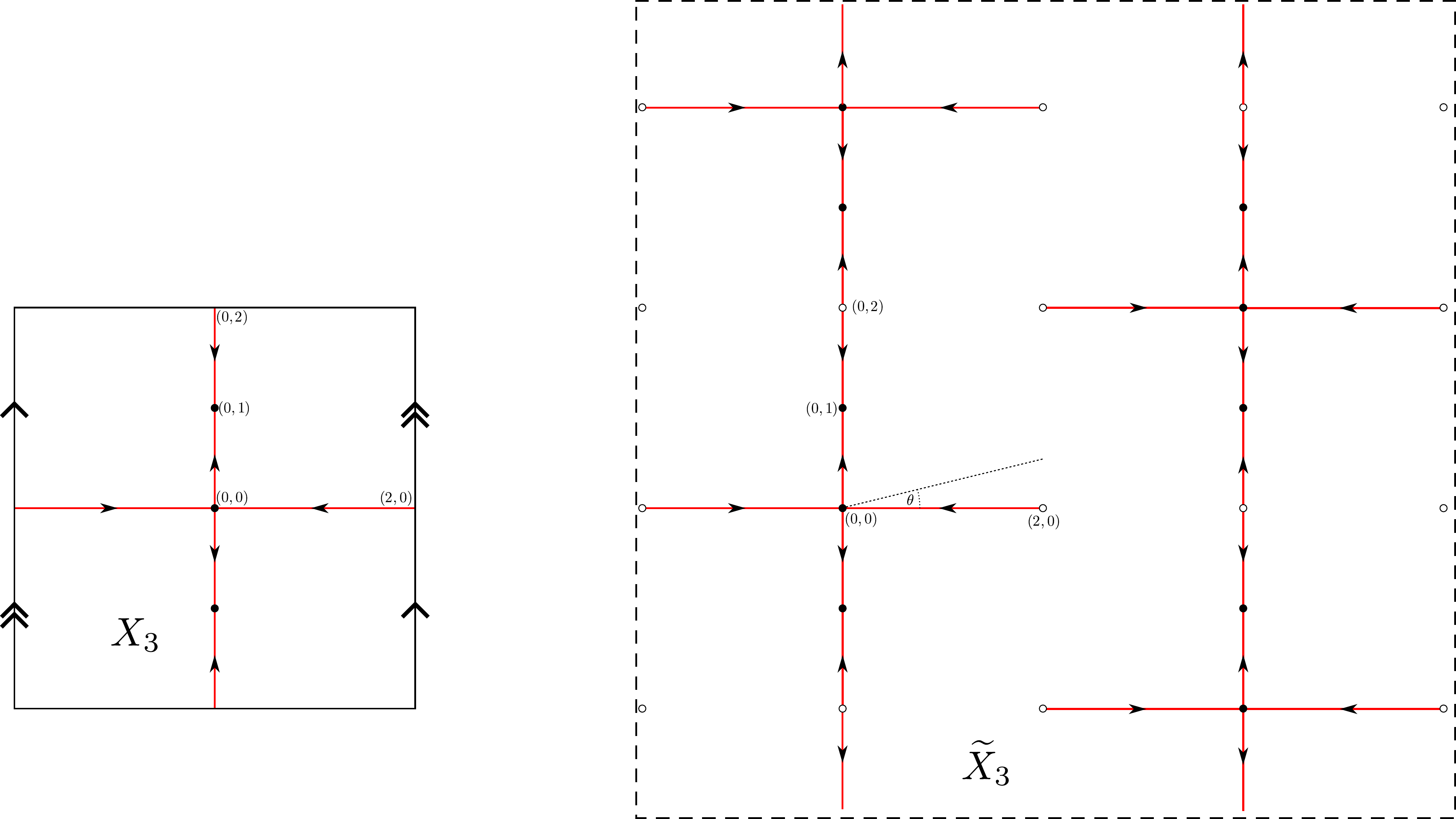}
 \caption{The quadratic surface $X_3$ (rotated by $\pi/2$) and its universal cover $\widetilde{X}_3$ (Wollmilchsau differential)}
 \label{flat_german}
\end{figure}

\begin{theorem}\label{thm:ergodicity:intr1}
The Wollmilchsau differential is an ergodic, maximal Eaton differential.
\end{theorem}
Ergodicity follows because Theorem \ref{thm:mainquadr:intr} applies. To show the other statements of the Theorem we need to describe an Eaton lens configuration depending continuously on $\theta \in \R/\pi\Z$ and show that it is admissible. This is done in Proposition~\ref{prop:adm}, see the comment after that.

Eaton lenses may overlap when placed at slit-fold centers.
To resolve this problem we deform the measured foliation tangential to its direction
$\theta \in \R/ \pi \Z$ to a measure equivalent foliation by moving slit-folds parallel to $\theta$.
More precisely take a direction foliation $\mathcal{F}_{\theta}(q)$ of a quadratic differential $q$ that contains a slit-fold.  Then changing the location of the slit-fold while
keeping its endpoints (and therefore its center points) on the same leaves of $\mathcal{F}_{\theta}(q)$ is called a {\em railed motion}.
Changing a slit-fold skeleton using railed motions is called a {\em railed deformation}. In terms of Teichm\"uller
Theory railed deformations are isotopies, or {\em Whitehead moves} that preserve the transverse measure of a measured foliation. In particular, two measured foliations that differ by railed deformations are Whitehead equivalent.
A Whitehead move is a deformation of a foliated surface that collapses a leaf connecting
two singular points, or it is the inverse of such a deformation, see \cite[page 116]{Papa}. Figure \ref{german_animal_1_transition} shows railed deformations deforming skeletons into disjoint slit-folds.
Each of those consists of several Whitehead moves.
Some railed motions are shown in Figure \ref{slit_conversion} to the left.  After performing a railed deformation, appropriately sized Eaton lenses are placed at the slit-fold centers.
\begin{figure}[!htb]
 \centering
 \includegraphics[width=1\textwidth]{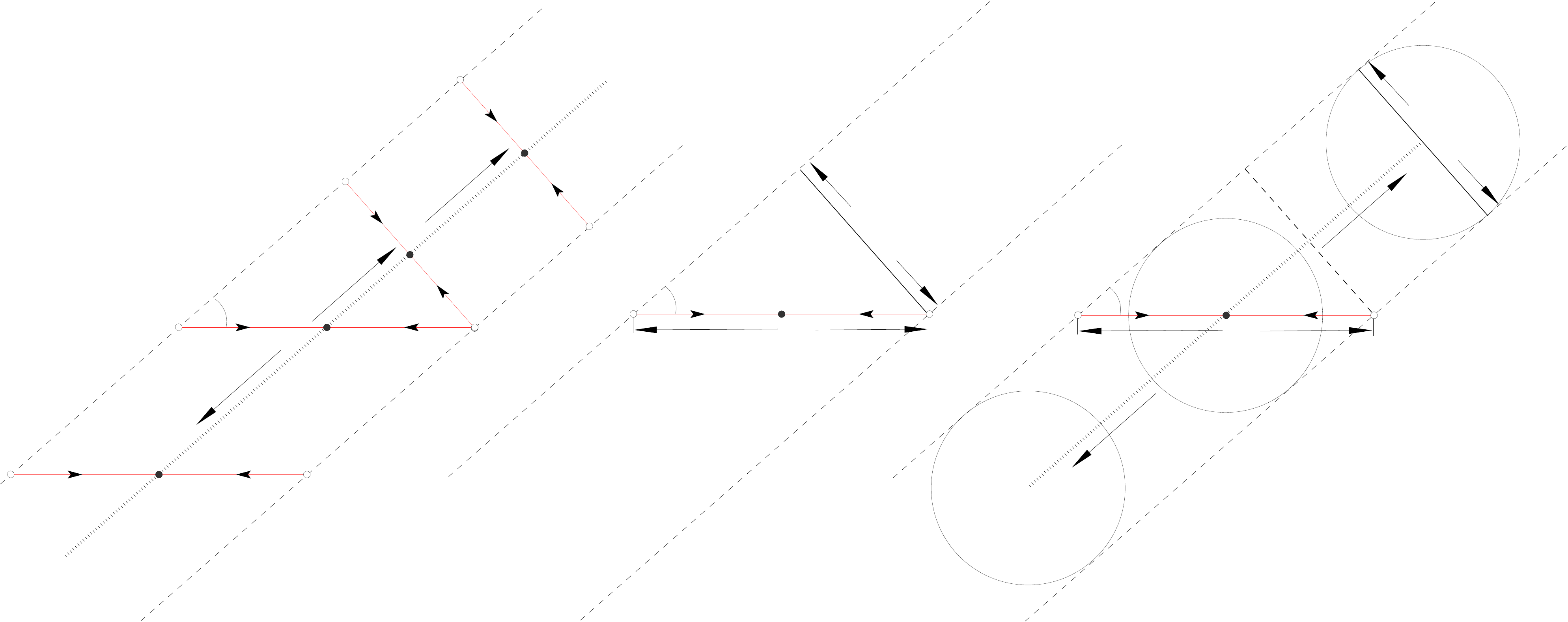}
 \caption{Railed moves of slit-folds and Eaton lenses in direction $\theta$}
 \label{slit_conversion}
\end{figure}

\subsection{The Eaton lens configurations along $\gamma_W$ are admissible}\label{subsec:adm}
The following result together with Theorem~\ref{thm:ergodicity:intr} gives the proof of Theorem \ref{thm:erglens}.
\begin{proposition}\label{prop:adm} The Eaton lens configurations defined by $\gamma_W(\theta)$
are admissible and for all $\theta \in [0,\pi]$ and the ergodicity of the geodesic flow
$(\mathfrak{g}^{\gamma_W(\theta),\theta}_t)_{t\in\R}$ is equivalent to the ergodicity of the
directional foliation generated by the Wollmilchsau differential $\widetilde{X}_3$ in the direction $\theta$.
\end{proposition}
\begin{proof}
For this proof we will use complex coordinates on the plane.
Let us consider the situation for light directions $\theta \in [0,\pi/4]$ first.
For those angles the Eaton lens configurations are periodic with respect to the lattice
$\Lambda:=\Z 4i \oplus \Z(4+2i) $. Therefore it is enough to show that Eaton lenses centered
inside the strip $S = \{z\in\C;|\Re z| \leq 2\}$ are pairwise disjoint and do not leave the strip, i.e.\
do not cross the boundary of the strip.

%Given a direction $\theta \in [0,\pi/4]$ the strip $S$ decomposes into parallelograms
%cut out by the $\Z 2i$ translates of the line in direction $\theta$ through the origin.
Modulo the action of $\Lambda$ there are three Eaton lenses on $\gamma_W(\theta)$.
The first one $L_0(\theta)$ has radius $r_0(\theta)=2\sin \theta$ and is centered at the origin.
Then there is a pair of lenses denoted by $L_{\pm}(\theta)$ centered at
$c_{\pm}(\theta)=\pm(1+i(1+\tan \theta))$, both of radius $r_{\pm}(\theta)= \cos \theta$, see Figure~\ref{german_animal_1}.
% show that those are nonintersecting and then relate their flow to the Wollmilchsau
%skeleton
Since the radius of the Eaton lenses $L_{\pm}(\theta)$ is less then $1$
and the radius of $L_{0}(\theta)$ is bounded by $2$, the lenses in the $\Lambda$ orbit
of any one of those three Eaton lenses are pairwise disjoint.
For the same reason the $\Z 4i$ orbit of all three Eaton lenses lies in the strip $S$.

The line in direction $\theta$ through the point $i$ contains the center of $L_{+}(\theta)$
since its slope is $\tan \theta$. The distance of that line to its parallel through the
origin, denoted by $t_{\pm}(\theta)$, is $\cos \theta$, equaling the radius of $L_{+}(\theta)$.
So the lines $t_{\pm}(\theta)$ and $t_{\pm}(\theta)+2i$
are tangent to $L_{+}(\theta)$. Then by central symmetry the lines $t_{\pm}(\theta)$ and
$t_{\pm}(\theta)-2i$ are tangents to $L_{-}(\theta)$. It follows that $L_{+}(\theta)+4ni$ lies between the lines
$t_{\pm}(\theta)+4ni$ and $t_{\pm}(\theta)+(4n+2)i$ and $L_{-}(\theta)+4ni$ lies between the lines
$t_{\pm}(\theta)+(4n-2)i$ and $t_{\pm}(\theta)+4ni$ for every $n\in\Z$. Therefore, no pair of Eaton lenses in the $\Z 4i$ orbits of $L_{\pm}(\theta)$
intersect. Since the $\Z(4+2i)$ translates of $S$ cover the whole plane, intersecting only in their boundary lines,
we conclude that no pair of Eaton lenses in the $\Lambda$ orbits of $L_{\pm}(\theta)$
intersect.

Since $L_0(\theta)$, the lens in the origin, has radius $2\sin \theta$ the line in direction
$\theta$ through $-2$, denoted by $t_0(\theta)$,
is tangent to it. By reflection symmetry with respect to the vertical
axis, the line through $2$ in direction $\pi - \theta$ is also a tangent to $L_0(\theta)$.
Let us denote this (tangent-)line by $\overline{t}_{0}(\theta)$, we shall see it is
also tangent to $L_{+}(\theta)$. Indeed, the reflection of
$\overline{t}_{0}(\theta)$ with respect to the vertical through the center of
$L_{+}(\theta)$ is the tangent $t_{\pm}(\theta)$.
Since the centers of  $L_{+}(\theta)$ and $L_0(\theta)$ lie on different sides
of their common tangent $\overline{t}_{0}(\theta)$ these lenses do not intersect.
By central symmetry the same is true for $L_{-}(\theta)$ and $L_0(\theta)$.
Since all three lenses $L_{\pm}(\theta)$ and $L_0(\theta)$
in the parallelogram in $S$ bounded by $t_{\pm}(\theta) \pm 2i$ are disjoint
and these parallelograms have a (modulo boundary) disjoint $\Lambda$ orbit,
we conclude that the lens distribution given by $\gamma_W(\theta)$
is disjoint for all $\theta \in [0,\pi/4]$.

\begin{figure}[!htb]
 \centering
 \includegraphics[width=0.9\textwidth]{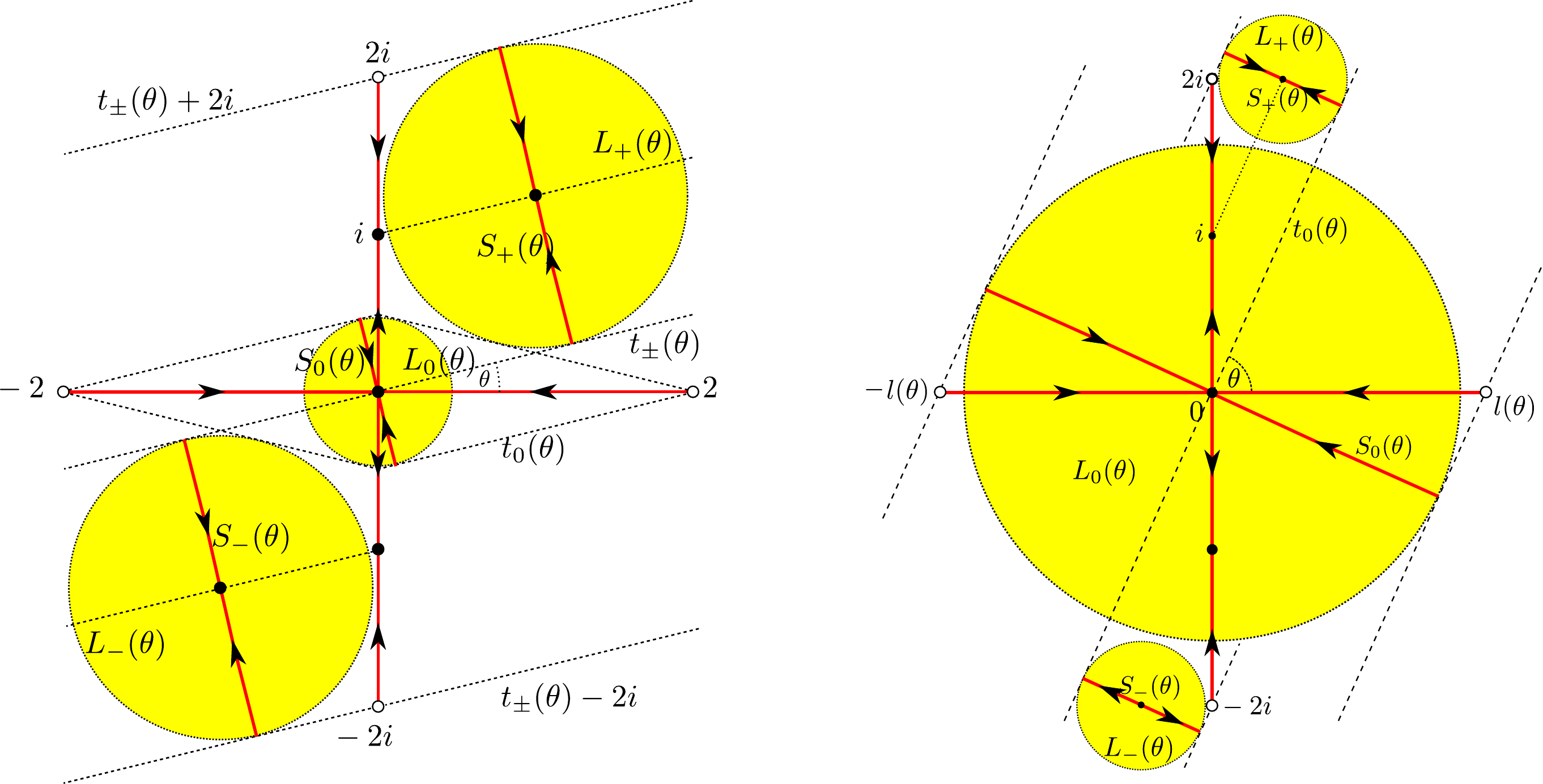}
 \caption{Transitions from Eaton lenses to the Wollmilchsau skeleton}
 \label{german_animal_1_transition}
\end{figure}

For the same interval of angles the geodesic flow $(\mathfrak{g}^{\gamma_W(\theta),\theta}_t)_{t\in\R}$
is measure equivalent to the direction $\theta$ dynamics defined by the surface $\widetilde{X}_3$.
First the results of Appendix~\ref{system_of_Eaton_lenses} imply, that for given $\theta \in [0,\pi/4]$ the ergodicity of the geodesic flow
is equivalent to the ergodicity of the measured foliation defined by the slit-fold distribution obtained
from the flat lens representation of Eaton lenses. That is, for given $\theta \in [0,\pi/4]$ we
replace every Eaton lens by a slit-fold centered at the lens' center, perpendicular to $\theta$
and with length equal to the diameter of the lens. In fact modulo $\Lambda$ we
obtain the slit-folds
\[S_{\pm}(\theta):=\pm \rangle 1+\cos \theta \sin \theta +   i(1+\tan \theta - \cos^2 \theta),
1-\cos \theta \sin \theta +   i(1+\tan \theta + \cos^2 \theta)
\langle\]  through the centers of $L_{\pm}(\theta)$ and
\[S_0(\theta):=\ \rangle -2\sin^2 \theta + 2i \sin \theta \cos \theta,  2\sin^2 \theta - 2i \sin \theta \cos \theta \langle\]
through the origin, see Figure~\ref{german_animal_1_transition}. The endpoints of the slit-fold $S_+(\theta)$ lie on the lines (and direction $\theta$ foliation leaves) $t_{\pm}(\theta)$ and $t_{\pm}(\theta)+2i$. That means we can perform a railed deformation of $S_+(\theta)$ along those leaves
terminating in the slit-fold $\rangle 0,2i  \langle$. By central symmetry there is
a railed deformation of $S_-(\theta)$ to the slit-fold $\rangle 0, -2i  \langle$.
The end points of the slit-fold $S_0(\theta)$ are located on the direction $\theta$ foliation leaves
through the point $\pm 2$, so $S_0(\theta)$ has a railed deformation to the slit-fold
$\rangle -2, 2  \langle$. But that means the skeleton
$S_+ \cup S_- \cup S_0$ is Whitehead equivalent to the skeleton
$ \rangle -2,2 \langle \ \cup \ \rangle 0,2i \langle \ \cup \ \rangle 0,-2i \langle$.
The $\Lambda$ orbit of the latter is the Wollmilchsau skeleton in the plane, showing
the claim on equivalence of ergodicity for angles $\theta \in [0,\pi/4]$.

The strategy we have just used to replace an Eaton lens with a slit-fold is the same for every angle. Let us describe this process for the slit-folds in the Wollmilchsau skeleton: For a fixed direction $\theta \in \R/\pi\Z$ a slit-fold, say $S$, replaces an Eaton lens, say $L$, if the two lines
in direction $\theta$ through the endpoints of $S$ are tangent to $L$. Step by step,
the flat lens equivalent to $L$ is in the quadratic differential interpretation the slit-fold $S_L$
perpendicular to the direction $ \theta$ with diameter and center matching those of $L$.
In that case, the endpoints of $S_L$ lie on the two said tangents to $L$ and therefore there is
a railed deformation of $S_L$ to $S$. If, as in our case, more than one slit-fold is involved
it must be checked that the tangent segments between $S$ and $S_L$ do not cross
another slit-fold. This is illustrated  in Figure \ref{german_animal_1_transition}
for an angle $\theta \in [0,\pi/4]$ (left) and for an angle $\theta \in [\pi/4, \pi/2]$ (right).
This same strategy is applied for the angles $\theta \in [\pi/4, \pi/2]$ below. The tangent
lines necessary to show equivalence to the Wollmilchsau skeleton are also needed to show admissibility.

For the angles $\theta \in [\pi/4, \pi/2]$
%the lenses $L_{\pm}(\theta)$ move on a horizontal path instead of a vertical one. Moreover
the lattice of translation depends on the angle. In fact
$\Lambda_{\theta}:= \Z 4i +\Z(2l(\theta)+2i)$, where $l(\theta)=2-\cot \theta (1-\cot \theta)$.
While $L_{0}(\theta)$ is still centered at the origin, now with radius
$r_0(\theta)=l(\theta)\sin\theta=\frac{1}{\sin \theta} +\sin \theta - \cos \theta $
the other two lenses $L_{\pm}(\theta)$ as before of radius $\cos \theta$
are now centered at $c_{\pm}(\theta)=\pm (\cot \theta +2i)$, see Figure~\ref{german_animal_3}.
In particular the radii of the lenses $L_{\pm}(\theta)$ are bounded by $1<l(\theta)\leq 2$
and the radius of the lens $L_{0}(\theta)$ is bounded by $l(\theta)\leq 2$.
Because the generators of the lattice $\Lambda_{\theta}$ move each lens by
at least twice their diameter there are no pairwise intersections
possible among the lenses in one $\Lambda_{\theta}$ orbit. Moreover
the $\Z4i$ orbit of $L_{-}(\theta)$ lies on the left of the vertical
through the origin while the $\Z4i$ orbit of $L_{+}(\theta)$ lies on the right
of that line. As $0\leq \cot\theta\leq 1$ we have
\[\cot\theta+\cos\theta\leq 2\cot\theta\leq 2-\cot \theta (1-\cot \theta)=l(\theta).\]
Moreover, $r_0(\theta)=l(\theta)\sin\theta\leq l(\theta)$.
It follows that the $\Z4i$ orbits of $L_{\pm}(\theta)$ and $L_0(\theta)$ are contained in the strip $S = \{z\in\C;|\Re z| \leq l(\theta)\}$.
Since the $\Z(4+2i)$ translates of $S$ cover the whole plane, intersecting only in their boundary lines,
we conclude that no pair of Eaton lenses in the $\Lambda_{\theta}$ orbits of $L_{\pm}(\theta)$
intersect.

Restricted to the $\Z4i$ orbit the lens configuration have for all $\theta \in [\pi/4, \pi/2]$
reflection symmetries around the coordinate axes. More precisely
the $\Z4i$ orbit of each lens is invariant under the reflection at the horizontal
while the $\Z4i$ orbits of $L_{\pm}(\theta)$ are interchanged by reflection
at the vertical. % Give reason
Given these symmetries, all that remains to be seen is that $L_{0}(\theta)$
does not intersect with $L_{+}(\theta)$.
To do this we find a common tangent to $L_{0}(\theta)$ and $L_{+}(\theta)$
that separates them. Let us consider the tangent line $s(\theta)$ to $L_{0}(\theta)$
at the intersection point of its boundary with the half-line $t_0(\theta)$ in direction $\theta$
through the origin. The the direction of $s(\theta)$ is $\pi/2-\theta$.
The half-line $t_0(\theta)+i$ in direction $\theta$ through
the point $i$  intersects perpendicularly $s(\theta)$ and goes through the center of
$L_{+}(\theta)$. By elementary geometry, see also Figure \ref{german_animal_3}, the distance from $i$
to the center of $L_{+}(\theta)$ is $(\sin \theta)^{-1}$.
The leg of the right triangle with hypothenuse the segment from $0$ to $i$
lying on $t_0(\theta)$ has length $\sin \theta$.
So the intersection point of $s(\theta)$ with $t_{0}(\theta)+i$
must be at distance $r_0(\theta)-\sin \theta = \frac{1}{\sin \theta} - \cos \theta $
from $i$. But then it has distance $\cos \theta$ from the center of $L_{+}(\theta)$
and so the tangent $s(\theta)$ to $L_{0}(\theta)$ is also tangent to $L_{+}(\theta)$.

To show admissibility for one of the remaining angles, say $\theta \in [\pi/2, \pi]$, notice that
$L_{\pm}(\theta)$ are the lenses $L_{\pm}(\pi-\theta)$ reflected at the vertical through the origin.
 We also have $L_0(\theta)=L_0(\pi-\theta)$ and the lattice of translations has the same symmetry
 $\Lambda_{\theta}=\Lambda_{\pi -\theta}$.
 So the $\Lambda_{\theta}$ orbits of these (reflected) lenses match the distribution given in the introduction.
 Since for $\theta=\pi/2$ the lenses $L_{\pm}(\theta)$ are located on the vertical
 coordinate axis, this continuation of $\gamma_W$ is continuous at $\pi/2$.
 Moreover globally the lens distribution for $\theta \in [\pi/2, \pi]$ equals the one
 for $\pi-\theta$ reflected at the vertical coordinate axis. Since a reflection is an isometry,
 it preserves admissibility of lens distributions.
 Finally the Eaton lens configuration at $\theta = \pi$ matches that at $\theta=0$,
 since $\gamma_W(\pi)+2i=\gamma_W(0)$.
\end{proof}
In particular the proof of Proposition \ref{prop:adm} shows that
%\begin{corollary}\label{cor:adm}
the Wollmilchsau differential is a maximal Eaton differential.
%\end{corollary}
This, together with the fact that the Wollmilchsau differential
appears as a cyclic pillow case cover branched over exactly three points,
shows Theorem \ref{thm:ergodicity:intr1}.

%There are ergodic curves with different fundamental domains.

\section{Quadratic differentials on tori in the determinant locus}\label{sec:quad}

\subsection{Quadratic and Abelian differentials}
In this article quadratic differentials are the fundamental objects. They appear in various presentations,
analytical, polygonal and geometrical. All of those play important roles in different parts of our text.

Consider a Riemann surface $X$, i.e.\ a one dimensional complex manifold, not necessarily compact, and a quadratic differential $q$ on $X$ with poles of order at most one.  A quadratic differential is a tensor that can locally be written as $f(z)\,dz^2=f(z)\,dz \otimes dz$, where $f$ is a meromorphic function with poles of order at most one.
Away from the poles and zeros of $f$ one may use $q$ to define {\em natural coordinates}
on $X$
$$\zeta = \int^z_{z_0} \sqrt{f(z)}\, dz= \int^z_{z_0} \sqrt{q}.$$
If $\zeta_1$ and $\zeta_2$ are local coordinates, then
$d \zeta_1 = \sqrt{f(z)}\,dz = \pm  d \zeta_2$ in the intersection of the coordinate patches,
so $\zeta_1 = \pm  \zeta_2 + c$  for some $c \in \C$. That way the pair $(X, q)$ defines
a maximal atlas made of natural coordinates and is therefore called {\em half-translation surface}.
The maximal atlas is also called {\em half-translation structure}.
The coordinate changes for any two charts from a half-translation structure
are translations combined with half-turns (180 degree rotations)
and this motivates the name half-translation surface.
Similarly to a quadratic differential it is possible to consider an Abelian differential (holomorphic $1$-form) $\omega$ on $X$.
If $\Sigma\subset X$ denotes the set of zeros of $\omega$, as for quadratic differentials, away from $\Sigma$ Abelian differential defines {\em natural coordinates}
on $X$
\[\zeta = \int^z_{z_0} \omega.\]
If $\zeta_1$ and $\zeta_2$ are local coordinates and their coordinate patches intersect
then $\zeta_1 = \zeta_2 + c$  for some $c \in \C$. So the pair $(X, \omega)$ defines
a maximal atlas made of natural coordinates and is called \emph{translation surface}.
Here the maximal atlas is called \emph{translation structure}.

Objects on the plane that are invariant under {\em translations}
pull back via natural charts to $X$ and glue together to give global objects
on the translation surface $(X, \omega)$. Among those objects are
the euclidean metric, the differential $dz$, and constant vector fields
in any given direction. In fact, the pull back of the differential $dz$ recovers $\omega$
on the translation surface $(X, \omega)$. Similarly objects on the plane that are
invariant under {\em translations} and {\em half-turns} define global objects
on the half-translation surface $(X,q)$. Here objects of interest are again
the euclidean metric, the quadratic differential $dz^2$ (recovering $q$),
and any direction foliation by (non-oriented) parallel lines.
Since there is one line foliation on $\C$ for each angle
$\theta \in \R/\pi \Z$ that is tangent to $\pm \exp i \theta$, we denote its pullback to $X$ by
$\mathcal{F}_{\theta}(q)$, or $\mathcal{F}_{\theta}$ if there is no confusion about
the quadratic differential. For a translation surface, say $(X, \omega)$,
the constant unit vector field on $\C$ in direction $\theta \in \R/2\pi\Z $ defines
a directional unit vector field $V_\theta=V^{\omega}_\theta$ on $X\setminus\Sigma$.
Then the corresponding directional flow
$(\varphi^{\theta}_t)_{t\in\R}=(\varphi^{\omega,\theta}_t)_{t\in\R}$ (also known as \emph{translation flow}) on $X\setminus\Sigma$ preserves the area
measure $\mu_\omega$ given by $\mu_\omega(A)=|\int_A\frac{i}{2}\omega\wedge\overline{\omega}|$. If the surface $X$ is compact then the measure $\mu_\omega$ is finite.
We will use the notation $(\varphi^{v}_t)_{t\in\R}$  for the \emph{vertical flow} (corresponding to $\theta = \frac{\pi}{2}$) and
$(\varphi^{h}_t)_{t\in\R}$ for the \emph{horizontal flow} respectively ($\theta = 0$).

For every half-translation surface $(X,q)$ there exists a unique double cover $\pi_o: (\widehat{X}, \widehat{q}) \rightarrow (X,q)$, the \emph{orientation cover}, characterized by the property that it is branched precisely over all singular points with odd order. The pull-back $\widehat{q}= \pi^{\ast}_o q$
is the square $\widehat{q}=\omega^2$ of an abelian differential $\omega \in \Omega(X)$.
If $M=\widehat{X}$ then the translation surface $(M,\omega)$ is called also the orientation cover of the half-translation surface $(X,q)$. The pull-back $\widehat{\mathcal{F}}_{\theta}$ of any direction foliation
$\mathcal{F}_{\theta}$ is orientable. This foliation coincides with the foliations determined by
the directional flows  $(\varphi^{\theta}_t)_{t\in\R}$ and $(\varphi^{\theta+\pi}_t)_{t\in\R}$ on $(M,\omega)$.
Moreover, the ergodicity of the foliation $\mathcal{F}_{\theta}$ is equivalent to the ergodicity of the translation flow
$(\varphi^{\theta}_t)_{t\in\R}$.

\subsection*{Particular representations of half-translation structures.}
The quadratic differential $(dz)^2$ on $\C$ is invariant under translations and rotations
of $180$ degrees, that group generated by those isometries are in the group of
half-translations. Invariance of $(dz)^2$ under that group results in a variety of possible
constructions of quadratic differentials, or equivalently half-translation surfaces.

Most notably a (compact) polygon in $\C$ all of whose edges appear in parallel pairs,
 together with an prescribed identification of edge pairs by half-translations.
It is known, that any quadratic differential on a compact surface can be represented by such
a polygon.
A second way is to take suitable quotients of $\C$ under certain discrete groups of half-translations.
Here any torus $\C/\Lambda$ with a lattice $\Lambda$ of translations is an example.
Our way to built quadratic differentials in the plane $\C \cong \R^2$ and on
a torus is by successively adding (non-intersecting) slit-folds. Since the identifications
of the edges of a slit-fold are half-translations the given quadratic differential defines
a canonical new one on the surface with slit-fold. One important properties of
slit-folds is that they do not change the genus of the half-translation surface to which
they are added. Not only slit-folds have this property of defining quadratic differentials without
changing the genus. In fact more general types of ``folds'' are shown in
Appendix~\ref{app:foldsskel}. They are helpful in the construction of other ergodic curves.

\subsection{Cyclic covers of pillowcases}
In this section we classify those quadratic differentials on tori that arise as pullbacks of the \emph{pillowcase}
along a covering map (cyclic covers) which is unbranched over one point. Two of those examples are quotients
of the well known {\em Ornithorynque} and  {\em Eierlegende Wollmilchsau} under an involution.
%give rise to quadratic differentials on tori which are branched at three points over the {\em pillowcase} $\mathbb{CP}^1$.

Given a Riemann surface $X$ and a finite subset $\Sigma^* \subset X$
it is well known that the elements of $\xi\in H_1(X,\Sigma^*; G)$, $G$ an abelian group, define
a regular cover $\pi:X_\xi\to X$ over $X\backslash \Sigma^*$ branched over $\Sigma^* \subset X$ with
deck transformation group $G$. To describe this cover formally first denote by $\langle\,\cdot\,
,\,\cdot\, \rangle: H_1(X,\Sigma^*;G)\times H_1(X\setminus\Sigma^*;G)\to G$  the algebraic intersection form.
If $\sigma:[t_0, t_1]\to X$  is a closed curve in $X$ and $\sigma_\xi: [t_0, t_1]\to X_\xi$ is any
of its lifts to $X_\xi$  then
$\sigma_\xi(t_1) = \langle \xi, [\sigma] \rangle \cdot \sigma_\xi(t_0)$,
where $\cdot$ denotes the deck group action of $G$ on $X_\xi$.
Here we consider the case when the homology group is a direct sum of cyclic groups
of the kind $H_1(X,\Sigma^*; \Z/m\Z)$.

Let us look at the pillowcase $\mathcal{P}$ with underlying space $X= \mathbb{CP}^1$
and take $\Sigma \subset \mathbb{CP}^1$ to be the pillowcases four singular points.
We are looking for  pillowcase covers with at most three branch points.
That means such a cover is unbranched over at least one singular point of the pillowcase.
Then the result of Hubert and Griveaux \cite{Gr-Hu} implies that the cover
is in the {\em determinant locus}. We now construct those covers.

\subsection{Differentials in the determinant locus}
Take the pillowcase $X=\mathcal{P}$ with named singular points $p_1, p_2, p_3, p_4=p_F \in \mathcal{P}$
put in clockwise order starting from the upper left.
We assume the point $p_F$ is fixed under all automorphisms (and affine maps) of $\mathcal{P}$.
We further assume all branching of covers is restricted to the set $\Sigma^*=\{p_1, p_2, p_3\}$.

Let $\gamma_{12}, \gamma_{23} $ be generators in $H_1(\mathcal{P}, \{p_1,p_2, p_3\};\Z/d\Z)$ so that $\gamma_{12}$ is the class of the oriented horizontal path joining $p_1$ and $p_2$ and
 $\gamma_{23}$ is the class of the oriented vertical path joining $p_2$ and $p_3$.
 Let $\gamma_h, \gamma_v$ be generators in $H_1(\mathcal{P} \setminus \{p_1,p_2, p_3\};\Z/d\Z) $ such that $\gamma_h$ is
the class the horizontal (right oriented) simple loop and $\gamma_v$ is the class of the simple loop around $p_1$ with counterclockwise orientation.
Then
\[\langle \gamma_{12}, \gamma_v\rangle =\langle \gamma_{23}, \gamma_h\rangle = 1\ \text{ and }\langle \gamma_{12}, \gamma_h\rangle =\langle \gamma_{23}, \gamma_v\rangle = 0.\]

Let us consider any cyclic degree $d$ cover $\mathcal{P}_\xi$ of $\mathcal{P}$ branched over $\Sigma^*$ which is defined by a homology class $\xi=w_h \gamma_{12}+w_v \gamma_{23} \in H_1(X, \Sigma^*;\Z/d\Z)$.
Here \[w_h =\langle  \xi,\gamma_v\rangle\in \Z/d\Z\ \text{ and }\ w_v =\langle \xi,\gamma_h\rangle\in \Z/d\Z\]
are called weights of the cover $\mathcal{P}_\xi\to\mathcal{P}$.
Therefore the cover is determined by the triple $(d, w_h, w_v) \in \N \times \Z/d\Z \times \Z/d\Z$
and we will denote it by $X_d(w_h, w_v)\rightarrow \mathcal{P}$. The cover $X_d(w_h, w_v)$ is connected iff $\gcd(d,w_h,w_v)=1$.
The cover defined by those data has a straightforward geometric realization. Namely,
cut the pillowcase along the three line segments joining: $p_1$ with $p_2$, $p_2$ with $p_3$ and $p_3$ with $p_F$.
The resulting surface is isometric to a rectangle of width $2$ and height $1$ in the complex plane.
Let us denote this polygonal presentation of $X$ with cuts by $X^{c}$ and take $d$ labeled copies
$X^{c} \times \{1,\ldots,d\}=X^{c}_1 \sqcup \cdots \sqcup X^{c}_d$. Now identify the
vertical right edge of $X^{c}_i$ with the vertical left edge $X^{c}_{i + w_v \bmod d}$ by a translation. Then identify
the right half of the upper horizontal edge of $X^{c}_i$ with the left half of the  upper horizontal edge of $X^{c}_{i + w_h \bmod d}$ using a half turn
and identify
the right half of the lower horizontal edge of $X^{c}_i$ with the left half of the lower horizontal edge of $X^{c}_{i}$ using a half turn.
This determines $X_d(w_h,w_v)$ because of the covers cyclic nature.
By eventually renaming the decks we may assume that $w_v=\gcd(w_v,d)$ divides $d$.
Indeed, if $A:\Z/d\Z\to\Z/d\Z$ is a group automorphism then using $A$ to rename the decks we obtain $X_d(w_h,w_v)\cong X_d(A^{-1}w_h,A^{-1}w_v)$.
Let $w_v=\gcd(w_v,d)l$ and let $A$ be the multiplication by $l$ on $\Z/d$. Since $\gcd(l,d)=1$, $A$ is an automorphism for which $A(\gcd(w_v,d))=w_v$.
Then  $X_d(w_h,w_v)\cong X_d(A^{-1}w_h, \gcd(w_v,d))$. See \cite{EKZ1} and \cite{Fo-Ma-Zo11} for a more background and  applications of cyclic covers.

We now determine those cyclic covers that are torus differentials, i.e.\ have genus $1$.
To calculate the genus of $X_d(w_h, w_v)$ we note, that the covering has
$\gcd(w_h,d)$ preimages over $p_1$,
$\gcd(|w_h-w_v|,d)$ preimages over $p_2$ and $w_v=\gcd(w_v,d)$
preimages over $p_3$ because it is cyclic.
It follows that the respective branching orders are
$o_1= d/ \gcd(w_h,d)$ at $p_1$, $o_2=d/ \gcd(|w_h-w_v|,d)$ at $p_2$ and
$o_3=d/ \gcd(w_v,d)= d/ w_v$ at $p_3$.
That means we have an angle excess of $(o_i-2)\pi$ around any preimage of $p_i$ for $i=1,2,3$.
\begin{proposition} \label{prop:gen}
The genus $g_{d,w_h,w_v}$ of $X_d(w_h,w_v)$ is given by
\begin{eqnarray*} g_{d,w_h,w_v}-1&=&(d -\gcd(w_h,d)-\gcd(w_v,d)- \gcd(|w_h-w_v|,d))/2\\
&=&(d -w_v -\gcd(w_h,d)- \gcd(|w_h-w_v|,d))/2.
\end{eqnarray*}
\end{proposition}
\begin{proof}
Write down the standard formula expressing the Euler characteristic of quadratic
differentials in terms of total angle deficit for singular points and
total angle excess for cone points:
\begin{align*}
2 \chi(X_d&(w_h, w_v))=d+ \gcd(w_h,d)(2 - d/\gcd(w_h,d))\\
&+\gcd(w_v,d)(2 - d/\gcd(w_v,d)) + \gcd(|w_h-w_v|,d)(2 - d/\gcd(|w_h-w_v|,d))\\
=&2( -d + \gcd(w_h,d) + \gcd(w_v,d)+\gcd(|w_h-w_v|,d)).
\end{align*}
The result follows since $ \chi(X_d(w_1,w_2))=2(1-g_{d,w_h,w_v})$.
\end{proof}
By definition the degree of the pillowcase cover $\pi_d(w_h,w_v)\!:\! X_d(w_h,w_v) \!\rightarrow \!\mathcal{P}$ is
$d$.
\begin{proposition} \label{torus_cover_class}
If $X_d(w_h,w_v)$ has genus $1$, then $d \in \{3,4,6\}$.
\end{proposition}
\begin{proof}
A torus has vanishing Euler characteristic, thus from Proposition~\ref{prop:gen}
we directly derive the condition
\[ d = \gcd(w_h,d)+ \gcd(w_v,d)+\gcd(|w_h-w_v|,d).\]
Dividing by $d$, we see that a torus presents a positive integer solution of the problem
\[ 1= \frac{1}{a}+\frac{1}{b}+\frac{1}{c}, \]
where $a,b,c$ represent the natural numbers $d/\gcd(w_h,d)$, $d/\gcd(w_v,d)$, $d/\gcd(|w_h-w_v|,d)$.
Without restriction of generality we may assume that any solution fulfills $c\geq b\geq a >0$. It follows that $2\leq a\leq 3$.

If $a=2$ then $1/b+1/c=1/2$ which gives $b\leq 4$. Therefore we obtain  two possibilities $(b,c)=(3,6)$ or $(4,4)$.

If $a=3$ then $1/b+1/c=2/3$ with $c\geq b\geq 3$. It leads to $(b,c)=(3,3)$. It follows that we get only  $(3,3,3)$, $(2,4,4)$, $(2,3,6)$ as solutions.
Since
\[\gcd\big(\gcd(w_h,d),\gcd(w_v,d),\gcd(|w_h-w_v|,d)\big)=1,\]
we obtain $\operatorname{lcm}(a,b,c)=d$. It follows that $d=3, 4, 6$ respectively.
\end{proof}
\subsection{Branched pillow case covers that are torus differentials}
In spite of Proposition  \ref{torus_cover_class} all we need to do
to exhaust the list of possible of torus covers is to go through a short list of possible cases.
Because $p_F$ is assumed to be fixed the pillowcase has no automorphisms.
For $d=3,4$ and $6$ we need to find the weights $1 \leq w_v, w_h <d$
with $\gcd(w_h,w_v,d)=1$ satisfying the condition
\[ w_v|d\ \text{ and }\ d= \gcd(w_h,d)+ w_v + \gcd(|w_h-w_v|,d). \]
The weights cannot be
$0$ or $d$, because the cover must be branched over all three points $p_1,p_2$ and
$p_3$ to give a surface of genus larger than zero, the genus of the pillowcase.
Thus without loss of generality we can pick the weights $w_h, w_v$ from $\{1,\ldots,d-1\}$.
For $d=6$ we obtain the following weight pairs fulfilling the conditions:
\[
 (w_h,w_v)\in \{(1,3), (3,1), (3,2), (2,3), (4,1), (4,3), (5,2), (5,3)\}
\]
The weights tell us the number of deck changes that occur when we go over either homology
class. By renaming the decks so that deck $k$ becomes deck $d-k$ we obtain the cover
$X_d(d-w_h,d-w_v)$ from $X_d(w_h, w_v)$. Thus those are isomorphic, in particular
for $d=6$ we have $X_6(1,3) \cong X_6(5,3)$ and $X_6(2,3) \cong X_6(4,3)$.
For $d=3$ and $d=4$ the same line of arguments applies and leads to the following list of covers: \\
\begin{center}\label{zoo}
\begin{tabular}{|c |c | c | c | c | c | l | }
\hline
\multicolumn{7}{ |c| }{Torus differentials of degree $d=3,4$ and $6$} \\
\hline
Degree $d$ & $w_h$ & $w_v$ & \# $\pi^{-1}(p_1)$ &  \# $\pi^{-1}(p_2)$ &  \# $\pi^{-1}(p_3)$ & Surface\\
\hline
 $3$ & $2$& $1$ & 1 & 1 & 1 & $X_3(2,1)$\\
\hline
\multirow{3}{*}{$4$}
 & $2$ & $1$ & $2$ & 1 & 1 & $X_4(2,1)$\\
 & $3$ & $1$  & $1$ & 2 & 1 & $X_4(3,1)$\\
 & $3$ & $2$  & $1$ & 1 & 2 & $X_4(3,2)$\\
 \hline
\multirow{6}{*}{$6$}
 & $3$ & 1 & 3 & 2 & 1 & $X_6(3,1)$\\
 & $3$ & 2 & 3 & 1 &  2 & $X_6(3,2)$\\
 & $4$ & 1 & 2 & 3 & 1 & $X_6(4,1)$\\
 & $4$ & 3 & 2 & 1 &  3 & $X_6(4,3)$\\
 & $5$ & 2 & 1 & 3 & 2 & $X_6(5,2)$\\
 & $5$ & 3 & 1 & 2 & 3 & $X_6(5,3)$\\

 \hline

\end{tabular}
 \end{center}
 \vspace*{2mm}
The group $\SL_2(\R)$ acts real linearly on the plane and defines a map on half-translation surfaces
by post composition with local coordinates.
Alternatively one may take a polygon representation of the surface and apply a matrix
$A \in \SL_2(\R)$, viewed as linear map of $\R^2$, to it.
The edges of the polygon are then identified exactly as before
the deformation. That defines an action of $ \SL_2(\R)$ on surfaces with quadratic differential.
We denote by $A \cdot X$ the deformation of $X$ by $A \in \SL_2(\R)$.

Let $X_\xi\to X$ be a  branched $G$-cover over $\Sigma^*\subset X$ and determined by $\xi\in H_1(X,\Sigma^*;G)$.
Then the deformation $A\cdot X_\xi$ is a branched cover determined by $A_*\xi\in H_1(A\cdot X,\Sigma^*;G)$.

The pillowcase is stabilized by all elements of $\SL_2(\Z)$, as one can easily check
on the two (parabolic) generators $P_h:=\left [\begin{smallmatrix}1& 1\\ 0 &1  \end{smallmatrix}\right] \in \SL_2(\Z)$
and $P_v:=\left [\begin{smallmatrix}1& 0\\ 1 &1  \end{smallmatrix}\right] \in \SL_2(\Z)$.
Stabilized means the original pillowcase can be obtained from the deformed pillowcase by successively cutting off polygons, translating and if needed rotating them to another boundary in tune with the edge
identification rules of the pillowcase.

Let us consider any cover $X_d(w_h,w_v)=\mathcal{P}_\xi$ (with $\xi=w_h \gamma_{12}+w_v \gamma_{23}$) and $A\in SL_2(\Z)$. Since $A\cdot\mathcal{P}=\mathcal{P}$, we have
\[A\cdot X_d(w_h,w_v)=A\cdot\big(\mathcal{P}_\xi\big)=\big(A\cdot\mathcal{P}\big)_{A_*\xi}=\mathcal{P}_{A_*\xi}=X_d(\langle  A_*\xi,\gamma_v\rangle,\langle  A_*\xi,\gamma_h\rangle)\]
and $\langle  A_*\xi,\gamma_h\rangle=\langle  \xi,A^{-1}_*\gamma_h\rangle$, $\langle  A_*\xi,\gamma_v\rangle=\langle  \xi,A^{-1}_*\gamma_v\rangle$.
Moreover for the parabolic generators $P_h$ and $P_v$ we have
\[(P_h^{-1})_*\gamma_h=\gamma_h,\quad (P_h^{-1})_*\gamma_v=\gamma_h-\gamma_v,\quad (P_v^{-1})_*\gamma_v=\gamma_v,\quad (P_v^{-1})_*\gamma_h=\gamma_v-\gamma_h,\]
and hence
\begin{align*}
&\langle  \xi,(P_h^{-1})_*\gamma_v\rangle=\langle w_h \gamma_{12}+w_v \gamma_{23},\gamma_h-\gamma_v\rangle=w_v-w_h,\\
&\langle  \xi,(P_h^{-1})_*\gamma_h\rangle=\langle w_h \gamma_{12}+w_v \gamma_{23},\gamma_h\rangle=w_v,\\
&\langle  \xi,(P_v^{-1})_*\gamma_v\rangle=\langle w_h \gamma_{12}+w_v \gamma_{23},\gamma_v\rangle=w_h,\\
&\langle  \xi,(P_v^{-1})_*\gamma_h\rangle=\langle w_h \gamma_{12}+w_v \gamma_{23},\gamma_v-\gamma_h\rangle=w_h-w_v.
\end{align*}
This yields the action of parabolic matrices on degree $d$ pillowcase covers:
\[
P_h \cdot X_d(w_h,w_v) = X_d(w_v -w_h,w_v)\ \text{ and }\ P_v \cdot X_d(w_h,w_v) = X_d(w_h,w_h-w_v).\]
Since the group of maps generated by two involutions $(x,y)\mapsto(x,y-x)$ and $(x,y)\mapsto(y-x,y)$ has exactly $6$ elements, so we obtain the following:
\begin{proposition}
The $\SL_2(\Z)$ orbit of a pillowcase cover is given by
%to the group elements $\{\id, P_h,P_v, P_hP_v, P_vP_h, P_hP_vP_h\}$.
\begin{align*} \SL_2(\Z) \cdot X_d(w_h,w_v)=  \big\{ X_d(w_h,w_v), X_d(w_h,w_h-w_v), X_d(w_v -w_h,w_v)& ,\\
 X_d(-w_v,w_h- w_v), X_d(w_v-w_h, -w_h), X_d(-w_v,-w_h)& \big\}.
\end{align*}
\end{proposition}
Note, that for low degree this orbit is even smaller: The orbits of degree
three and four covers contain less than six tori. As can be easily
seen from the proposition, compare the table of surfaces, that the relevant torus differentials
of fixed degree lie on one $\SL_2(\Z)$ orbit.

\subsection*{Orientation covers of some pillow case covers.}
We consider the orientation covers of $X_d(2,1)$, for $d=3,4$ and $X_d(3,1)$ for $d=6$ drawn on Figure~\ref{three_animal}.
Recall that the orientation cover $(\widehat{X}, \omega^2) \rightarrow (X, q)$ of a quadratic differential $(X,q)$ is uniquely characterized as the degree two cover, branched precisely over
the cone points having an odd total angle (in multiples of $\pi$).
There is a sheet exchanging involution $\rho$ on $\widehat{X}$ that has
the preimages of the odd cone points as fixed-points. The involution is locally a rotation by $\pi$, eventually followed by a translation.
\begin{figure}[!htb]
 \centering
  \includegraphics[scale=0.40]{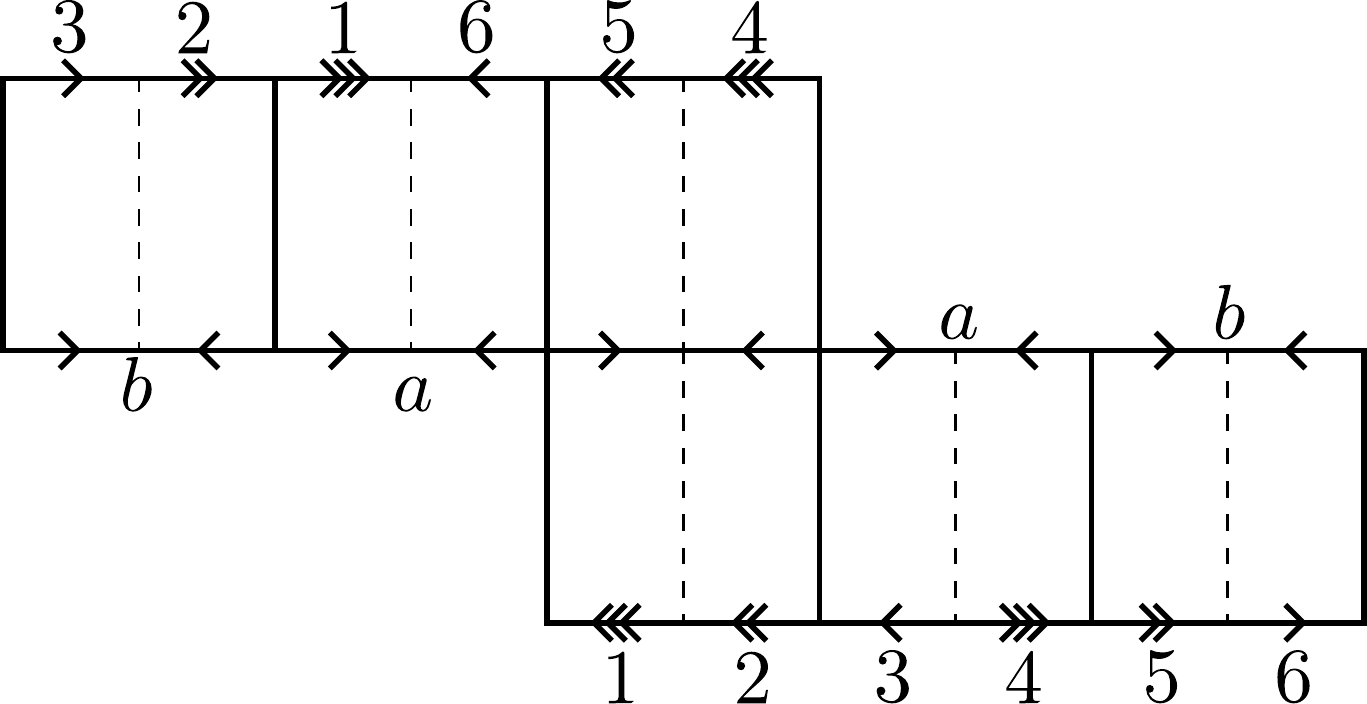}
 \caption{The Ornithorynque as orientation cover of $X_3(2,1)$}
 \label{Ornithorynque}
\end{figure}
\begin{figure}[!htb]
 \centering
 \includegraphics[scale=0.40]{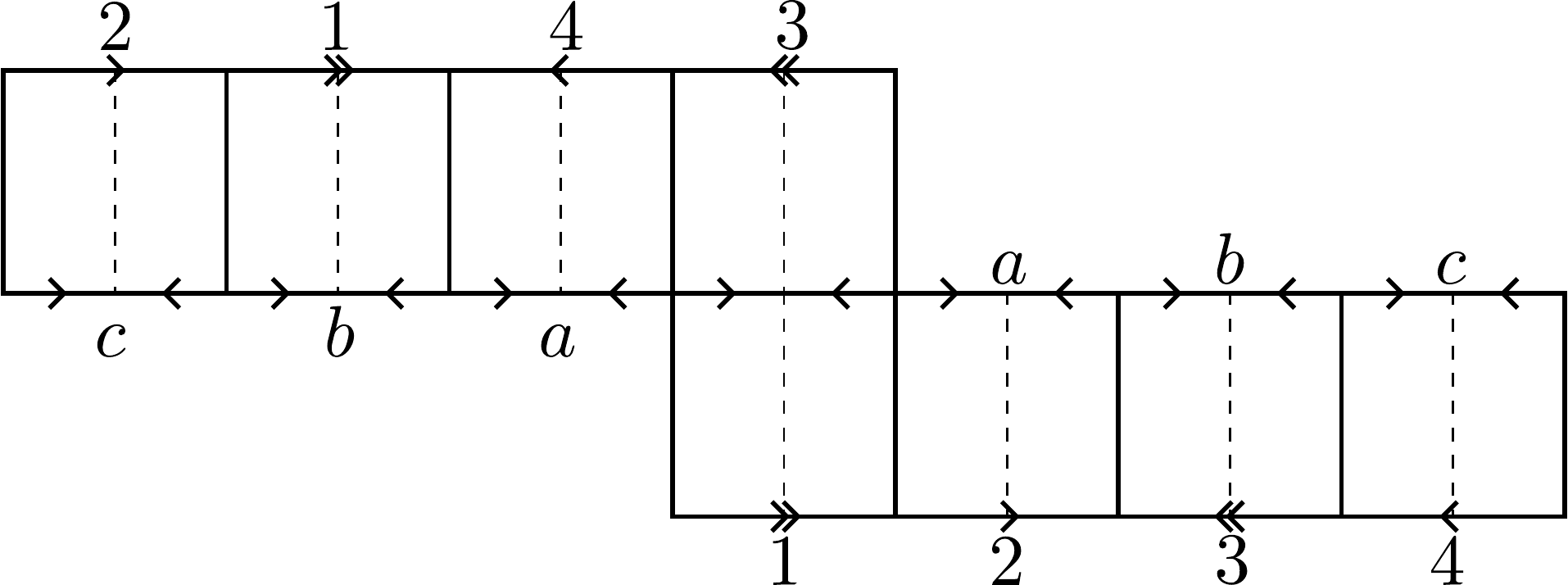}
 \caption{The Eierlegende Wollmilchsau as orientation cover of $X_4(2,1)$}
 \label{Wollmilchsau}
\end{figure}
Using this one may construct orientation covers given a polygonal representation.
One considers two copies of the polygon and whenever two edges were identified by a rotation
on the original polygon, one identifies any of those two edges as before but now to the
corresponding edge of the other copy. Turning any one copy by $180$ degrees
the new identifications become translations and we have a translation surface.
\begin{figure}[!htb]
 \centering
  \includegraphics[scale=0.30]{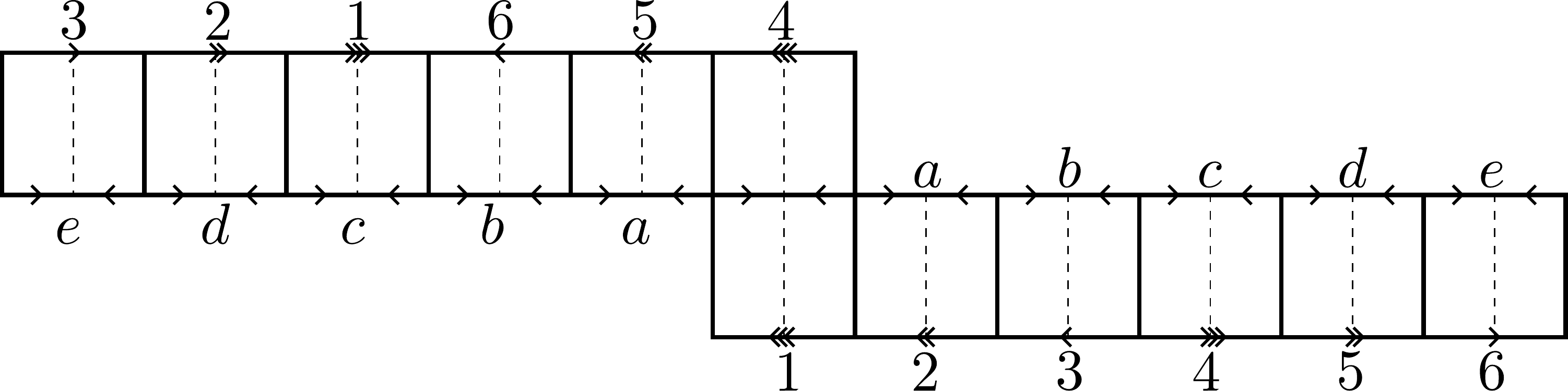}
 \caption{The orientation cover of $X_6(3,1)$}
 \label{polish_animal}
\end{figure}
For the surfaces at hand this procedure is reflected in the following Figures
\ref{Ornithorynque}, \ref{Wollmilchsau} and \ref{polish_animal}.
The first two are splendid specimens in the zoo of square tiled surfaces.
If the name did not immidiately give it away, a look at the figures should explain
the idea of a square tiled surface.
In fact, $\widehat{X}_3(2,1)  \cong \widehat{X}_2 $ is the
{\em Ornithorynque} and $\widehat{X}_4(2,1)  \cong \widehat{X}_3 $
is known as the {\em Eierlegende Wollmilchsau}. Both names
reflect the surfaces multiple rather exceptional properties,
each of them has vanishing Lyapunov exponents.
To our best knowledge the orientation cover of $X_6(3,1)$ is not
such a well studied square tiled surface and we are not able to provide
a direct reason to motivate such research.

\section{Ergodicity of translation flows and measured foliations on infinite covers}
In this section we prove  a useful
criterion on ergodicity for translation flows on $\Z^d$-covers (see Theorem~\ref{thm:ergV}). The key Theorem~\ref{thm:mainquadr:intr}
follows directly from this criterion.

For relevant background material concerning IETs and their relations to translation surfaces, we refer the reader to \cite{Fo-Ma}, \cite{Ma}, \cite{ViB}, \cite{YoLN} and \cite{ZoFlat}.

\subsection{$\Z^d$ covers}
Let $\widetilde{X}$ be a $\Z^d$-cover of a compact connected surface $X$
and let $p:\widetilde{X}\to X$ be the covering map, i.e.\ there
exists a properly discontinuous $\Z^d$-action on $\widetilde{X}$ such that
$\widetilde{X}/\Z^d$ is homeomorphic to $X$. Then $p:\widetilde{X}\to
X$ is the composition of the natural projection
$\widetilde{X}\to\widetilde{X}/\Z^d$ and the homeomorphism.
Denote by $\langle\,\cdot\,
,\,\cdot\, \rangle: H_1(X,\Z)\times H_1(X,\Z)\to \Z$  the algebraic intersection form.
Then any $\Z^d$-cover $\widetilde{X}$ is determined by a
$d$-tuple $\gamma=(\gamma_1,\ldots,\gamma_d)\in H_1(X,\Z)^d$ so
that if $\sigma:[t_0, t_1]\to X$  is a closed curve in $X$ and $\widetilde{\sigma}: [t_0, t_1]\to \widetilde{X}$ is any its lift to $\widetilde{X}$
then
\[ \widetilde{\sigma}(t_1) = \langle \gamma, [\sigma] \rangle \cdot \widetilde{\sigma}(t_0),\]
where
\[\langle \gamma, [\sigma] \rangle=(\langle \gamma_1, [\sigma] \rangle,\ldots,\langle \gamma_d, [\sigma] \rangle ) \in
\mathbb{Z}^d\qquad([\sigma]\in H_1(X,\Z))\]
and $\cdot$ denotes
the action of $\Z^d$ on $\widetilde{X}$. The
$\Z^d$-cover corresponding to $\gamma$ will be denoted by
$\widetilde{X}_\gamma$.

\begin{remark}
Note that the surface $\widetilde{X}_\gamma$ is connected if and only if
the group homomorphism $H_1(X,\Z)\ni\xi\mapsto \langle\gamma,\xi\rangle\in \Z^d$
is surjective.
\end{remark}

If $q$ is a
quadratic differential on $X$ then the pull-back $p^*(q)$ of $q$ by $p$ is also a
quadratic differential on $\widetilde{X}_\gamma$ and will be denoted by
$\widetilde{q}_\gamma$. For any $\theta\in\R/\pi\Z$ we denote by $\widetilde{\mathcal{F}}_\theta=\widetilde{\mathcal{F}}_\theta^\gamma$  the
corresponding measurable foliation on  $(\widetilde{X}_\gamma,\widetilde{q}_\gamma)$.

If $(M,\omega)$ is a compact translation surface and $\gamma\in H_1(M,\Z)^d$ is a $d$-tuple then
the translation flow on the $\Z^d$-cover $(\widetilde{M}_\gamma,\widetilde{\omega}_\gamma)$ in the direction $\theta$ is denoted by
$(\widetilde{\varphi}^\theta_t)_{t\in\R}$.

Let $(X,q)$ be a connected half-translation surface and denote by $(M,\omega)$ its orientation cover which is a translation surface.
Then there exist a branched covering map $\pi:M\to X$ such that $\pi^*(q)=\omega^2$ and an
idempotent $\sigma:X\to X$ such that $\pi\circ \sigma=\pi$ and $\sigma^*(\omega)=-\omega$.

The space $H_1(M,\R)$ has an orthogonal (symplectic) splitting into spaces
$H^+_1(M,\R)$ and $H^-_1(M,\R)$ of $\sigma_*$-invariant
and $\sigma_*$-anti-invariant homology classes, respectively. Moreover, the subspace $H^+_1(M,\R)$
is canonically isomorphic to $H_1(X,\R)$ via the map
$\pi_*:H^+_1(M,\R)\to H_1(X,\R)$, so we identify
both spaces.

Recall that  the measured foliation $\mathcal{F}_\theta$ of $X$ is ergodic for some $\theta \in \R/2\pi\Z$ if and only if
the translation flow $(\varphi_t^\theta)_{t\in \R}$ on $M$ is ergodic with respect to the measure $\mu_{\omega}$ (possibly infinite).

\begin{remark}\label{rem:erg}
Let $\gamma\in (H_1(X,\Z))^d$  be a $d$-tuple such that the $\Z^d$-cover $\widetilde{X}_{\gamma}$ is connected.
Since $H^+_1(M,\Z)$ and  $H_1(X,\Z)$ are identified, we can treat $\gamma$ as a $d$-tuple in
$(H^+_1(M,\Z))^d$.
Let us consider  the corresponding $\Z^d$-cover $\widetilde{M}_\gamma$. Then the maps $\pi:M\to X$ and $\sigma:M\to M$ can be lifted to a branched covering map
$\widetilde{\pi}:\widetilde{M}_\gamma\to \widetilde{X}_{\gamma}$ and an involution $\widetilde{\sigma}:\widetilde{M}_\gamma\to\widetilde{M}_\gamma$ so that
$\widetilde{\pi}\circ\widetilde{\sigma}=\widetilde{\pi}$. Then $\widetilde{\pi}$ establishes
an orientation cover $(\widetilde{M}_{\gamma},\widetilde{\omega}_{\gamma})$ of the half-translation surface $(\widetilde{X}_{\gamma},\widetilde{q}_{\gamma})$.
Therefore, for every $\theta \in \R/2\pi\Z$ the ergodicity of the measured foliation $\widetilde{\mathcal{F}}_\theta$ of $(\widetilde{X}_{\gamma},\widetilde{q}_{\gamma})$
is equivalent to the ergodicity of the translation flow $(\widetilde{\varphi}_t^\theta)_{t\in \R}$ on $(\widetilde{M}_{\gamma},\widetilde{\omega}_{\gamma})$. Note, that the measure
$\mu_{\widetilde{\omega}_{\gamma}}$ is an infinite Radon measure.
\end{remark}

\subsection{The Teichm\"uller flow and the Kontsevich-Zorich cocycle}\label{Teich:sec}
Given  a connected compact oriented surface $M$ of genus $g$,
% and a finite set
%$\Sigma\subset M$
denote by $\operatorname{Diff}^+(M)$ the
group of orientation-preserving homeomorphisms of $M$. Denote by $\operatorname{Diff}_0^+(M)$ the
subgroup of elements $\operatorname{Diff}^+(M)$ which are
isotopic to the identity. Let us denote by
$\Gamma(M):=\operatorname{Diff}^+(M)/\operatorname{Diff}_0^+(M)$
the {\em mapping-class} group. We will denote by $\mathcal{T}(M)$
(respectively $\mathcal{T}_1(M)$ ) the {\em Teichm\"uller space of
Abelian differentials } (respectively of unit area Abelian
differentials), that is the space of orbits of the natural action
of $\operatorname{Diff}_0^+(M)$ on the space of all
Abelian differentials on $M$ (respectively, the ones with total
area $\mu_\omega(M)=1$). We
will denote by $\mathcal{M}(M)$ ($\mathcal{M}_1(M)$) the {\em
moduli space of (unit area) Abelian differentials}, that is the
space of orbits of the natural action of
$\operatorname{Diff}^+(M)$ on the space of (unit area)
Abelian differentials on $M$. Thus
$\mathcal{M}(M)=\mathcal{T}(M)/\Gamma(M)$ and
$\mathcal{M}_1(M)=\mathcal{T}_1(M)/\Gamma(M)$.

The moduli space $\mathcal{M}(M)$ is stratified according to the number and multiplicity of the holomorphic one-forms zeros and the $SL(2,\R)$-action respects this stratification. Define  the  stratum $\mathcal{M}(\kappa_1,\ldots,\kappa_s)$  as  the  collection of
translations surfaces $(M,\omega)$ such $\omega$ has $s$ zeros and the multiplicity of the zeros of $\omega$ is given by $(\kappa_1,\ldots,\kappa_s)$.
Then $\kappa_1+\ldots+\kappa_s=2g-2$.

Denote by $\mathcal{Q}(X)$ the moduli space of half-translation surfaces which is also naturally stratified by the number and the types of singularities. We denote by $\mathcal{Q}(\kappa_1,\ldots,\kappa_s)$ the stratum of quadratic differentials $(X,q)$ which are not the squares of Abelian
differentials, and which have $s$ singularities and their orders are $(\kappa_1,\ldots,\kappa_s)$, where $\kappa_i\geq -1$.
Then $\kappa_1+\ldots+\kappa_s=4g_X-4$, where $g_X$ is the genus of $X$.

The group $SL(2,\R)$ acts naturally on  $\mathcal{T}_1(M)$ and $\mathcal{M}_1(M)$ as follows.  Given a translation structure $\omega$, consider the charts
given by local primitives of the holomorphic $1$-form. The new charts defined by postcomposition of these charts with an element of $SL(2,\R)$ yield a new
complex structure and a new differential that is Abelian with respect to this new complex structure, thus a new translation structure.
We denote by $g\cdot \omega$ the translation structure  on $M$
obtained acting by $g \in SL(2,\R)$ on a translation structure
$\omega$ on $M$.

The {\em Teichm\"uller flow} $(g_t)_{t\in\R}$ is the restriction
of this action to the diagonal subgroup
$(\operatorname{diag}(e^t,e^{-t}))_{t\in\R}$ of $SL(2,\R)$ on
$\mathcal{T}_1(M)$ and $\mathcal{M}_1(M)$. We will deal also with
the rotations $(r_{\theta})_{\theta\in \R/2\pi\Z }$ that acts on
$\mathcal{T}_1(M)$ and $\mathcal{M}_1(M)$ by
$r_\theta\omega=e^{i\theta}\omega$.

\begin{theorem}[see \cite{Ma0}]\label{thm:masur}
For every Abelian differential $\omega$ on a compact connected surface $M$ for almost all
directions $\theta\in \R/2\pi\Z $ the vertical and horizontal flows on $(M,r_\theta\omega)$ are uniquely ergodic.
\end{theorem}
Every  $\theta\in \R/2\pi\Z $ for which the assertion of the
theorem holds is called \emph{Masur generic}.

 The {\em Kontsevich-Zorich (KZ) cocycle}
$(G^{KZ}_t)_{t\in\R}$ is the quotient of the trivial cocycle
\[g_t\times\operatorname{Id}:\mathcal{T}_1(M)\times H_1(M,\R)\to\mathcal{T}_1(M)\times H_1(M,\R)\]
by the action of the mapping-class group
$\Gamma(M)$. The mapping class group acts on
the fiber $H_1(M,\R)$ by induced maps. The cocycle
$(G^{KZ}_t)_{t\in\R}$ acts on the homology vector bundle
\[\mathcal{H}_1(M,\R)=(\mathcal{T}_1(M)\times H_1(M,\R))/\Gamma(M)\]
 over the Teichm\"uller flow
$(g_t)_{t \in \R}$ on the moduli space $\mathcal{M}_1(M)$.

Clearly the fibers of the  bundle $\mathcal{H}_1(M,\R)$ can be
identified with $H_1(M,\R)$.  The space $H_1(M,\R)$ is endowed
with the symplectic form given by the algebraic intersection
number. This symplectic structure  is preserved by the action of
the mapping-class group and hence is invariant under the action of
$SL(2,\R)$.

The standard definition of KZ-cocycle uses the cohomological
bundle. The identification of the homological and cohomological
bundle and the corresponding KZ-cocycles is established by the
Poincar\'e duality $\mathcal{P}:H_1(M,\R)\to H^1(M,\R)$. This
correspondence allows us to define the so called Hodge norm (see
\cite{For-dev} for cohomological bundle) on each fiber of the
bundle $\mathcal{H}_1(M,\R)$. The norm on the fiber $H_1(M,\R)$
over $\omega\in\mathcal{M}_1(M)$ will be denoted by
$\|\,\cdot\,\|_\omega$.

Let  $\omega \in\mathcal{M}_1(M)$ and denote by
$\mathcal{M}=\overline{SL(2,\R)\omega }$ the closure of the $SL(2,\R)$-orbit of $\omega $ in $\mathcal{M}_1(M)$. The celebrated result of Eskin, Mirzakhani and Mohammadi, proved in
\cite{EMM} and \cite{EM}, says that $\mathcal{M}\subset
\mathcal{M}_1(M)$ is an affine $SL(2,\R)$-invariant submanifold.
Denote by $\nu_{\mathcal{M}}$ the corresponding affine
$SL(2,\R)$-invariant probability measure supported on
$\mathcal{M}$. The above results say in addition, that the measure $\nu_{\mathcal{M}}$ is ergodic under
the action of the Teichm\"uller flow. It follows, that $\nu_{\mathcal{M}}$-almost every element of $\mathcal{M}$
is Birkhoff generic, i.e.\ the pointwise ergodic theorem holds for the Teichm\"uller flow and every continuous integrable function on $\mathcal{M}$.
The following recent result is more refined and yields Birkhoff generic elements among $r_\theta\omega $  for $\theta\in \R/2\pi\Z $.

\begin{theorem}[see \cite{Es-Ch}]\label{thm:esch}
For almost all
$\theta\in \R/2\pi\Z $ we have
\begin{equation*}
\lim_{T\to\infty}\frac{1}{T}\int_0^T\phi(g_tr_\theta\omega )\,dt=\int_{\mathcal{M}}\phi\,d\nu_{\mathcal{M}}\ \text{ for every }\ \phi\in C_c(\mathcal{M}_1(M)).
\end{equation*}
\end{theorem}

All directions $\theta\in \R/2\pi\Z $ for which the assertion of the
theorem holds are called \emph{Birkhoff generic}.

Let $\mathcal{V}\to \mathcal{M}$ be an $SL(2,\R)$-invariant  subbundle of $\mathcal{H}_1(M,\R)$
which is defined and continuous over $\mathcal{M}$. For every $\omega\in \mathcal{M}$ we denote by $\mathcal{V}_\omega$ its fiber over $\omega$.

Let us consider the KZ-cocycle $(G_t^{\mathcal{V}})_{t\in\R}$
restricted to $\mathcal{V}$. By Oseledets' theorem, there exists Lyapunov
exponents of $(G_t^{\mathcal{V}})_{t\in\R}$ with respect to the
measure $\nu_{\mathcal{M}}$. If additionally, the subbundle $\mathcal{V}$ is symplectic, its Lyapunov exponents
with respect to the measure $\nu_{\mathcal{M}}$ are:
\[\lambda^{\mathcal{V}}_1\geq\lambda^{\mathcal{V}}_2\geq\ldots\geq\lambda^{\mathcal{V}}_d\geq-
\lambda^{\mathcal{V}}_d\geq\ldots\geq-\lambda^{\mathcal{V}}_2\geq-\lambda^{\mathcal{V}}_1.\]

\begin{theorem}[see \cite{Es-Ch}]
Let
$\lambda^{\mathcal{V}}_1=\overline{\lambda}_1>\overline{\lambda}_2>
\ldots>\overline{\lambda}_{s-1}>\overline{\lambda}_s=-\lambda^{\mathcal{V}}_1$
be distinct Lyapunov exponents of
$(G_t^{\mathcal{V}})_{t\in\R}$ with respect to
$\nu_{\mathcal{M}}$. Then for a.e.\ $\theta\in \R/2\pi\Z $ there exists a
direct splitting of the fibre
$\mathcal{V}_{r_\theta\omega}=\bigoplus_{i=1}^s\mathcal{U}^i_{r_\theta\omega}$ such that for
every $\xi\in \mathcal{U}^i_{r_\theta\omega}$ we have
\begin{equation}
\lim_{t\to\infty}\frac{1}{t}\log\|\xi\|_{g_tr_\theta\omega}=\overline{\lambda}_i.
\end{equation}
\end{theorem}
Each $\theta\in \R/2\pi\Z $ for which the assertion of the
theorem holds is called \emph{Oseledets generic}.
Then $\mathcal{V}_{r_\theta\omega}$ has a  direct splitting
\[\mathcal{V}_{r_\theta\omega}=E_{r_\theta\omega}^+\oplus E_{r_\theta\omega}^0\oplus E_{r_\theta\omega}^-\]
into unstable, central and stable subspaces
\begin{align*}
E_{r_\theta\omega}^+&=\Big\{\xi\in \mathcal{V}_{r_\theta\omega}:
\lim_{t\to+\infty}\frac{1}{t}\log\|\xi\|_{g_{-t}r_\theta\omega}<0\Big\},
\label{stabledef}\\
E_{r_\theta\omega}^0&=\Big\{\xi\in\mathcal{V}_{r_\theta\omega}:
\lim_{t\to\infty}\frac{1}{t}\log\|\xi\|_{g_{t}r_\theta\omega}=0\Big\},\nonumber
\\
E_{r_\theta\omega}^-&=\Big\{\xi\in\mathcal{V}_{r_\theta\omega}:
\lim_{t\to+\infty}\frac{1}{t}\log\|\xi\|_{g_{t}r_\theta\omega}<0\Big\}.\nonumber
\end{align*}
The dimensions of $E_{r_\theta\omega}^+$ and $E_{r_\theta\omega}^-$ are equal to the
number of positive Lyapunov exponents of $(G^{\mathcal{V}}_t)_{t\in\R}$.

One of the main objectives of this paper is to prove (in Section~\ref{subsec:proof erg}) the following
criterion on ergodicity for translation flows on $\Z^d$-covers.

\begin{theorem}\label{thm:ergV}
Let $(M,\omega)$ be a compact connected translation surface and let $\mathcal{M}=\overline{SL(2,\R)\omega}$.
Suppose that $\mathcal{V}\to\mathcal{M}$ is a continuous $SL(2,\R)$-invariant subbundle of $\mathcal{H}_1(M,\R)$  such that
all Lyapunov exponents of the KZ-cocycle $(G_t^{\mathcal{V}})_{t\in\R}$ vanish.
Then for every connected $\Z^d$-cover $(\widetilde{M}_\gamma,\widetilde{\omega}_\gamma)$ given
by a $d$-tuple $\gamma=(\gamma_1,\ldots,\gamma_d)\in (\mathcal{V}_\omega\cap H_1(M,\Z))^d$
the directional flow in direction $\theta\in \R/2\pi\Z $ on the translation surface
$(\widetilde{M}_\gamma,\widetilde{\omega}_\gamma)$ is ergodic for
a.e.\ $\theta$.
\end{theorem}

By Theorem~3 in \cite{Fo-Ma-Zo} we have the following result that will be applied  in the proof of Theorem~\ref{thm:ergV}.

\begin{theorem}\label{thm:fil}
Let $\mathcal{V}\to \mathcal{M}$ be a continuous $SL(2,\R)$-invariant subbundle of $\mathcal{H}_1(M,\R)$.
If all Lyapunov exponents of the KZ-cocycle $(G_t^{\mathcal{V}})_{t\in\R}$ vanish then $\|\xi\|_{g\omega}=\|\xi\|_{\omega}$ for all  $\xi\in \mathcal{V}_{\omega}$ and $g\in SL(2,\R)$.
\end{theorem}

Suppose that $(M,\omega)$ is an orientation cover of a compact half-translation surface $(X,q)$.
Then the $SL(2,\R)$-invariant symplectic subspace $H^+_1(M,\R)$ determines an $SL(2,\R)$-invariant symplectic subbundle of $\mathcal{H}^+_1$
which is defined and continuous over $\mathcal{M}$. The fibers of this bundle can be identified
with the space $H^+_1(M,\R)=H_1(X,\R)$ so the dimension of each fiber is $2g_X$, where $g_X$ is the genus of $X$.
The Lyapunov exponents of the bundle $\mathcal{H}^+_1$ are called the Lyapunov exponents of the half-translation surface $(X,q)$.
We denote by $\lambda_{top}(q)$ the largest exponent.

\begin{proof}[Proof of Theorem~\ref{thm:mainquadr:intr}]
Theorem~\ref{thm:ergV} applied to the subbundle $\mathcal{H}^+_1$ together with  Remark~\ref{rem:erg} completes the proof.
\end{proof}

\subsection{Skew product representation}
Let $\theta\in\R/2\pi\Z$ be a direction such that
the flow $({\varphi}^\theta_t)_{t\in\R}$ on $(M,\omega)$ is ergodic and has no saddle connections. Let $I\subset M\setminus\Sigma$ be an interval transversal to the direction $\theta$ with no
self-intersections. Then the Poincar\'e return map $T:I\to I$ is an ergodic interval exchange transformation (IET) which satisfies the Keane property. Denote by $(I_\alpha)_{\alpha\in\mathcal{A}}$ the family of exchanged
intervals. For every $\alpha\in\mathcal{A}$ we will denote by $\xi_\alpha=\xi_\alpha(\omega,I)\in H_1(M,\Z)$ the homology class of any loop formed by the segment of orbit for
$(\varphi^\theta_t)_{t\in\R}$ starting at any $x\in\Int I_\alpha$ and ending at $Tx$ together with the segment of $I$ that joins $Tx$ and $x$, that we will
denote by $[Tx,x]$.

\begin{proposition}[see \cite{Fr-Ulc:nonerg} for $d=1$]\label{lem_flow_auto}
For every $\gamma\in H_1(M,\Z)^d$ the
directional flow $(\widetilde{\varphi}^\theta_t)_{t\in\R}$  on the
$\Z^d$-cover $(\widetilde{M}_\gamma,\widetilde{\omega}_\gamma)$
has a special representation over the skew product
$T_{\psi_\gamma}:I\times\Z^d\to I\times\Z^d$ of the form
$T_{\psi_\gamma}(x,n)=(Tx,n+\psi_\gamma(x))$, where
$\psi_\gamma:I\to\Z^d$ is a piecewise constant function given by
\begin{equation}\label{eq:psigamma}
\psi_\gamma(x)=\langle \gamma,\xi_\alpha\rangle=\big(\langle \gamma_1,\xi_\alpha\rangle,\ldots,\langle
\gamma_d,\xi_\alpha\rangle\big)  \quad\text{ if }\quad x\in
I_\alpha\quad\text{ for }\alpha\in\mathcal{A}.
\end{equation}
In particular, the ergodicity  of the flow
$(\widetilde{\varphi}^\theta_t)_{t\in\R}$ on
$(\widetilde{M}_\gamma,\widetilde{\omega}_\gamma)$ is equivalent
to the ergodicity  of the skew product $T_{\psi_\gamma}:I\times
\Z^d\to I\times\Z^d$.
\end{proposition}
Since the ergodicity of the flow $(\widetilde{\varphi}^\theta_t)_{t\in\R}$ is equivalent to the ergodicity of
$T_{\psi_\gamma}$, this will allow us to apply the theory of essential values of cocycles to prove Theorem~\ref{thm:ergV} in Section~\ref{subsec:proof erg}.

\subsection{Ergodicity of skew products}
In this subsection we recall some general facts about
cocycles. For relevant background material concerning skew products
and infinite measure-preserving dynamical systems, we refer the
reader to \cite{Sch} and \cite{Aa}.

Let $G$ be a locally compact abelian second countable group. We
denote by $0$ its identity element, by $\mathcal{B}_{G}$ its
$\sigma$-algebra of Borel sets and by $m_{G}$ its Haar measure.
Recall that, for each ergodic automorphism $T:\xbm\to\xbm$ of a
standard Borel probability space, each measurable function
$\psi:X\rightarrow G$ defines a {\it skew product} automorphism
$T_{\psi}$ which preserves the $\sigma$-finite measure
${\mu}\times m_{G}$:
\begin{eqnarray*}
T_{\psi}:(X\times
G,{\mathcal{B}}\times{\mathcal{B}_{G}},{\mu}\times m_{G})
&\rightarrow&
(X\times G,{\mathcal{B}}\times{\mathcal{B}_{G}},{\mu}\times m_{G}), \nonumber \\
T_{\psi}(x,g)&=&(Tx,g+{\psi}(x)),
\end{eqnarray*}
Here we use $G=\Z^d$. The function $\psi:X\rightarrow
G$ determines also a {\em cocycle} $\psi^{(\,\cdot\,)}:\Z\times
X \to G$ for the automorphism $T$ by the formula
\[
\psi^{(n)}(x)=\begin{cases}
\quad\sum_{0\leq j<n}\psi(T^jx) & \text{if }  n\geq0 \\
-\sum_{n\leq j<0}\psi(T^jx) &
\text{if }  n<0.
\end{cases}
\]
Then $T^n_{\psi}(x,g)=(T^nx,g+{\psi}^{(n)}(x))$ for every
$n\in\Z$.

An element $g\in G$ is said to be an {\em essential value} of $\psi$, if for every open neighbourhood $V_g$ of $g$ in $G$ and any set $B\in\mathcal{B}$,
$\mu(B)>0$, there exists $n\in\Z$ such that
\[
\mu(B\cap T^{-n}B\cap\{x\in X:\psi^{(n)}(x)\in V_g\})>0.
\]
The set of essential values of $\psi$ is denoted by
${E}(\psi)$.

\begin{proposition}[see \cite{Sch}]
The set of essential values $E(\psi)$ is a closed subgroup of $G$ and the
skew product $T_{\psi}$ is ergodic if and only if $E(\psi)=G$.
\end{proposition}

\begin{proposition}[see \cite{Co-Fr}]\label{prop:conze-fraczek}
Let $(X,d)$ be a compact metric space, $\mathcal{B}$ the $\sigma$--algebra of  Borel sets and
$\mu$ be a probability Borel measure on $X$.  Suppose that $T:\xbm\to\xbm$ is an ergodic measure--preserving automorphism and there exists an increasing sequence of natural numbers $(h_n)_{n\geq 1
}$ and a sequence of Borel sets $(C_n)_{n\geq 1}$ such that
\[\mu(C_n)\to\alpha>0,\;\;\mu(C_n\triangle T^{-1}C_n)\to 0\;\;\mbox{ and }
\sup_{x\in C_n}d(x,T^{h_n}x)\to 0.\]
If $\psi:X\to G$ is a measurable cocycle such that $\psi^{(h_n)}(x)=g$ for all $x\in C_n$,
then $g\in E(\psi)$.
\end{proposition}

\subsection{Prof of Theorem~\ref{thm:ergV}}\label{subsec:proof erg}
In this section we prove the following result. In view of Theorems~\ref{thm:masur},~\ref{thm:esch}~and~\ref{thm:fil}, it proves Theorem~\ref{thm:ergV}.
\begin{theorem}\label{thm:pomocgl}
Let $(M,\omega)$ be a compact connected translation surface and
let $\gamma=(\gamma_1,\ldots,\gamma_d)\in H_1(M,\Z)^d$ be a $d$-tuple such that
the $\Z^d$-cover $\widetilde{M}_\gamma$ is connected and
$\|\gamma_i\|_{g\omega}=\|\gamma_i\|_{\omega}$ for all $1\leq i\leq d$ and $g\in SL(2,\R)$.
If a direction $\pi/2-\theta\in \R/2\pi\Z$ is Birkhoff and Masur generic for $\omega$
then the directional flow in direction $\theta$ on $(\widetilde{M}_\gamma,\widetilde{\omega}_\gamma)$ is ergodic.
\end{theorem}

Suppose that the directional flow $({\varphi}^\theta_t)_{t\in\R}$
on $(M,\omega)$ in a direction $\theta\in\R/2\pi\Z$ is ergodic and minimal.
Let $I\subset M\setminus\Sigma$ ($\Sigma$ is the set of zeros of
$\omega$) be an interval transversal to the direction $\theta$
with no self-intersections. The Poincar\'e return map $T:I\to I$
is a minimal ergodic IET, denote by $I_\alpha$, $\alpha\in\mathcal{A}$ the intervals exchanged
by $T$. Let $\lambda_\alpha(\omega,I)$ stands for the length of the interval $I_\alpha$.

Denote by $\tau:I\to\R_+$ the map of the first return time to $I$
for the flow $({\varphi}^\theta_t)_{t\in\R}$. Then $\tau$ is constant on each $I_\alpha$
and denote by $\tau_\alpha=\tau_\alpha(\omega,I)>0$ its value on $I_\alpha$ for all $\alpha\in\mathcal{A}$.
Let us denote by $\delta(\omega,I)>0$ the maximal number $\Delta>0$ for which
the set $\{\varphi^\theta_tx:t\in[0,\Delta), x\in I\}$ does not contain any singular point (from $\Sigma$).

Denote by $(\widetilde{\varphi}^\theta_t)_{t\in\R}$ the
directional flow for a  $\Z^d$-cover
$(\widetilde{M}_\gamma,\widetilde{\omega}_\gamma)$ of $(M,\omega)$.

In view of Proposition~\ref{lem_flow_auto},  there exist generators
$\xi_\alpha(I)=\xi_\alpha(\omega,I)$, $\alpha\in\mathcal{A}$  of $H_1(M,\Z)$
such that  the Poincar\'e  return map $\widetilde{T}$ of the flow
$(\widetilde{\varphi}^\theta_t)_{t\in\R}$ to $p^{-1}(I)$ ($p:\widetilde{M}_\gamma\to M$ the covering map) is isomorphic  to the
skew product $T_\psi:I\times\Z^d\to I\times\Z^d$ of the form
$T_\psi(x,n)=(Tx,n+\psi(x))$, where $\psi=\psi_{\gamma,I}:I\to\Z^d$ is
a piecewise constant function given by
\[\psi_{\gamma,I}(x)=\langle \gamma,\xi_\alpha(I)\rangle=\big(\langle \gamma_1,\xi_\alpha(I)\rangle,\ldots,\langle \gamma_d,\xi_\alpha(I)\rangle\big)  \ \text{ if
}\ x\in I_\alpha\ \text{ for each }\alpha\in\mathcal{A}.\]

Suppose that $J\subset I$ is a subinterval. Denote by $S:J\to J$
the Poincar\'e return map to $J$ for the flow
$({\varphi}^\theta_t)_{t\in\R}$. Then $S$ is also an IET and suppose
it exchanges intervals $(J_\alpha)_{\alpha\in\mathcal{A}}$. The IET $S$ is the
induced transformation for $T$ on $J$. Moreover, all elements of $J_\alpha$ have the same
first return time to $J$ for the transformation $T$. Let us
denote this return time by $h_\alpha\geq 0$ for all $\alpha\in\mathcal{A}$.
Then $I$ is the union of disjoint towers $\{T^jJ_\alpha:0\leq j<h_\alpha\}$, $\alpha\in\mathcal{A}$.

\begin{lemma}\label{lem:aux1}
Suppose that $0\leq h\leq \min\{h_\alpha :\alpha\in\mathcal{A}\}$ is a number such that each $T^{j}J$ for $0\leq j<h$ is a subinterval of
some interval $I_\beta$, $\beta\in\mathcal{A}$.
Then for every $\alpha\in\mathcal{A}$ we have
\begin{equation}\label{eq:indcocyc}
\psi_{\gamma,I}^{(h_\alpha )}(x)=\langle \gamma,\xi_\alpha(J) \rangle\text{ and }|T^{h_\alpha}x-x|\leq |J|\text{ for every } x\in C_\alpha:=\bigcup_{0\leq j\leq h}T^jJ_\alpha.
\end{equation}
\end{lemma}
\begin{proof}
Let $\psi_{\gamma,J}:J\to\Z^d$ be the cocycle associated to the interval $J$.
Then
\[\psi_{\gamma,J}(x)=\sum_{0\leq j<h_\alpha }\psi_{\gamma,I}(T^jx)=\psi_{\gamma,I}^{(h_\alpha )}(x)\quad{ if }\quad x\in J_\alpha.\]
On the other hand, $\psi_{\gamma,J}(x)=\langle \gamma,\xi_\alpha (J) \rangle$  for $x\in J_\alpha $, so $\psi^{(h_\alpha  )}=\langle \gamma,\xi_\alpha (J) \rangle$ on $J_\alpha $.

If $x\in C_\alpha $ then $x=T^jx_0$ with $x_0\in J_\alpha  $ and $0\leq j\leq h$. Moreover,
\[\psi_{\gamma,I}^{(h_\alpha  )}(x)-\psi_{\gamma,I}^{(h_\alpha  )}(x_0)=\psi_{\gamma,I}^{(h_\alpha  )}(T^jx_0)-\psi_{\gamma,I}^{(h_\alpha  )}(x_0)=\sum_{i=0}^{j-1}(\psi_{\gamma,I}(T^iT^{h_\alpha }x_0)-\psi_{\gamma,I}(T^ix_0)).\]
Since $x_0$ and $T^{h_\alpha  }x_0=Sx_0$ belong to $J$, by assumption, for all $0\leq i<h$ the points $T^iT^{h_\alpha  }x_0$  and $T^ix_0$ belong
the interval $T^iJ\subset I_\beta$ for some $\beta\in\mathcal{A}$.
Therefore,
\[|T^{h_\alpha }x-x|=|T^jT^{h_\alpha  }x_0-T^jx_0|\leq|T^jJ|=|J|\text{ and }\psi_{\gamma,I}(T^iT^{h_\alpha  }x_0)=\psi_{\gamma,I}(T^ix_0)\]
for every $0\leq i<j$. It follows that $\psi_{\gamma,I}^{(h_\alpha  )}(x)=\psi_{\gamma,I}^{(h_\alpha  )}(x_0)=\langle \gamma,\xi_\alpha (J) \rangle$.
\end{proof}

\begin{lemma}\label{lem:aux2}
Let $\Delta>0$ be such that the set $\{\varphi^\theta_tx:t\in[0,\Delta), x\in J\}$ does not contain any singular point.
Let $h=[\Delta/|\tau|]$, where $|\tau|=\max\{\tau_\alpha :\alpha\in\mathcal{A}\}$. Then for every $0\leq j<h$ the set $T^{j}J$ is a subinterval
some interval $I_\beta$, $\beta\in\mathcal{A}$.
\end{lemma}
\begin{proof} Suppose, contrary to our claim, that $T^{j}J$ contains an end $x$ of some interval $I_\beta$. Then
$x=\varphi^\theta_{\tau^{(j)}(x_0)}(x_0)$ for some $x_0\in J$ and there is $0\leq s<\tau(x)$ such that  $\varphi^\theta_{s}x$ is a singular point.
Therefore,  $\varphi^\theta_{\tau^{(j)}(x_0)+s}x_0$ is a singular point and $\tau^{(j)}(x_0)+s<(j+1)|\tau|\leq h|\tau|\leq \Delta$, contrary to the assumption.
\end{proof}

The following result follows directly from Lemmas A.3 and A.4 in \cite{Fr-Hu}.
\begin{lemma}\label{lemma:consterg}
For every $(M,\omega)$ there exist positive
constants $A,C,c>0$ such that if $0\in \R/2\pi\Z $ is Birkhoff and Masur generic
then there exists a a sequence of nested horizontal intervals $(I_k)_{k\geq 0}$ in $(M,\omega)$ and an increasing divergent sequence of  real numbers $(t_k)_{k\geq 0}$
such that $t_0=0$ and  for every $k\geq 0$ we have
\begin{equation}\label{baseuniformerg}
 \frac{1}{c} \| \xi \|_{g_{t_k}\omega}   \leq
\max_{\alpha} \left| \langle\xi_\alpha (g_{t_k}\omega,I_k), \xi \rangle\right|
\leq c\| \xi \|_{g_{t_k}\omega} \quad \text{for every} \quad
 \xi \in H_1(M, \R),
\end{equation}
\begin{equation}\label{balanceerg}
 \lambda_\alpha(g_{t_k}\omega,I_k) \,
\delta(g_{t_k}\omega,I_k)  \geq A \ \text{and}\ \frac{1}{C}\leq
\tau_\alpha(g_{t_k}\omega,I_k)\leq C\ \text{for any}\  \alpha\in\mathcal{A}.
\end{equation}
\end{lemma}

\begin{proof}[Proof of Theorem~\ref{thm:pomocgl}]
Assume that the total area of $(M,\omega)$ is $1$.
Taking $\omega_0=r_{\pi/2-\theta}\omega$ we have $0\in \R/2\pi\Z $ is Birkhoff and Masur generic for $\omega_0$.
Since the flow $(\widetilde{\varphi}_t^\theta)_{t\in\R}$ on $(\widetilde{M}_\gamma,\widetilde{\omega}_\gamma)$ coincides with
the vertical flow on $(\widetilde{M}_\gamma,\widetilde{(\omega_0)}_\gamma)$, we need to prove the ergodcity of the latter flow.

By Lemma~\ref{lemma:consterg}, there exists a  sequence of nested horizontal intervals $(I_k)_{k\geq 0}$ in $(M,\omega_0)$ and an increasing divergent sequence of real numbers $(t_k)_{k\geq 0}$
such that \eqref{baseuniformerg} and \eqref{balanceerg} hold for $k\geq 0$ and $t_0=0$.

Let $I:=I_{0}$ and for the flow $(\widetilde{\varphi}_t^v)_{t\in\R}$ on $(\widetilde{M}_\gamma,\widetilde{(\omega_0)}_\gamma)$
denote by $T:I\to I$ and  $\psi:I\to\Z^d$ the corresponding IET and cocycle respectively.
For every $k\geq 1$ the first Poincar\'e return map $T_k:I_{k}\to I_{k}$  to $I_k$ for the vertical flow $(\varphi^v_t)_{t\in\R}$ on $(M,\omega_0)$
is an IET exchanging intervals $(I_k)_\alpha$, $\alpha\in\mathcal{A}$ whose length in $(M,\omega_0)$ are equal to $e^{-t_k}\lambda_\alpha(g_{t_k}\omega_0,I_k)$, $\alpha\in\mathcal{A}$, resp.
In view of \eqref{balanceerg}, the length of $I_k$ in $(M,\omega_0)$ is
\begin{equation*}
|I_k|=\sum_{\alpha\in\mathcal{A}}e^{-t_k}\lambda_\alpha(g_{t_k}\omega_0,I_k)\leq Ce^{-t_k}\sum_{\alpha\in\mathcal{A}}\lambda_\alpha(g_{t_k}\omega_0,I_k)\tau_\alpha(g_{t_k}\omega_0,I_k)=Ce^{-t_k}.
\end{equation*}
Moreover, by the definition of $\delta$, the set
\[\big\{\varphi_t^v(x):t\in\big[0,e^{t_k}\delta(g_{t_k}\omega_0,I_k)\big), x\in I_{k}\big\}\]
does not contain any singular point.

Denote by $h^k_\alpha\geq 0$ the first return time of the interval $(I_k)_\alpha$ to $I_k$ for the IET $T$.
Let
\[h_k:=\big[e^{t_k}\delta(g_{t_k}\omega_0,I_k)/|\tau(\omega_0,I)|\big]\text{ and }C_\alpha^k:=\bigcup_{0\leq j\leq h_k}T^j(I_k)_\alpha.\]
Now Lemmas~\ref{lem:aux1}~and~\ref{lem:aux2} applied to $J=I_k$ and  $\Delta=e^{t_k}\delta(g_{t_k}\omega_0)$ give
\begin{equation}\label{eq:indcocyc1}
\psi^{(h_\alpha^k)}(x)=\langle \gamma,\xi_\alpha(g_{t_k}\omega_0,I_k)\rangle\text{ and }|T^{h_\alpha^k}x-x|\leq|I_k|\leq Ce^{-t_k}\ \text{ for }\ x\in C_\alpha^k
\end{equation}
for every $k\geq 1$ and $\alpha\in\mathcal{A}$.
Moreover, by \eqref{balanceerg},
\begin{equation}\label{eq:Ck}
Leb(C_\alpha^k)=
(h_k+1)|(I_k)_\alpha|
\geq \frac{e^{t_k}\delta(g_{t_k}\omega_0,I_k)}{|\tau(\omega_0,I)|}e^{-t_k}\lambda_\alpha(g_{t_k}\omega_0,I_k)\geq \frac{A}{|\tau(\omega_0,I)|}.
\end{equation}
By assumption, in view of \eqref{baseuniformerg}, we have
\[ c^{-1} \| \gamma_i\|_{g_{t_k}\omega_0}\leq\max_{\alpha\in\mathcal{A}}\|\langle \gamma_i,\xi_\alpha(g_{t_k}\omega_0,I_k)\rangle\|\leq c\| \gamma_i\|_{g_{t_k}\omega_0}=c\| \gamma_i\|_{\omega_0}\ \text{ for }\ 1\leq i\leq d.\]
Therefore for every $\alpha\in\mathcal{A}$ the sequence $\{\langle \gamma,\xi_\alpha(g_{t_k}\omega_0,I_k)\rangle\}_{k\geq 1}$ in $\Z^d$
is bounded. Passing to a subsequence, if necessary, we can assume the above sequences are constant.
In view of \eqref{eq:indcocyc1} and \eqref{eq:Ck}, Proposition~\ref{prop:conze-fraczek} gives $\langle \gamma,\xi_\alpha (g_{t_k}\omega_0,I_k)\rangle\in E(\psi)$ for every $\alpha\in\mathcal{A}$ and $k\geq 1$.
Recall that for every $k\geq 1$ the homology classes $\xi_\alpha (g_{t_k}\omega_0,I_k)$, $\alpha\in\mathcal{A}$ generate $H_1(M,\Z)$.
As $\widetilde{M}_\gamma$ is connected, the homomorphism $H_1(M,\Z)\ni\xi\mapsto \langle\gamma,\xi\rangle\in\Z^d$ is surjective.
Therefore, for every $k\geq 1$ the vectors $\langle \gamma,\xi_\alpha (g_{t_k}\omega_0,I_k)\rangle$, $\alpha\in\mathcal{A}$ generate $\Z^d$.
Since $E(\psi)$ is a group and contains all these vectors, we obtain $E(\psi)=\Z^d$, so the skew product $T_\psi$ is ergodic.
In view of Proposition~\ref{lem_flow_auto}, the vertical flow on $(\widetilde{M}_\gamma,\widetilde{(\omega_0)}_\gamma)$ is ergodic,
which completes the proof.
\end{proof}

\subsection{Some comments on Theorem~\ref{thm:ergV}}
Let  $\omega \in\mathcal{M}_1(M)$ and denote by
$\mathcal{M}=\overline{SL(2,\R)\omega }$ the closure of the $SL(2,\R)$-orbit of $\omega $ in $\mathcal{M}_1(M)$. Denote by $\nu_{\mathcal{M}}$ the corresponding affine
$SL(2,\R)$-invariant ergodic probability measure supported on $\mathcal{M}$.
In view of \cite{EM} and \cite{Fil}, for any $SL(2,\R)$-invariant symplectic $\mathcal{V}$ subbundle defined over $\mathcal{M}$ there exists an $SL(2,\R)$-invariant continuous direct decomposition
\[\mathcal{V}=\mathcal{V}^1\oplus\mathcal{V}^2\oplus\ldots\oplus\mathcal{V}^m\]
such that each subbundle $\mathcal{V}^i$ is strongly irreducible. Denote by $\lambda^{\mathcal{V}^i}_{top}$ the maximal Lyapunov exponent of the reduced
Kontsevich-Zorich cocycle $(G_t^{\mathcal{V}^i})_{t\in\R}$ and with respect to the measure $\nu_{\mathcal{M}}$. As a step of the proof of Theorem 1.4 in \cite{Es-Ch} the authors showed also the following result:
\begin{theorem}\label{thm:maxexp}
If  $\xi\in \mathcal{V}^i_\omega$ is non-zero then for a.e.\ $\theta\in \R/2\pi\Z $ we have
\[\lim_{t\to\infty}\frac{1}{|t|}\log\|\xi\|_{g_t r_\theta\omega }=\lambda^{\mathcal{V}^i}_{top}.\]
\end{theorem}
A consequence of this result is the following:
\begin{theorem}
For every $\omega \in\mathcal{M}_1(M)$ and $\xi\in H_1(M,\R)$ there exists $\lambda(\omega ,\xi)\geq 0$ such that
\[\lim_{t\to\infty}\frac{1}{|t|}\log\|\xi\|_{g_t r_\theta\omega }=\lambda(\omega ,\xi)\text{ for a.e.\ $\theta\in \R/2\pi\Z $.}\]
\end{theorem}

\begin{proof}
Let us consider the bundle $\mathcal{H}_1(M,\R)$ defined over $\mathcal{M}$. Then there exists a continuous $SL_2(\R)$-invariant splitting
\begin{equation}\label{eq:split}
\mathcal{H}_1(M,\R)=\mathcal{V}^1\oplus\mathcal{V}^2\oplus\ldots\oplus\mathcal{V}^m
\end{equation}
such that each subbundle $\mathcal{V}^i$ is strongly irreducible. Then $\xi=\sum_{i=1}^m \xi_i$
such that $\xi_i\in \mathcal{V}^i_\omega$. Therefore, by Theorem~\ref{thm:maxexp}, for a.e.\ $\theta$ we have
\[\lim_{t\to\infty}\frac{1}{|t|}\log\|\xi\|_{g_t r_\theta\omega }=\max\{\lambda^{\mathcal{V}^i}_{top}: 1\leq i\leq m, \xi_i\neq 0\}\]
which completes the proof.
\end{proof}

The following result is a direct consequence of Theorem~\ref{thm:ergV} and yields some relationship between the value of the Lyapunov exponent $\lambda(\omega,\gamma)$ for $\gamma\in H_1(M,\Z)$
and the ergodic properties of translation flows on the $\Z^d$-cover $(\widetilde{M}_\gamma,\widetilde{\omega}_\gamma)$.

\begin{theorem}\label{thm:zreoabst}
Let $(M,\omega)$ be a compact translation surface and let $\gamma\in H_1(M,\Z)^d$ be such that $\widetilde{M}_\gamma$ is connected and $\lambda(\omega,\gamma_i)=0$ for $1\leq i\leq d$.
Then $(\widetilde{\varphi}^\theta_t)_{t\in \R}$  is ergodic for almost every $\theta\in\R/2\pi\Z$.
\end{theorem}

\begin{proof}
We present the arguments of the proof only  for $d=1$. In the higher dimensional case, the proof runs along similar lines.

Let us consider the $SL_2(\R)$-invariant splitting \eqref{eq:split} into strongly irreducible subbundles and let $\gamma=\sum_{i=1}^m\gamma_i$ be such that
$\gamma_i\in \mathcal{V}^i_\omega$. Since $\lambda(\omega,\gamma)=0$, by Theorem~\ref{thm:maxexp}, $\gamma_i\neq 0$   implies $\lambda_{top}^{\mathcal{V}^i}=0$. Let
\[\mathcal{V}^\gamma:=\bigoplus\{\mathcal{V}^i:1\leq i\leq m, \gamma_i\neq 0\}.\]
Then $\mathcal{V}^\gamma$ is a non-zero $SL_2(\R)$-invariant subbundle so that $\gamma\in \mathcal{V}^\gamma_\omega$ and all Lyapunov  exponents of the restricted
KZ-cocycle $(G_t^{\mathcal{V}^\gamma})_{t\in\R}$ with respect to the measure $\nu_{\mathcal{M}}$ vanish. Then Theorem~\ref{thm:ergV} provides the final argument.
\end{proof}

Finally, we can formulate a conjecture which was stated  so far informally in the translation surface community.
It expresses completely the relationship between the value of the Lyapunov exponent
and the ergodic properties of translation flows on the $\Z$-covers on compact surfaces.

\begin{conjecture}
Let $(M,\omega)$ be a compact translation surface and let $(\widetilde{M}_\gamma,\widetilde{\omega}_\gamma)$ be its connected $\Z$-cover given by $\gamma\in H_1(M,\Z)$.
Then
\begin{itemize}
\item[(i)] if $\lambda(\omega,\gamma)=0$ then $(\widetilde{\varphi}^\theta_t)_{t\in \R}$ is ergodic for almost every $\theta\in\R/2\pi\Z$;
\item[(ii)] if $\lambda(\omega,\gamma)>0$ then $(\widetilde{\varphi}^\theta_t)_{t\in \R}$ is non-ergodic for almost every $\theta\in\R/2\pi\Z$.
\end{itemize}
\end{conjecture}

The claim (i) is confirmed by Theorem~\ref{thm:zreoabst}. The truth of the claim (ii) is suggested only
by a much weaker result proved in \cite{Fr-Ulc:nonerg}.

\section{Non-ergodicity and trapping for typical choice of periodic system of Eaton lenses}\label{sec:typtrap}
In this section we present the proof of Theorem~\ref{thm:typtrap}.

Let $\Lambda\subset \C$ be a lattice. For any  quadratic differential  $q$  on the torus $X:=\C/\Lambda$ we denote by $\widetilde{q}$ the pullback of $q$ by the projection map $p:\C\to\C/\Lambda$.
Denote by ${\mathcal{F}}_{\theta}$ and $\widetilde{\mathcal{F}}_{\theta}$ the measured foliations  in a direction $\theta\in\R/\pi\Z$ derived from $(X,q)$ and $(\C,\widetilde{q})$ respectively.
Recall that a foliation $\widetilde{\mathcal{F}}_{\theta}$ \emph{trapped}, if there exists a vector ${v} \in S^1 \subset \C$ and a constant $C$ such that every leaf of $\widetilde{\mathcal{F}}_{\theta}$ is trapped
in an infinite band of width $C$ parallel to ${v}$. Of course, every trapped foliation is highly non-ergodic.

Let $(M,\omega)$ be the orientation  cover of the half-translation torus $(X,q)$ and let $\pi:M\to X$ be the corresponding branched covering map. Then the space $H_1^+(M,\R)\simeq H_1(X,\R)$ of vectors invariant under the deck exchange map on homology
is a two dimensional real space.
Denote by $\gamma_1,\gamma_2\in H_1(X,\Z)\simeq H_1^+(M,\Z)$ two homology elements determining the $\Z^2$-covering $p:\C\to X$. Since $\gamma_1,\gamma_2$ are linearly independent, they span the space
$H_1(X,\R)\simeq H_1^+(M,\R)$.
Let $(\widetilde{M},\widetilde{\omega})$ be the $\Z^2$-cover of $(M,\omega)$
given by the pair $(\gamma_1,\gamma_2)\in H_1^+(M,\Z)^2$.
For every $\theta\in\R/2\pi\Z$ let $M^+_\theta$ be the set of points $x\in M$ such that  the positive semi-orbit $(\varphi^\theta_t(x))_{t\geq 0}$ on $(M,\omega)$ is well defined.

Let $D\subset \widetilde{M}$
be a bounded fundamental domain of the $\Z^2$-cover such that the
interior of $D$ is path-connected and the boundary of $D$ is a
finite union of intervals. For every $x\in M^+_\theta$ and $t>0$
define the element $\sigma^{\theta}_t(x) \in H_1(M,\Z)$ as
the homology class of the loop formed by the segment of the orbit
of $x$ from $x$ to $\varphi^\theta_t(x)$ closed up by the shortest curve
joining $\varphi^\theta_t(x)$ with $x$ that does not cross
$p^{-1}(\partial D)$.

The following result is a more general version of Theorem~3.2 in \cite{Fr-S}.
Since its proof runs essentially as in \cite{Fr-S}, we omit it.

\begin{proposition}
Assume that for a direction $\theta\in\R/2\pi\Z$ there is a non-zero homology class
$\xi\in\ H_1^+(M,\R)=\R\gamma_1+\R\gamma_2$ and $C>0$ such that
\[|\langle\sigma^\theta_t(x),\xi\rangle|\leq C\text{ for every }x\in
M^+_\theta\text{ and }t>0.\]
If the foliation $\mathcal{F}_\theta$ has no vertical saddle connection the lifted foliation $\widetilde{\mathcal{F}}_\theta$ is trapped.
\end{proposition}

Let $\mathcal{M}$ be the closure of the $SL(2,\R)$-orbit
of $(M,\omega)$ and denote by $\nu_{\mathcal{M}}$ the affine probability measure on $\mathcal{M}$. Let us consider the restriction
of the Konsevich-Zorich cocycle $(G_t^{\mathcal{H}_1^+})_{t\in\R}$ to the subbundle $\mathcal{H}_1^+(M,\R)\to \mathcal{M}$. Recall that
a.e.\ $\theta\in \R/2\pi\Z$ is Oseledets generic for the subbundle. This implies the existence of the stable subspace $E^-_{r_\theta\omega}\subset {H}_1^+(M,\R)$
whose dimension is equal to the number of positive Lyapunov exponents of $(G_t^{\mathcal{H}_1^+})_{t\in\R}$.
Moreover, by  Theorem~4.4 in \cite{Fr-Hu} we have.

\begin{proposition}
Suppose that
$\pi/2-\theta\in \R/2\pi\Z$ is a Birkhoff, Oseledets and Masur (BOM) generic direction for $(M,\omega)$. Then for every $\xi\in
E_{r_{\pi/2-\theta}\omega}^-$ there exists $C>0$ such that $|\langle\sigma^{\theta}_t(x),
\xi\rangle|\leq C$ for all $x\in M^+_\theta$ and $t>0$.
\end{proposition}

Since almost every direction is BOM generic, the previous two results yield the following criterion.
\begin{proposition}\label{prop:trap}
Suppose that the Lyapunov exponent $\lambda_{top}(q)$ of $(\C/\Lambda,q)$ is positive. Then for a.e.\ $\theta\in\R/\pi\Z$
the measured foliation $\widetilde{\mathcal{F}}_\theta$ on $(\C,\widetilde{q})$ is trapped.
\end{proposition}

To show the positivity of the Lyapunov exponents we will use Forni's criterion:

\begin{proposition}[Theorem~1.6 in \cite{Forni2011}]\label{forni:crit}
Let $(M,\omega)$ be a translation surface of genus $g$. Let $\mathcal{M}$ be the closure of the $SL(2,\R)$-orbit
of $(M,\omega)$ and denote by $\nu_{\mathcal{M}}$ the affine probability measure on $\mathcal{M}$. Suppose that all vertical regular orbits on
$(M,\omega)$ are periodic and there are $g$ different periodic orbits $\mathscr{O}_1,\ldots, \mathscr{O}_g$ such that
$M\setminus\{\mathscr{O}_1,\ldots, \mathscr{O}_g\}$ is homeomorphic to the $2g$-holed sphere.
Then all Lyapunov exponents of the Kontsevich-Zorich cocycle with respect to the measure $\nu_{\mathcal{M}}$ are positive.
\end{proposition}

Let $\Lambda\subset\C$ be a lattice and ${w}\in\Lambda$ a non-zero vector.
Let us fix a unit vector ${v}\in S^1 \subset \C$ linearly independent from $w$,
a $k$-tuple $\overline{c}=(c_1,\ldots,c_k)$ of different points on the torus
$\C/\Lambda$ and a $k$-tuple $\overline{r}=(r_1,\ldots,r_k)$ of positive numbers. Denote by $q_{v,\overline{c},\overline{r}}$ the quadratic differential on the torus $\C/\Lambda$ arising from the $k$ slit-folds parallel to $v$, centered at points $c_1,\ldots,c_k\in \C/\Lambda$ and with radii
$r_1,\ldots,r_k$ respectively. If all slit-folds are pairwise disjoint then $q_{v,\overline{c},\overline{r}}\in\mathcal{Q}((-1)^{2k},2^{k})$.

For every $1\leq j\leq k$ denote by $S_j(w)\subset \C/\Lambda$ the shadow of the $j$-th slit in the direction $w$, i.e.\ $S_j(w)=\{c_j+sv+tw: s\in[-r_j,r_j],t\in[0,1]\}$. A quadratic differential $q_{v,\overline{c},\overline{r}}$ is called \emph{separated by the vector ${w}\in \Lambda$}, if
each shadow $S_j(w)$ is a proper cylinder (not the whole torus) and any two different shadows $S_j(w)$, $S_{j'}(w)$ are either pairwise disjoint or
the centers $c_j$, $c_{j'}$ lie on the same linear loop parallel to the vector ${w}\in\Lambda$.

\begin{lemma}\label{lem:posexp}
If $q_{v,\overline{c},\overline{r}}$ is a quadratic differential on $\C/\Lambda$ which is separated by a non-zero vector ${w}\in\Lambda$ then
the Lyapunov exponent  $\lambda_{top}(q_{v,\overline{c},\overline{r}})$ is positive.
\end{lemma}

\begin{proof}
Without loss of generality we may assume $\Lambda=\Z^2$, so $w=(0,1)$ and $v=(1,0)$.
This assumption simplifies the argument. Let us divide the slit centers into $N$ cliques ($1\leq N\leq k$). Centers that lie on the same vertical linear loop are in a clique. Denote by $x_1,\ldots,x_N\in \R/\Z$ the horizontal coordinates of the cliques so that $x_1<x_2<\ldots<x_N<x_1+1$. We will also need cliques of the corresponding slit-folds; two slit-folds are in the same clique, if and only if their shadows in the vertical direction intersect, see Figure~\ref{sep_clique}.

Suppose that the $j$-th clique contains $m_j\geq 1$ slit-folds centered at $c_{j,l}:=(x_j,y_{j,l})\in\C/\Lambda$ for $1\leq l\leq m_j$ so that $y_{j,1}<y_{j,2}<\ldots<y_{j,m_j}<y_{j,1}+1$.
Then $\sum_{j=1}^Nm_j=k$.

Since the quadratic differential is separated by the vertical direction, there are exactly $N$ vertical linear loops that separate the cliques of slit-folds.
For $1\leq j\leq N$, denote by $s_j$ a vertical upward-oriented linear loop separating the $j$-th and
$(j+1)$-th cliques of slit-folds, see Figure~\ref{sep_clique}. We adopt throughout the periodicity convention that the $(N+1)$-th clique is the first one, i.e.\ $x_{N+1}=x_1$.

Let $(M_{v,\overline{c},\overline{r}},\omega_{v,\overline{c},\overline{r}})$ be the orientation cover of %the quadratic differential
$(\C/\Lambda,q_{v,\overline{c},\overline{r}})$.
Using Forni's criterion we will show that all Lyapunov exponents of $\omega_{v,\overline{c},\overline{r}}$ are positive. This implies the positivity of $\lambda_{top}(q_{v,\overline{c},\overline{r}})$.
Let $\pi: (M_{v,\overline{c},\overline{r}},\omega_{v,\overline{c},\overline{r}})\to (\C/\Lambda,q_{v,\overline{c},\overline{r}})$ be the natural projection.
Then the holomorphic one form $\omega_{v,\overline{c},\overline{r}}$ lies in  $\mathcal{M}(1^{2k})$ and the genus of $M_{v,\overline{c},\overline{r}}$ is $k+1$. More geometrically, $M_{v,\overline{c},\overline{r}}$
is the translation surface made of two copies of a slitted torus $\C/\Lambda$ (denoted by $\T_+$ -- left; and $\T_-$ -- right), where the slits replace the slit-folds on $(\C/\Lambda,q_{v,\overline{c},\overline{r}})$, see Figure~\ref{sep_clique}.
Let $\sigma:M_{v,\overline{c},\overline{r}}\to M_{v,\overline{c},\overline{r}}$ be the involution that exchanges the slitted tori $\T_+$ and $\T_-$ by translation. Finally, each side of any slit on $\T_+$ and $\T_-$
is glued to its $\sigma$-image by a 180 degree rotation. Denote by $\pi^{-1}_{\pm}:\C/\Lambda\to\T_\pm$ the two branches of the inverse of $\pi$.

\begin{figure}[!htb]
\centering
\includegraphics[width=1\textwidth]{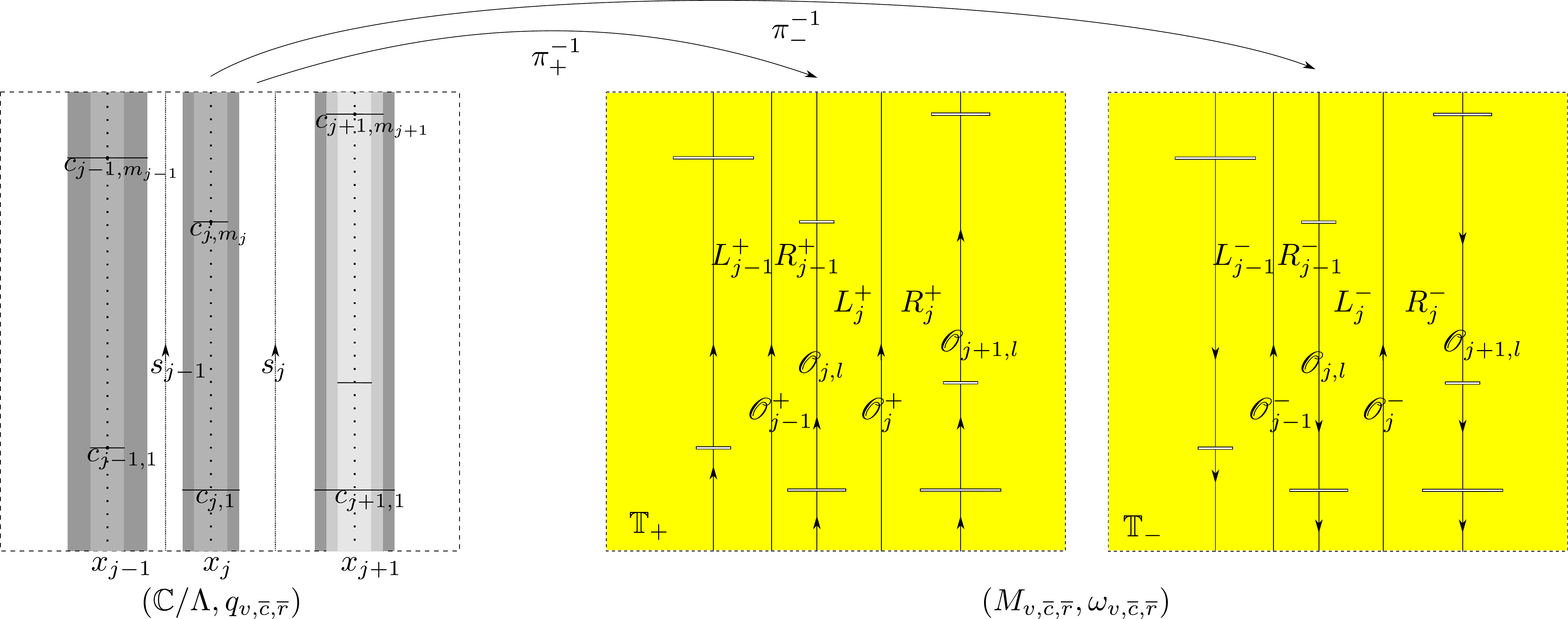}
\caption{The half-translation surface $(\C/\Lambda,q_{v,\overline{c},\overline{r}})$ and its orientation cover $(M_{v,\overline{c},\overline{r}},\omega_{v,\overline{c},\overline{r}})$.}
\label{sep_clique}
\end{figure}

Note that all regular vertical orbits on $(M_{v,\overline{c},\overline{r}},\omega_{v,\overline{c},\overline{r}})$ are periodic. We distinguish $k+2N$ such orbits:
\begin{itemize}
\item for every $1\leq j\leq N$ let $\mathscr{O}^{\pm}_j=\pi^{-1}_{\pm}(s_j)$;
\item for every $1\leq j\leq N$ and $1\leq l\leq m_j$ the orbit $\mathscr{O}_{j,l}$ is made of two vertical segments:
 the first one joins $\pi^{-1}_+(c_{j,l})$ and $\pi^{-1}_+(c_{j,l+1})$ inside $\T_+$
 and the second one joins $\pi^{-1}_-(c_{j,l+1})$ and $\pi^{-1}_-(c_{j,l})$ inside $\T_-$ (we adopt the convention that $c_{j,m_j+1}=c_{j,1}$).
\end{itemize}
Since $\pi^{-1}_+(c_{j,l})=\pi^{-1}_-(c_{j,l})$ in $M_{v,\overline{c},\overline{r}}$, the above two segments together yield a periodic orbit $\mathscr{O}_{j,l}$.

From these $k+2N$ periodic orbits we choose $k+1$, so that the surface obtained after removing the distinguished  $k+1$ orbits from $M_{v,\overline{c},\overline{r}}$ is homeomorphic to the $2(k+1)$-punctured sphere.
The choice of the periodic orbits depends on the parity of $N$. At first let us look at the surface
\[\underline{M}:=M_{v,\overline{c},\overline{r}}\setminus\Big(\bigcup_{j=1}^N\mathscr{O}^{+}_j\cup\bigcup_{j=1}^N\mathscr{O}^{-}_j \cup\bigcup_{j=1}^N\bigcup_{l=1}^{m_j}\mathscr{O}_{j,l}\Big).\]
For every $1\leq j\leq N$ let $R_j^{\pm}$ be the region of $\T_{\pm}$ that is bounded by the orbit $\mathscr{O}^{\pm}_j$ and the union $\bigcup_{l=1}^{m_{j+1}}\mathscr{O}_{j+1,l}$, see Figure~\ref{sep_clique}. Similarly,
$L_j^{\pm}$ is the region of $\T_{\pm}$ bounded by the orbit $\mathscr{O}^{\pm}_j$ and the union $\bigcup_{l=1}^{m_j}\mathscr{O}_{j,l}$. Then $\underline{M}$ is the union of $2N$ connected components and each such component $A_j^{\pm}$ is the union of $L_j^{\pm}\cup R_{j-1}^{\mp}$ for $1\leq j\leq N$; where we adopt the convention that $R_{0}^{\pm}=R_{N}^{\pm}$. The component $A_j^{\pm}$ is homeomorphic to the $m_{j}$-punctured annulus ($(m_{j}+2)$--punctured sphere)
and its boundary consists of orbits $\mathscr{O}^{\pm}_j$, $\mathscr{O}^{\mp}_{j-1}$ and $\mathscr{O}_{j,l}$ for $1\leq l\leq m_{j}$, see Figure~\ref{annulus}.

\begin{figure}[!htb]
\centering
\includegraphics[width=0.3\textwidth]{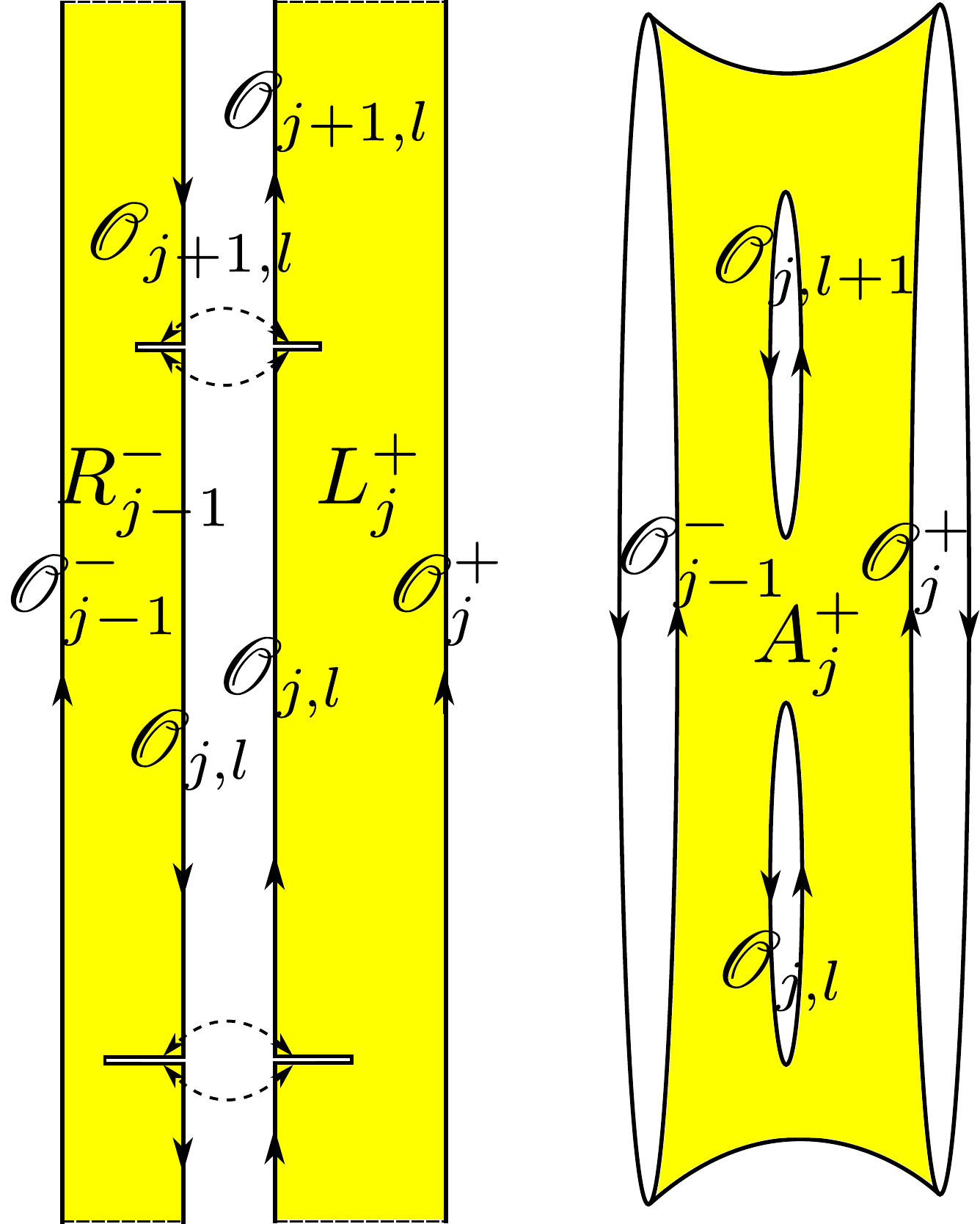}
\caption{The annulus  $A_j^+$.}
\label{annulus}
\end{figure}

\textbf{Odd case.} If $N$ is odd then we take:
$\mathscr{O}^{+}_1$ and $\mathscr{O}_{j,l}$ for $1\leq j\leq N$ and $1\leq l\leq m_{j}$. Since $\sum_{j=1}^Nm_j=k$, this yields
a family of $k+1$ vertical periodic orbits. Then the surface
\[\underline{M}_1:=M_{v,\overline{c},\overline{r}}\setminus\Big(\mathscr{O}^{+}_1\cup\bigcup_{j=1}^N\bigcup_{l=1}^{m_j}\mathscr{O}_{j,l}\Big)\]
is made of the punctured annuli $A^+_j$, $A^-_j$, $1\leq j\leq N$ glued along the loops $\mathscr{O}^{+}_j$ for $2\leq j\leq N$ and $\mathscr{O}^{-}_j$ for $1\leq j\leq N$.
Each such gluing yields a pattern $A^+_{j}, \mathscr{O}^{+}_{j},  A^-_{j+1}$ or $A^-_{j}, \mathscr{O}^{-}_{j},  A^+_{j+1}$; we adopt the convention that $A^\pm_{N+1}=A^{\pm}_1$.
Since $N$ is odd, all such junctures taken together are arranged in the following pattern:
\[ A^-_2,\mathscr{O}^{-}_{2}, A^+_3,  \ldots, A^{-}_{N-1}, \mathscr{O}^{-}_{N-1},  A^+_N, \mathscr{O}^{+}_{N}, A^-_1,\mathscr{O}^{-}_{1}, A^+_2, \ldots,A^{+}_{N-1},  \mathscr{O}^{+}_{N-1},  A^-_N, \mathscr{O}^{-}_{N},  A^+_1.\]
Since each annulus $A_j^{\pm}$ has $m_{j}$ punctures and appears in the above sequence exactly once, it follows that $\underline{M}_1$ is an annulus with $2\sum_{j=1}^N m_j=2k$ punctures. Therefore, $\underline{M}_1$ is homeomorphic to the $2(k+1)$-punctured sphere.

\textbf{Even case.} If $N$ is even then we take $k+1$ vertical periodic orbits: $\mathscr{O}^{+}_1$, $\mathscr{O}^{-}_1$,  $\mathscr{O}_{1,l}$ for  $2\leq l\leq m_{j}$ and $\mathscr{O}_{j,l}$ for $2\leq j\leq N$ and $1\leq l\leq m_{j}$.  Then the surface
\[\underline{M}_2:=M_{v,\overline{c},\overline{r}}\setminus\Big(\mathscr{O}^{+}_1\cup\mathscr{O}^{-}_1\cup\bigcup_{l=2}^{m_1}\mathscr{O}_{1,l}\cup\bigcup_{j=2}^N\bigcup_{l=1}^{m_j}\mathscr{O}_{j,l}\Big)\]
is made of the punctured annuli $A^+_j$, $A^-_j$, $1\leq j\leq N$ glued along the loops $\mathscr{O}^{+}_j$, $\mathscr{O}^{-}_j$ for $2\leq j\leq N$ and $\mathscr{O}_{1,1}$.
Each such gluing yields a pattern $A^+_{j}, \mathscr{O}^{+}_{j},  A^-_{j+1}$ or $A^-_{j}, \mathscr{O}^{-}_{j},  A^+_{j+1}$ or $A^{+}_{1},\mathscr{O}_{1,1},A^{-}_{1}$.
Since $N$ is even, all such junctures together are arranged in the following pattern:
\[ A^-_2,\mathscr{O}^{-}_{2},  A^+_{3},  \ldots,  A^+_{N-1}, \mathscr{O}^{+}_{N-1}, A^-_{N}, \mathscr{O}^{-}_{N},  A^+_1,\mathscr{O}_{1,1},A^-_1,\mathscr{O}^{+}_{N}, A^+_N,\mathscr{O}^{-}_{N-1},  A^+_{N-1},  \ldots \]
\[ \ldots , A^-_{3}, \mathscr{O}^{+}_{2},  A^+_2.\]
Since each annulus $A_j^{\pm}$ has $m_{j}$ punctures and appears in the above sequence exactly once, it follows that $\underline{M}_2$ is an annulus with $2\sum_{j=1}^N m_j=2k$ punctures. Therefore, $\underline{M}_1$ is homeomorphic to the $2(k+1)$-punctured sphere.

Applying Proposition~\ref{forni:crit} to the translation surface $(M_{v,\overline{c},\overline{r}},\omega_{v,\overline{c},\overline{r}})$ then yields the positivity of all Lyapunov exponents of  $\omega_{v,\overline{c},\overline{r}}$, and finally the positivity of  $\lambda_{top}(q_{v,\overline{c},\overline{r}})$.
\end{proof}

Lemma~\ref{lem:posexp} combined with Proposition~\ref{prop:trap} leads to a trapping criterion for slit-folds systems $\widetilde{q}_{v,\overline{c},\overline{r}}$. Recall that
$\widetilde{q}_{v,\overline{c},\overline{r}}$ is the half-translation structure on $\C$ given by the system of slit-folds parallel to the vector $v$, centered at $\{c_1,\ldots,c_k\}+\Lambda$ and whose radii are $r_1,\ldots,r_k$ respectively.
\begin{corollary}\label{cor:trap}
If $q_{v,\overline{c},\overline{r}}$ is a quadratic differential on $\C/\Lambda$ which is separated by a non-zero vector ${w}\in\Lambda$ then
the measured foliation $\widetilde{\mathcal{F}}_{\theta}$ of $(\C,\widetilde{q}_{v,\overline{c},\overline{r}})$ is trapped for almost every $\theta\in\R/\pi\Z$.
\end{corollary}

Let $S$ be an infinite system of Eaton lenses on $\C$ and let $\theta\in\R/\pi\Z$. Then $\mathscr{P}_{S,\theta}$ is an invariant set for the geodesic flow consisting of four copies of each lens and two copies of the complement of the lenses with planar geometry. This gives a natural projection $\pi_{S,\theta}:\mathscr{P}_{S,\theta}\to\C$ associating
the footpoint (in  $\C$) to any unit tangent vector in $\mathscr{P}_{S,\theta}$.
We call the geodesic flow on $\mathscr{P}_{S,\theta}$ trapped if
\[\exists_{C>0}\ \exists_{{u}\in\C, |{u}|=1}\ \forall_{t\in\R}\ \forall_{x\in\mathscr{P}_{S,\theta}}\quad|\langle\pi_{S,\theta}(\mathcal{G}^\mathcal{S,\theta}_tx)-\pi_{S,\theta}(x),{u}\rangle|\leq C.\]

\begin{remark}\label{rem:meas}
Note that the geodesic flow on $\mathscr{P}_{S,\theta}$ is trapped, if and only if
\begin{align*}\exists_{0<C\in\Q}\ \forall_{N\in\N}\ \exists_{{u}_N\in\Q\times\Q, 1\leq|{u}_N|\leq 2}\ \forall_{t\in\Q\cap[-N,N]}\ \forall_{y\in \Q\times\Q}\\ (\pi_{S,\theta}(x)=y)\ \Longrightarrow\ |\langle\pi_{S,\theta}(\mathcal{G}^\mathcal{S,\theta}_tx)-y,{u}_N\rangle|\leq C.
\end{align*}
Moreover, the geodesic flow on $\mathscr{P}_{S,\theta}$ is trapped, if and only if the direction $\theta$ foliation on the corresponding slit-fold plane is trapped.
\end{remark}

Let $\Lambda$ be a lattice on $\C$ and let $\overline{c}=(c_1,c_2,\ldots,c_k)\in \C^k$ be a vector  such that
the points $c_j+w$ are pairwise distinct for $1\leq j\leq k$ and ${w}\in\Lambda$. Each such vector is called \emph{proper}.
A vector of radii $\overline{r}=(r_1,r_2,\ldots,r_k)\in \R_{>0}^k$
is called $(\Lambda,\overline{c})$-\emph{admissible} if $\text{dist} (c_i+\Lambda, c_j+\Lambda)> r_i+r_j$ for $i\neq j$.
%on the torus $\C/\Lambda$ with respect to the distance induced from the euclidian distance in $\C$.
Admissibility guarantees that Eaton lenses of radius $r_j$ centered at $c_j+\Lambda$ for $1 \leq j\leq k$
do not intersect. Recall, that such a $\Lambda$-periodic system of Eaton lenses is denoted by $L(\Lambda,\overline{c},\overline{r})$.
Of course, the set of $(\Lambda,\overline{c})$-admissible vectors is open in $\R^k$.

Let $\mathcal{A}=\{A_1,\ldots, A_m\}$ be a partition of $\{1,\ldots,k\}$. Then for every $\overline{r}\in\R^k$ and $\overline{x}\in\R^m$
%, $m$ the number of parts in the partition,
denote by $\overline{r}(\overline{x})$ the vector in $\R^k$
defined by $\overline{r}(\overline{x})_j=x_lr_j$ whenever $j\in A_l$. In particular, taking  $\overline{x}=\overline{1}=(1,\ldots,1)\in \R^k$ gives $\overline{r}(\overline{1})=\overline{r}$.

Denote by $\text{Adm}_{\Lambda,\overline{c},\mathcal{A}}\subset \R^m_{>0}$ the set of all $\overline{x}\in \R^m_{>0}$ such that the vector $\overline{1}(\overline{x})$ is $(\Lambda,\overline{c})$-admissible. This is a non-empty open subset.

\begin{theorem}\label{thm:pertnonerg}
Suppose that a vector  $\overline{r}_0\in\R_{>0}^k$ is $(\Lambda,\overline{c})$-admissible. Then for every $\theta_0\in\R/\pi\Z$
there exists an open neighbourhood $U$ of $(\overline{1},\theta_0)$ in $\R_{>0}^m\times \R/\pi\Z$ such that for almost every $(\overline{x},\theta)\in U$
the vector $\overline{r}_0(\overline{x})$ is  $(\Lambda,\overline{c})$-admissible and
the geodesic flow on $\mathscr{P}_{L(\Lambda,\overline{c},\overline{r}_0(\overline{x})),\theta}$ is trapped,
and hence non-ergodic.
\end{theorem}

\begin{proof}
First we pass to the flat version of any admissible system $L(\Lambda,\overline{c},\overline{r})$ and its geodesic flow in direction $\theta\in\R/\pi\Z$. The resulting object
is the quadratic differential $\widetilde{q}_{e^{i(\theta+\pi/2)},\overline{c},\overline{r}}$ on $\C$ and its foliation  $\widetilde{\mathcal{F}}_{\theta}$. The geodesic flow on $\mathscr{P}_{L(\Lambda,\overline{c},\overline{r}),\theta}$  and
the foliation  $\widetilde{\mathcal{F}}_{\theta}$ are orbit equivalent.

For every $c\in\C$, $r>0$, $\theta\in\R/\pi\Z$ and $\xi\in\R/\pi\Z\setminus\{\theta\pm\pi/2\}$ let
\[\Delta_{c,r}(\theta,\xi)=\big\{c+rte^{i\theta}(s\tan(\theta-\xi)+i):t\in[-1,1],s\in[0,1]\big\}.\]
\begin{figure}[!htb]
\centering
\includegraphics[scale=0.6]{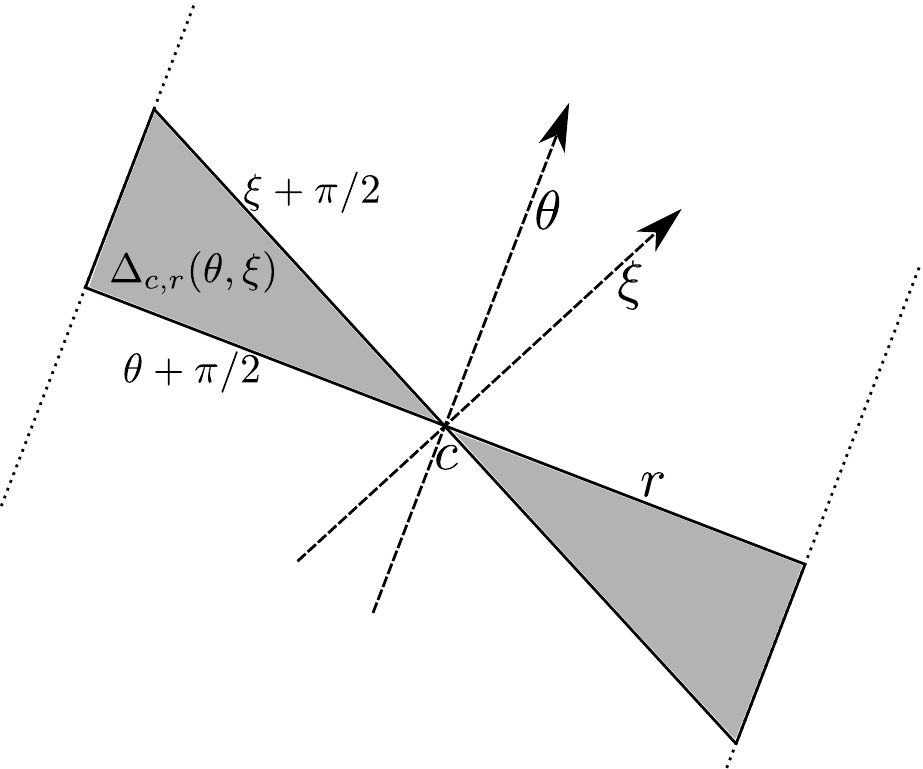}
\caption{The set $\Delta_{c,r}(\theta,\xi)$.}
\label{eaton6}
\end{figure}
Since $\overline{r}_0$ is $(\Lambda,\overline{c})$-admissible, the line segments $w+\Delta_{c_j,(\overline{r}_0)_j}(\theta_0,\theta_0)$ are pairwise disjoint for $1\leq j\leq k$ and ${w}\in\Lambda$.
Therefore we can choose  $\epsilon>0$  such that
for all $\overline{r}\in \overline{r}_0\big((0,(1+\epsilon)\sec\epsilon)^m\big)$  and $\theta,\xi\in(\theta_0-\epsilon,\theta_0+\epsilon)$
the sets $w+\Delta_{c_j,r_j}(\theta,\xi)$ are pairwise disjoint for $1\leq j\leq k$ and ${w}\in\Lambda$. Then $\widetilde{q}_{e^{i(\xi+\pi/2)},\overline{c},\sec(\theta-\xi)\overline{r}}$ is a railed deformation of $\widetilde{q}_{e^{i(\theta+\pi/2)},\overline{c},\overline{r}}$ along the direction $\theta$ and so their foliations in direction $\theta$
are Whitehead equivalent.

Since the set of directions arising from vectors in the lattice $\Lambda$ is dense, there is a vector ${w}\in\Lambda$
such that ${w}/|{w}|=ie^{i\theta_1}$ with $|\theta_1-\theta_0|<\epsilon$. Then all slit-folds
of $\widetilde{q}_{ie^{i\theta_1},\overline{c},\sec(\theta_0-\theta_1)(1+\epsilon)\overline{r}_0}$ are pairwise disjoint and parallel to the vector ${w}\in\Lambda$.
Next, choose a direction $\theta_2\neq \theta_1$  near enough to $\theta_1$ so that $|\theta_2-\theta_0|<\epsilon$ and $\widetilde{q}_{ie^{i\theta_2},\overline{c},\sec(\theta_0-\theta_2)(1+\epsilon)\overline{r}_0}$ is separated by $w\in\Lambda$.
It follows that $\widetilde{q}_{ie^{i\theta_2},\overline{c},\sec(\theta_0-\theta_2)\overline{r}_0(\overline{x})}$ is separated by $w\in\Lambda$ for every $\overline{x}\in(1-\epsilon,1+\epsilon)^m$.
Therefore, by Corollary~\ref{cor:trap}, for every $\overline{x}\in(1-\epsilon,1+\epsilon)^m$ and for a.e.\ $\theta\in (\theta_0-\epsilon,\theta_0+\epsilon)$ the foliation
$\widetilde{\mathcal{F}}_{\theta}$ on $\C$ derived from $\widetilde{q}_{ie^{i\theta_2},\overline{c},\sec(\theta_0-\theta_2)\overline{r}_0(\overline{x})}$
is trapped.

On the other hand  for every $\overline{x}\in(1-\epsilon,1+\epsilon)^m$ and $\theta\in (\theta_0-\epsilon,\theta_0+\epsilon)$  we have
\[\cos(\theta-\theta_2)\sec(\theta_0-\theta_2)\overline{x}\in (0,(1+\epsilon)\sec\epsilon)^m\ \text{ and }\ \theta,\theta_2\in (\theta_0-\epsilon,\theta_0+\epsilon).\]
Hence the quadratic differential
$\widetilde{q}_{ie^{i\theta_2},\overline{c},\sec(\theta_0-\theta_2)\overline{r}_0(\overline{x})}$ is a railed deformation of
$\widetilde{q}_{ie^{i\theta},\overline{c},\cos(\theta-\theta_2)\sec(\theta_0-\theta_2)\overline{r}_0(\overline{x})}$ along the direction $\theta$.
It follows that for every $\overline{x}\in(1-\epsilon,1+\epsilon)^m$ and for a.e.\ $\theta\in (\theta_0-\epsilon,\theta_0+\epsilon)$ the foliation
$\widetilde{\mathcal{F}}_{\theta}$ on $\C$ derived from $\widetilde{q}_{ie^{i\theta},\overline{c},\cos(\theta-\theta_2)\sec(\theta_0-\theta_2)\overline{r}_0(\overline{x})}$ is trapped, and hence
the geodesic flow restricted to  $\mathscr{P}_{L(\Lambda,\overline{c},\cos(\theta-\theta_2)\sec(\theta_0-\theta_2)\overline{r}_0(\overline{x})),\theta}$ is also trapped.

By Remark~\ref{rem:meas}, trapping is a measurable condition. Then a Fubini argument shows, that the geodesic flow on $\mathscr{P}_{L\left(\Lambda,\overline{c},\overline{r}_0\left(\cos(\theta-\theta_2)\sec(\theta_0-\theta_2)\overline{x}\right)\right),\theta}$ is trapped for a.e.\
$(\overline{x},\theta)\in (1-\epsilon,1+\epsilon)^m \times (\theta_0-\epsilon,\theta_0+\epsilon)$. Moreover, the map
\[(\overline{x}, \theta)\mapsto (\cos(\theta-\theta_2)\sec(\theta_0-\theta_2)\overline{x},\theta)\]
on $(1-\epsilon,1+\epsilon)^m \times (\theta_0-\epsilon,\theta_0+\epsilon)$ is a $C^\infty$ diffeomorphism. Denote by $U$ its image which is an open neighborhood of $(\overline{1},\theta_0)$. It follows that $\mathscr{P}_{L(\Lambda,\overline{c},\overline{r}_0(\overline{x})),\theta}$ is trapped for a.e.\ $(\overline{x},\theta)\in U$, which completes the proof.
\end{proof}

As a corollary we obtain the following more general version of Theorem~\ref{thm:typtrap}.

\begin{corollary}
For every lattice $\Lambda\subset\C$, every proper vector of centers $\overline{c}\in \C^k$ and every partition $\mathcal{A}$ of $\{1,\ldots,k\}$ the geodesic flow on $\mathscr{P}_{L(\Lambda,\overline{c},\overline{1}(\overline{r})),\theta}$ is trapped for a.e.\ $(\overline{r},\theta)\in \text{Adm}_{\Lambda,\overline{c},\mathcal{A}}\times\R/\pi\Z$.
\end{corollary}

\begin{example}
Let $\Lambda:=\Z(0,4)\oplus\Z(4,2)$. For every $\theta\in[0,\pi/4)$ let us consider the $\Lambda$-periodic pattern of lenses
\[\mathcal{L}_{\theta}=L\big(\Lambda,(0,\pm(1+i(1+\tan\theta)),(2\sin\theta,\cos\theta)\big).\]
This is the pattern of lenses drawn on Figure~\ref{german_animal_1}. By Theorem~\ref{thm:erglens}, for a.e.\ $\theta_0\in[0,\pi/4)$ the geodesic flow
on $\mathscr{P}_{\mathcal{L}_{\theta_0},\theta_0}$ is ergodic. On the other hand each pair $(\mathcal{L}_{\theta_0},\theta_0)$ satisfies the assumption of Theorem~\ref{thm:pertnonerg}.
Let us consider the partition $\mathcal{A}=\{\{1\},\{2,3\}\}$. Then, by Theorem~\ref{thm:pertnonerg}, after
almost every small perturbation of the direction $\theta_0$, the radius of the central lens and the radii
of the pair of symmetrically placed lenses, the ergodic properties of the geodesic flow change dramatically
to a highly non-ergodic trapped flow.

Let us now consider the partitions $\{\{1\},\{2\},\{3\}\}$ and $\{\{1,2,3\}\}$. By applying Theorem~\ref{thm:pertnonerg} to those, we obtain another type of results saying, that almost every small perturbation of
$(\mathcal{L}_{\theta_0},\theta_0)$ leads to a trapped geodesic flow. In the first case all radii are perturbed independently whereas in the second case all radii are perturbed simultaneously.

In summary, the curves of ergodic lens distributions described in the paper are very exceptional.
They are surrounded by highly non-ergodic systems. We have shown this phenomenon only for a particular "ergodic" curve, but for the other "ergodic" curves it can be shown along the same lines.

Moreover, we conjecture that the trapping property is measurably typical along many
curves transversal to the ergodic curves described in the paper. An interesting and highly
involved result of that type was proved in \cite{Fr-Shi-Ul}, where the authors consider
curves arising from fixed systems of lenses for which the direction $\theta$ varies.
\end{example}

\appendix

\section{Eaton lens dynamics} \label{system_of_Eaton_lenses}

To precisely describe the dynamics of light rays passing through an Eaton lens, we denote the lens of radius $R>0$ and centered at $(0,0)$ by $\overline{B}_R$. The refractive index (RI for short) in $\overline{B}_R$ depends only on the distance from the center $r:=\sqrt{x^2+y^2}\in(0,R]$
and is given by the formula $n(x,y)=n(r)=\sqrt{{2R}/{r}-1}$; at the center we put $n(0,0)=+\infty$.
Suppose, for simplicity, that the refractive index $n(x,y)$ is constant and equals $1$ outside $\overline{B}_R$.
Recall that the dynamics of light rays can be described as the geodesic flow on $\R^2\setminus(0,0)$ equipped with the Riemannian
metric $g=n\cdot(\operatorname{d}x\otimes \operatorname{d}x+\operatorname{d}y\otimes \operatorname{d}y)$.
Of course, the geodesics are straight lines or semi-lines outside $\overline{B}_R$. The dynamics of the geodesic flow inside $\overline{B}_R$ was described for example in \cite{Ha-Ha}.
After passing to polar coordinates $(r,\theta)$ we
use the Euler-Lagrange equation to see that any geodesic inside $\overline{B}_R$ satisfies
\begin{equation}\label{eq:EL}
 \frac{\operatorname{d}r}{\operatorname{d}\theta}=\pm\frac{r\sqrt{n(r)^2r^2-n(r_0)^2r_0^2}}{n(r_0)r_0}
 =\pm\frac{r\sqrt{r(2R-r)-r_0(2R-r_0)}}{\sqrt{r_0(2R-r_0)}},
\end{equation}
where $(r_0,\theta_0)$ is a point of the geodesic minimizing the distance to the center.
It follows that for any point $(r,\theta)$ of the geodesic in $\overline{B}_R$ we have
\begin{align*}
\pm(\theta-\theta_0)&=\int_{r_0}^r\frac{\sqrt{r_0(2R-r_0)}}{u\sqrt{u(2R-u)-r_0(2R-r_0)}}\,\operatorname{d}u\\&
=\Big[\arcsin\frac{Ru-r_0(2R-r_0)}{u(R-r_0)}\Big]^r_{r_0}=\arcsin\frac{Rr-r_0(2R-r_0)}{r(R-r_0)}+\frac{\pi}{2}.
\end{align*}
Consequently
\begin{equation}\label{eq:costheta}
-\cos(\theta-\theta_0)=\frac{Rr-r_0(2R-r_0)}{r(R-r_0)},
\end{equation}
and hence
\[\Big(\frac{r\cos(\theta-\theta_0)+(R-r_0)}{R}\Big)^2+\frac{(r\sin(\theta-\theta_0))^2}{R^2-(R-r_0)^2}=1.\]
In particular inside of $\overline{B}_R$ the geodesic is an arc of an ellipse.
Let $s:=\sqrt{R^2-(R-r_0)^2}$ and rotate the geodesic by $-\theta_0$. Then the equation of the ellipse becomes
\[\Big(\frac{x+\sqrt{R^2-s^2}}{R}\Big)^2+\Big(\frac{y}{s}\Big)^2=1.\]
\begin{figure}[!htb]
 \centering
\includegraphics[scale=0.40]{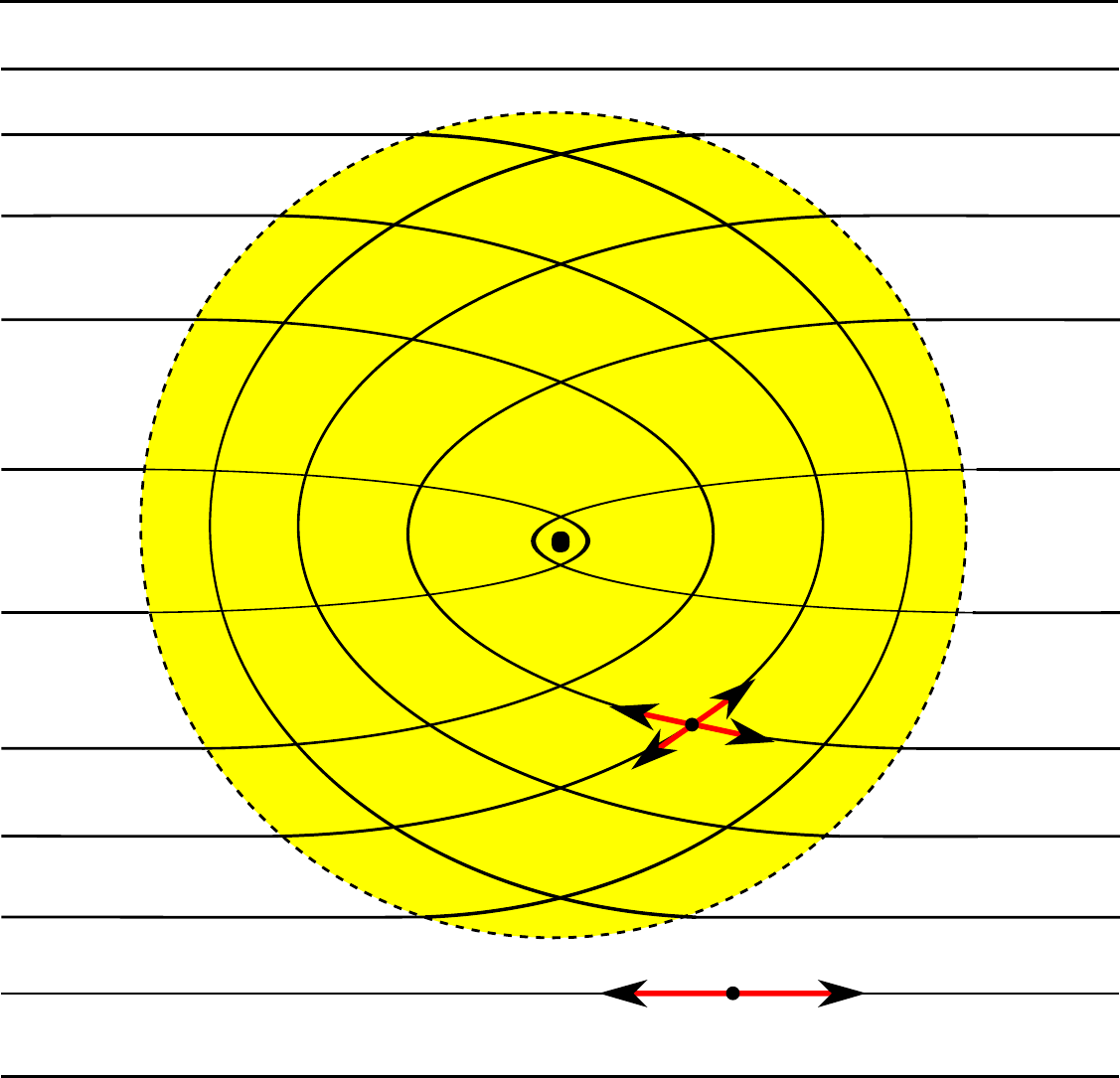}
\caption{Flow directions inside and outside of an Eaton lens}
\label{Eaton_hor}
\end{figure}
Since the ellipse is centered at $(-\sqrt{R^2-s^2},0)$ and $(-\sqrt{R^2-s^2},\pm s)$ are its intersection points
with the boundary of $\overline{B}_R$,
the geodesic has horizontal tangents at these intersecting points.
Rotating everything back to the original position we see, that the direction of any geodesic is reversed after passing through $\overline{B}_R$. The only exception is the trajectory that hits the center of the lens. For this trajectory we adopt
the convention, that at the center it turns and continues its motion backwards.
\begin{figure}[!htb]
\centering
\includegraphics[width=0.8\textwidth]{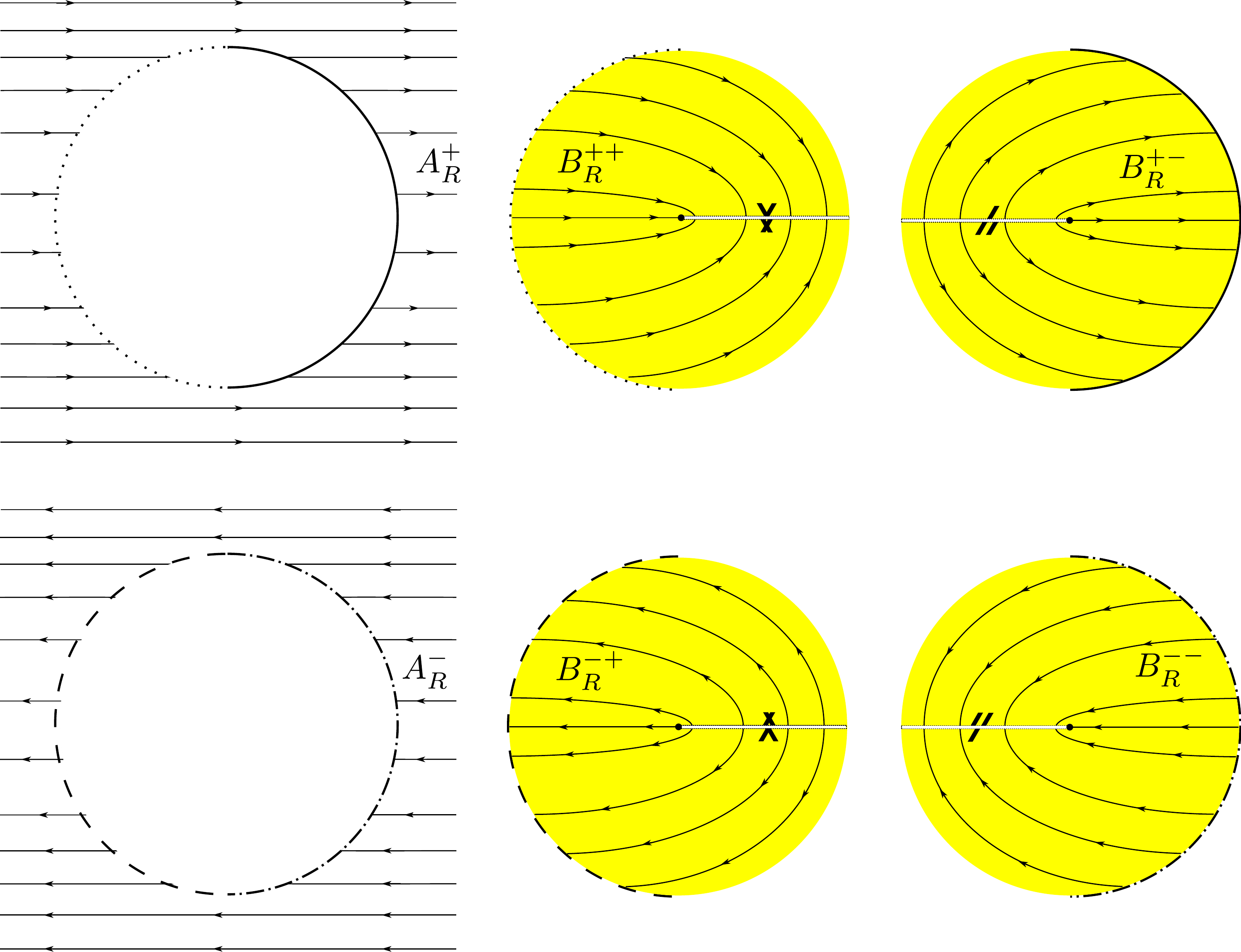}
\caption{Phase space of the horizontal flow in a neighborhood of an Eaton lens}
\label{Eaton_4}
\end{figure}

Now for every $\theta\in\R/\pi\Z$ consider the restriction of the geodesic flow $(\mathfrak{g}^\theta_t)_{t\in\R}$ to its invariant subset of the unit tangent bundle of $\R^2$ consisting of all trajectories assuming direction $\theta$ or $\pi+\theta$ outside $\overline{B}_R$. Denote by $\mathscr{P}_\theta$ the phase space of that  flow. Since all flows $(\mathfrak{g}^\theta_t)_{t\in\R}$ are isomorphic by rotations, we restrict our considerations to the horizontal flow $(\mathfrak{g}_t)_{t\in\R}=(\mathfrak{g}^0_t)_{t\in\R}$.

Denote by $B_R$ the interior of $\overline{B}_R$. Through every point of $B_R$ pass exactly four trajectories of $(\mathfrak{g}_t)_{t\in\R}$, while through every point of $A_R:=\R^2\setminus B_R$ pass exactly two,
in direction $0$ and $\pi$, see Figure \ref{Eaton_hor}. It follows, that $\mathscr{P}_\theta$ consists of four copies of $B_R$ ($B_R^{\pm\pm}$) and two copies of $A_R$ ($A_R^{\pm}$).

Let us take a closer look at the dynamics $(\mathfrak{g}_t)_{t\in\R}$ on the four copies of $B_R^{\pm\pm}$. Since they are related by a reflective symmetry or the reversal of time, we can restrict our considerations to one of them,
say $B_R^{++}$ in Figure \ref{Eaton_4}.
Consider the transversal curve for the flow  $(\mathfrak{g}_t)_{t\in\R}$ represented by the dotted semicircle $C_R$
parameterized by $(-\sqrt{R^2-s^2},s)$ for $s\in(-R,R)$. In view of \eqref{eq:EL} and \eqref{eq:costheta}
the trajectory of the point $(-\sqrt{R^2-s^2},s)\in C_R$ travels on an ellipse that is in polar coordinates given by
\begin{equation}\label{eq:sellipse}
-\cos\theta=\frac{Rr-s^2}{r\sqrt{R^2-s^2}}\quad\text{ and }\quad\frac{\operatorname{d}r}{\operatorname{d}\theta}=\frac{r\sqrt{r(2R-r)-s^2}}{s}
\end{equation}
before it escapes (the relevant copy of) $B_R^{++}$.
We write $(r(t,s),\theta(t,s))$ for the polar coordinates of $\mathfrak{g}_{t+\sqrt{R^2-s^2}}(-\sqrt{R^2-s^2},s)$. In particular,
$(r(-\sqrt{R^2-s^2},s),\theta(-\sqrt{R^2-s^2},s))$ are the polar coordinates of the point $(-\sqrt{R^2-s^2},s)$.
Since the velocity vectors of the geodesic flow have unit length with respect to the Riemannian metric $g$, we obtain
\begin{equation*}\label{eq:drthetadt}
\Big(\frac{\partial r}{\partial t}\Big)^2+r^2\Big(\frac{\partial \theta}{\partial t}\Big)^2=\frac{1}{n^2(r)}=\frac{r}{2R-r}.
\end{equation*}
Because of \eqref{eq:sellipse}, we have
\begin{equation}\label{eq:dtheta}
\frac{\partial \theta}{\partial t}=\frac{\partial r}{\partial t}\frac{s}{r\sqrt{r(2R-r)-s^2}}
\end{equation}
and hence
\[\frac{r}{2R-r}=\Big(\frac{\partial r}{\partial t}\Big)^2\Big(1+r^2\frac{s^2}{r^2(r(2R-r)-s^2)}\Big)=\Big(\frac{\partial r}{\partial t}\Big)^2\frac{r(2R-r)}{r(2R-r)-s^2}.\]
Therefore,
\begin{equation}\label{eq:drdt}
\frac{\partial r}{\partial t}=-\frac{\sqrt{r(2R-r)-s^2}}{2R-r}.
\end{equation}
Hence
\begin{align*}t+\sqrt{R^2-s^2}&=\int_{R}^{r(t,s)}\frac{u-2R}{\sqrt{u(2R-u)-s^2}}\operatorname{d}u\\
&=\Big[-\sqrt{u(2R-u)-s^2}+R\arcsin\frac{R-u}{\sqrt{R^2-s^2}}\Big]_{R}^{r(t,s)}\\
&=-\sqrt{r(2R-r)-s^2}+R\arcsin\frac{R-r}{\sqrt{R^2-s^2}}+\sqrt{R^2-s^2}
\end{align*}
and
\[t=-\sqrt{r(2R-r)-s^2}+R\arcsin\frac{R-r}{\sqrt{R^2-s^2}}.\]
Let $t_s$ be the exit time of $(-\sqrt{R^2-s^2},s)$ from $B_R^{++}$. Since $r(t_s,s)$ minimizes the distance to the origin, we have $s^2=R^2-(R-r(t_s,s))^2=r(t_s,s)(2R-r(t_s,s))$. It follows that
\[t_s=R\arcsin 1=\frac{1}{2}\pi R.\]

Introduce new coordinates on $B_R^{++}$ given by $(t,s)$. Then the set
 $E_R=B_R\cup([0,\pi R/2)\times (-R,R))$ is the domain of these coordinates
and they coincide with the cartesian coordinates on $C_R$.
Moreover, by definition, the geodesic flow $(\mathfrak{g}_t)_{t\in\R}$ in the new coordinates is the unit horizontal translation in positive direction.

One can define the same type of coordinates on the other copies $B_R^{+-}$, $B_R^{-+}$ and $B_R^{--}$. Let us consider a measure $\mu$ on $\mathscr{P}_0$ that coincides with the Lebesgue measure on $A_R^\pm$ and the Lebesgue measure  in the new coordinates on each $B_R^{\pm\pm}$.
This is a $(\mathfrak{g}_t)_{t\in\R}$-invariant measure and we will calculate its density in the next paragraph.

In view of \eqref{eq:drdt} and \eqref{eq:dtheta} we have
\begin{equation}\label{eq:dt}
\frac{\partial r}{\partial t}=-\frac{\sqrt{r(2R-r)-s^2}}{2R-r}\quad\text{ and }\quad\frac{\partial \theta}{\partial t}=-\frac{s}{r(2R-r)}.
\end{equation}
As
\[t=-\sqrt{r(2R-r)-s^2}+R\arcsin\frac{R-r}{\sqrt{R^2-s^2}},\]
differentiating it in the direction $s$ we obtain
\[0=-\frac{\frac{\partial r}{\partial s}(R-r)-s}{\sqrt{r(2R-r)-s^2}}+R\frac{-\frac{\partial r}{\partial s}+\frac{(R-r)s}{R^2-s^2}}{\sqrt{r(2R-r)-s^2}}.\]
Hence
\begin{equation}\label{eq:drds}
\frac{\partial r}{\partial s}=s\frac{R(2R-r)-s^2}{(R^2-s^2)(2R-r)}.
\end{equation}
Differentiating the first equality of \eqref{eq:sellipse}  in the direction $s$ we obtain
\[\frac{s\sqrt{r(2R-r)-s^2}}{r\sqrt{R^2-s^2}}\frac{\partial \theta}{\partial s}=\sin\theta\cdot\frac{\partial \theta}{\partial s}
=\frac{\partial r}{\partial s} \frac{s^2}{r^2\sqrt{R^2-s^2}}-\frac{s(R(2R-r)-s^2)}{r\sqrt{R^2-s^2}(R^2-s^2)}.\]
In view of \eqref{eq:drds}, it follows that
\begin{align}\label{eq:dthetads}
\begin{aligned}
\frac{\partial \theta}{\partial s}&=\frac{1}{\sqrt{r(2R-r)-s^2}}\Big(\frac{s^2(R(2R-r)-s^2)}{r(R^2-s^2)(2R-r)}-\frac{(R(2R-r)-s^2)}{R^2-s^2}\Big)\\
&=-\frac{R(2R-r)-s^2}{(R^2-s^2)}\frac{\sqrt{r(2R-r)-s^2}}{r(2R-r)}.
\end{aligned}
\end{align}
Putting \eqref{eq:dt}, \eqref{eq:drds} and \eqref{eq:dthetads} together, we have
\begin{align*}
\Big|\frac{\partial r}{\partial s}\cdot \frac{\partial \theta}{\partial t}-\frac{\partial r}{\partial t}\cdot \frac{\partial \theta}{\partial s}\Big|
&=\Big|\frac{s^2(R(2R-r)-s^2)}{r(R^2-s^2)(2R-r)^2}+\frac{(r(2R-r)-s^2)(R(2R-r)-s^2)}{r(R^2-s^2)(2R-r)^2}\Big|\\
&=\frac{(R(2R-r)-s^2)}{(R^2-s^2)(2R-r)}=\frac{1}{2R-r}\Big(1+\frac{R(R-r)}{R^2-s^2}\Big).
\end{align*}
By \eqref{eq:sellipse},
\[\sqrt{R^2-s^2}=\frac{\sqrt{r^2\cos^2\theta+4R(R-r)}-r\cos\theta}{2}=\frac{2R(R-r)}{\sqrt{r^2\cos^2\theta+4R(R-r)}+r\cos\theta}.\]
Hence,
\[\Big|\frac{\partial r}{\partial s}\cdot \frac{\partial \theta}{\partial t}-\frac{\partial r}{\partial t}\cdot \frac{\partial \theta}{\partial s}\Big|=
\frac{1}{2R-r}\Big(1+\frac{\big(\sqrt{r^2\cos^2\theta+4R(R-r)}+r\cos\theta\big)^2}{4R(R-r)}\Big)\]
Therefore, the density of the invariant measure $\mu$ restricted to $B_R^{++}$ in the cartesian coordinates is
\begin{equation*}\label{eq:density}
\xi_R(x,y)=\frac{2R-r}{r}\frac{4R(R-r)}{(\sqrt{x^2+4R(R-r)}+x)^2+4R(R-r)}.
\end{equation*}
On the other copies  $B_R^{\pm\pm}$ the measure $\mu$ is given by $\xi_R(\pm x,y)\operatorname{d}x\operatorname{d}y$.

For every $\theta \in\R/\pi\Z$ the flows $(\mathfrak{g}_t^\theta)_{t\in\R}$ phase space $\mathscr{P}_\theta$ is given by the rotation of  $\mathscr{P}_0$ by $\theta$ and the invariant measure $\mu_\theta$ is the rotation of $\mu$ by the same angle.

Generally instead of one Eaton lens on the plane we deal with a pattern $\mathcal{L}$ of infinitely many pairwise disjoint Eaton lenses on $\R^2$.
We are interested in the dynamics of the light rays provided by the geodesic flow $(\mathfrak{g}^\mathcal{L}_t)_{t\in\R}$ on $\R^2$ without the centers of lenses; the Riemann metric  is
given by $g_{(x,y)}=  n (x,y)\cdot (dx\otimes dx +dy\otimes dy)$.
The local behavior of the flow around  any lens was described in detail previously.
For every $\theta \in\R/\pi\Z$ there exists
an invariant set $\mathscr{P}_{\mathcal{L},\theta}$ in the unit tangent bundle, such that all trajectories on  $\mathscr{P}_{\mathcal{L},\theta}$ are tangent to $\pm e^{i\theta}$
outside the lenses. The restriction of $(\mathfrak{g}^\mathcal{L}_t)_{t\in\R}$ to  $\mathscr{P}_{\mathcal{L},\theta}$ is
denoted by $(\mathfrak{g}^{\mathcal{L},\theta}_t)_{t\in\R}$. Moreover,
$(\mathfrak{g}^{\mathcal{L},\theta}_t)_{t\in\R}$ possesses a natural invariant measure $\mu_{\mathcal{L},\theta}$ equivalent to the Lebesgue measure on $\mathscr{P}_{\mathcal{L},\theta}$. The density of $\mu_{\mathcal{L},\theta}$ is equal to one outside lenses and inside every lens of radius $R$ centered at $(c_1,c_2)$ is determined by $\xi_R(\pm(x-c_1),y-c_2)$ depending on its copy in the phase space. Moreover, the density is continuous on $\mathscr{P}_{\mathcal{L},\theta}$ and piecewise $C^\infty$.

\subsection{From the geodesic flow to translation surfaces and measured foliations}
For simplicity we return to a single lens and the horizontal flow $(\mathfrak{g}^\theta_t)_{t\in\R}$ on $\mathscr{P}_0$.
Representing $\mathscr{P}_0$ in $(t,s)$ coordinates, we can treat it as the union on $A^{\pm}_R$ and $D^{\pm\pm}_R$, see Figure~\ref{eaton_3}.
\begin{figure}[!htb]
\centering
\includegraphics[width=0.8\textwidth]{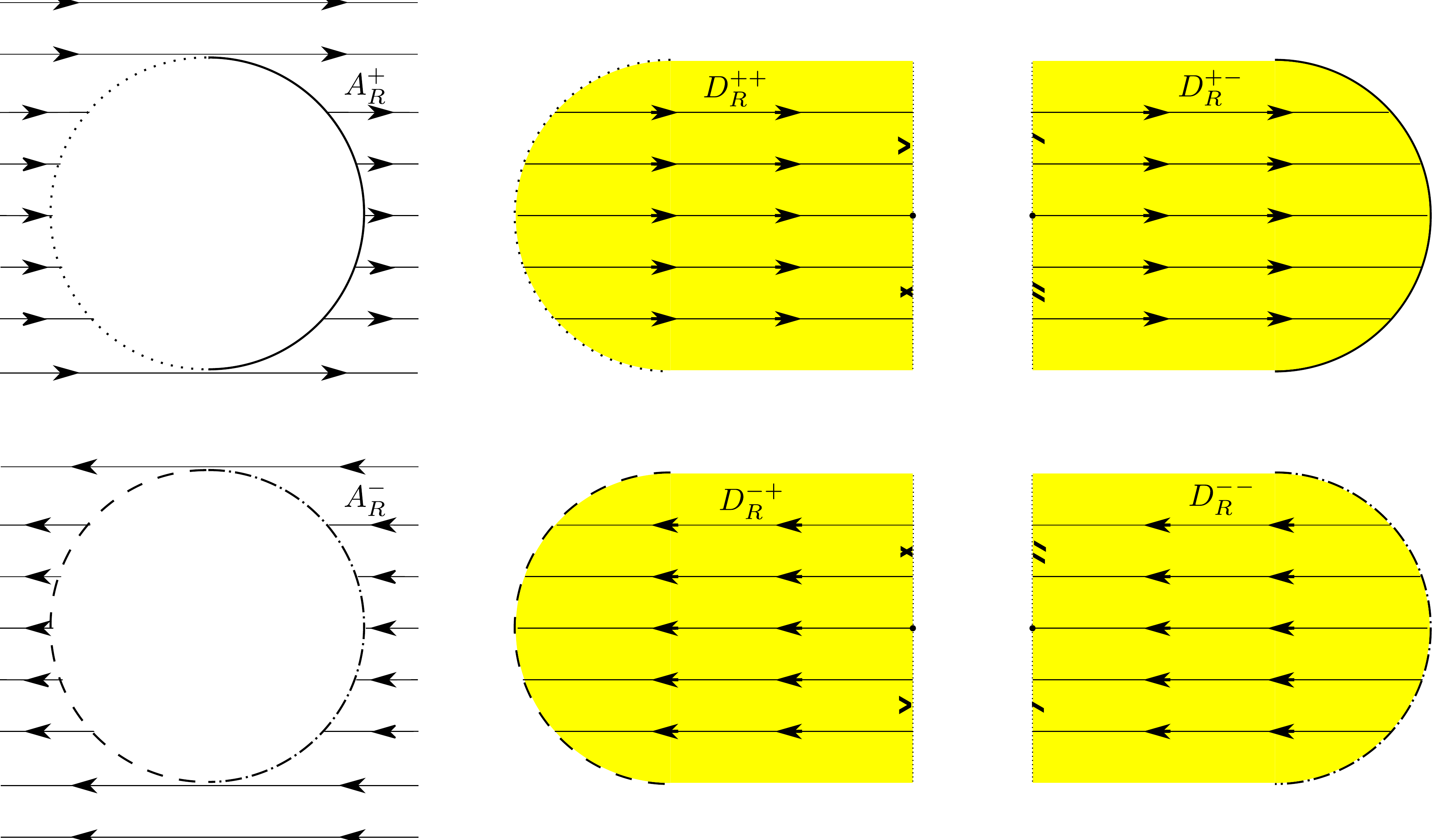}
\caption{Linearized Eaton lens flow in phase space}
\label{eaton_3}
\end{figure}
Moreover, the new coordinates give rise to a translation structure on the surface $\mathscr{P}_0$. Since the horizontal sides of $D_R^{\pm\pm}$ do not belong to $\mathscr{P}_0$, the surface is not closed.
However, we can complete the surface by adding the horizontal sides as in Figure~\ref{eaton_4}.
Let us denote the completed surface by $\overline{\mathscr{P}}_0$.
\begin{figure}[!htb]
\centering
\includegraphics[width=0.8\textwidth]{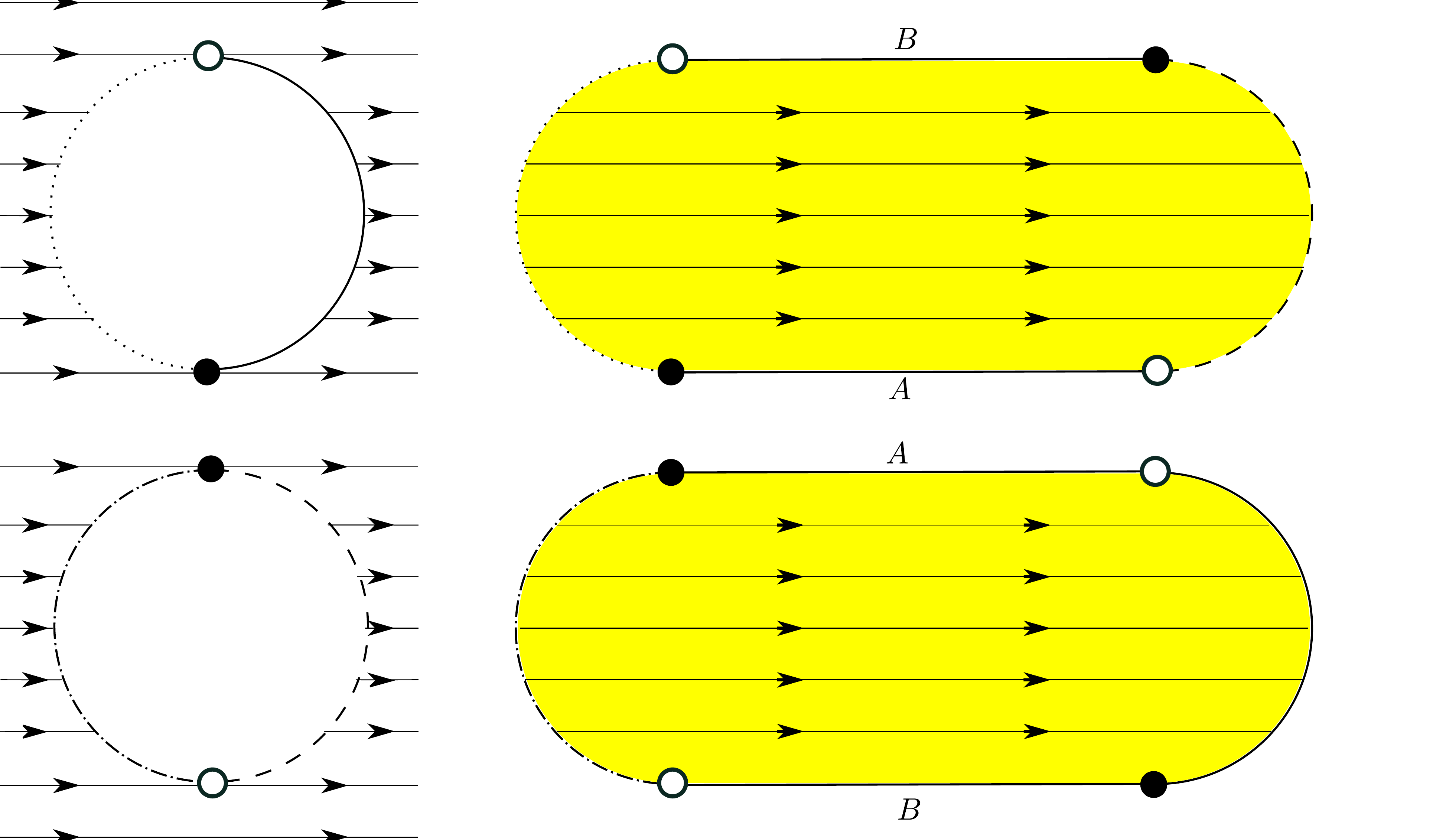}
\caption{The completed linearized phase space is a translation surface}
\label{eaton_4}
\end{figure}
It has two singular points with the cone angle $6\pi$ which are connected
by two horizontal saddle connections labeled by $A$ and $B$ in Figure~\ref{eaton_4}.
Moreover, the flow $(\mathfrak{g}_t)_{t\in\R}$ is measure-theoretically isomorphic to the horizontal translation flow on the translation surface $\overline{\mathscr{P}}_0$.

Let us consider an involution $\sigma:\overline{\mathscr{P}}_0\to \overline{\mathscr{P}}_0$ given by the translation between upper and lower parts of $\overline{\mathscr{P}}_0$ in  Figure~\ref{eaton_3}.
Then the quotient surface $\mathscr{Q}_0=\overline{\mathscr{P}}_0/<\sigma>$ is a half-translation surface.
It has two singular points (marked by circles) having cone angle $3\pi$
connected by a horizontal saddle  connection labeled by $A$ (and then continued as $A'$) and two poles (marked by squares), see Figure~\ref{eaton_quad}.
\begin{figure}[!htb]
\centering
\includegraphics[width=0.7\textwidth]{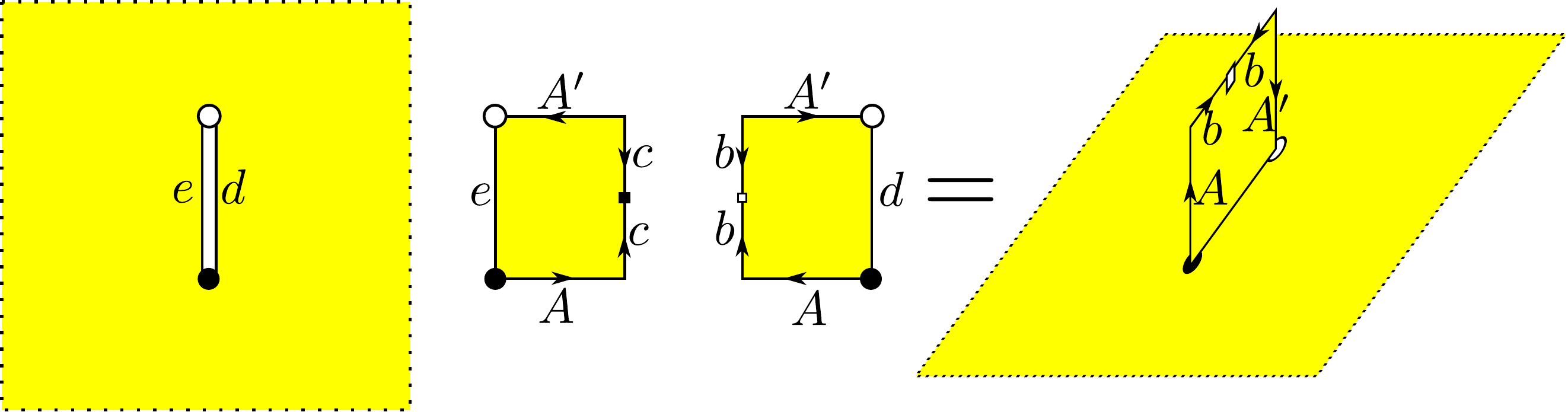}
\caption{The quadratic surface $\mathscr{Q}_0$}
\label{eaton_quad}
\end{figure}

If we consider an infinite pattern $\mathcal{L}$ of Eaton lenses on $\R^2$, then for every $\theta \in\R/\pi\Z$ we can similarly  represent the space $\mathscr{P}_{\mathcal{L},\theta}$ as a translation surface
which after a completion is a closed translation surface $\overline{\mathscr{P}}_{\mathcal{L},\theta}$. The translation  flow $(\varphi^{\mathcal{L},\theta}_t)_{t \in \R}$ on $\overline{\mathscr{P}}_{\mathcal{L},\theta}$ in the direction $\theta$ is measure-theoretically isomorphic to the flow
$(\mathfrak{g}^{\mathcal{L},\theta}_t)_{t\in\R}$. Moreover, the surface $\overline{\mathscr{P}}_{\mathcal{L},\theta}$ has an natural involution $\sigma$ which maps a unit vector to the vector at the same foot-point but oppositely directed. The quotient surface $\mathscr{Q}_{\mathcal{L},\theta}=\overline{\mathscr{P}}_{\mathcal{L},\theta}/<\sigma>$ is a half-translation surface that is the euclidian plane with a system of pockets each attached at the place of the corresponding lens. Each pocket is a rotated (by $\theta$) version of the pocket in Figure~\ref{eaton_quad}. Its length is equal to the diameter of the corresponding lens and is perpendicular to $\theta$. Most relevant for us, the ergodicity of measured foliation $\mathcal{F}^{\mathcal{L}}_{\theta}$ in the direction $\theta$ on $\mathscr{Q}_{\mathcal{L},\theta}$ is equivalent to the ergodicity of $(\varphi^{\mathcal{L},\theta}_t)_{t\in}$, and hence to the ergodicity of the flow $(\mathfrak{g}^{\mathcal{L},\theta}_t)_{t\in\R}$.

The measured foliation $\mathcal{F}^{\mathcal{L}}_{\theta}$ is Whitehead equivalent to the foliation $\mathcal{FL}^{\mathcal{L}}_{\theta}$ where each attached briefcase is replaced by the slit-fold stemming from the "flat lens" representation of the same Eaton lens in direction $\theta$, as in Figure~\ref{eaton_5}.
\begin{figure}[!htb]
\centering
\includegraphics[scale=0.3]{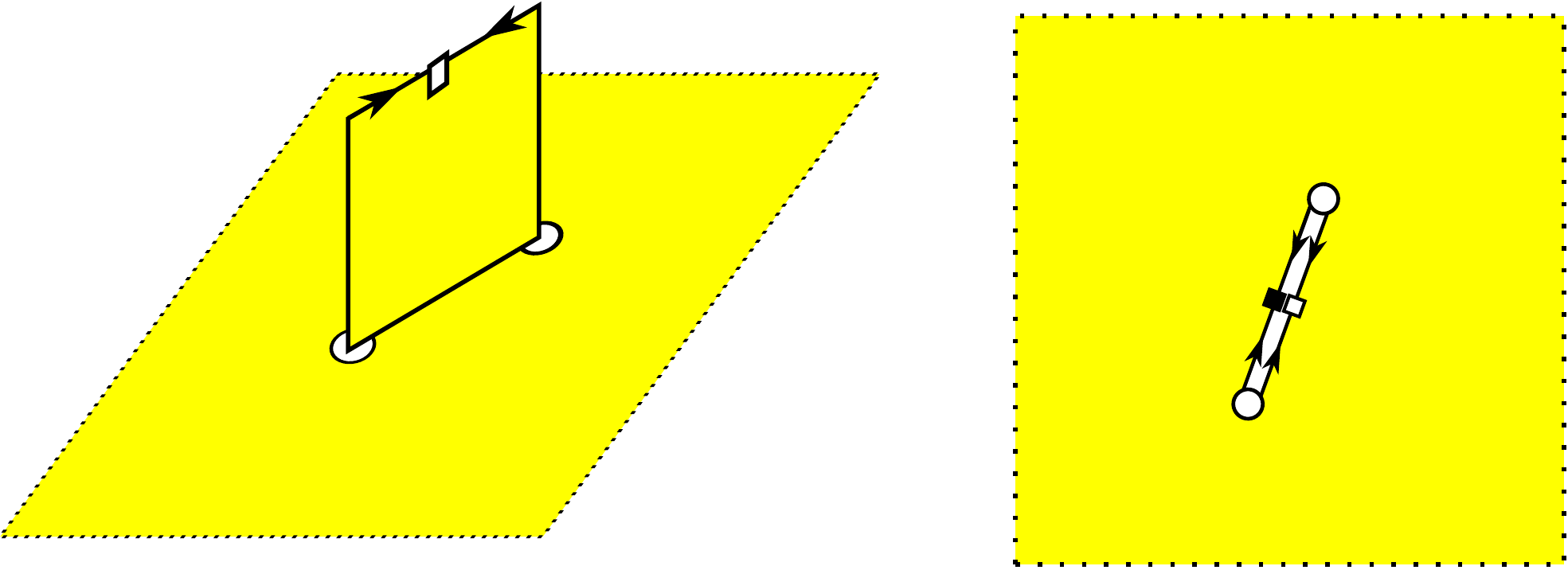}
\caption{The half-translation equivalent to an Eaton lens}
\label{eaton_5}
\end{figure}
In summary, instead of studying the ergodic properties of the geodesic flow  $(\mathfrak{g}^{\mathcal{L},\theta}_t)_{t\in\R}$ on the plane with a system of Eaton lenses
it suffices to pass to the measured foliation $\mathcal{FL}^{\mathcal{L}}_{\theta}$ where each Eaton lens is replaced by the corresponding flat lens of the same center and diameter as the lens attached perpendicular to $\theta$.

\section{Folds and Skeletons}\label{app:foldsskel}

In this section we describe examples of ergodic curves obtained from
other torus differentials. Starting with some of the quadratic differentials in our table
one obtains quadratic differentials on the plane that are not pre-Eaton differentials.
This section shows ways how to convert those into pre-Eaton differentials.
In particular, the quadratic differentials on the plane we deal with, have holes that need
to be removed. We model the holes by pillow-folds and then convert them to an
appropriate union of slit-folds.
\begin{figure}[!htb]
 \centering
 \resizebox{6cm}{!}{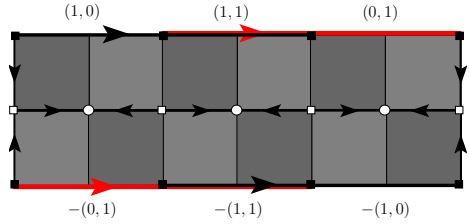}
 \caption{Homology generators and deck changes}
 \label{homology}
\end{figure}
\subsection{Skeleton representation for $\Z^2$-covers of $X_d(a,b)$ }
First we convert the standard polygonal representation of a pillowcase cover $X_d(a,b)$
into a pre-Eaton differential.
Recall from Section~\ref{sec:quad} that for $X_3(2,1)$, $X_4(2,1)$ and $X_6(3,1)$ this can
be done by a central cut followed by turning one half underneath the other.
\begin{figure}[!htb]
 \centering
 \resizebox{8cm}{!}{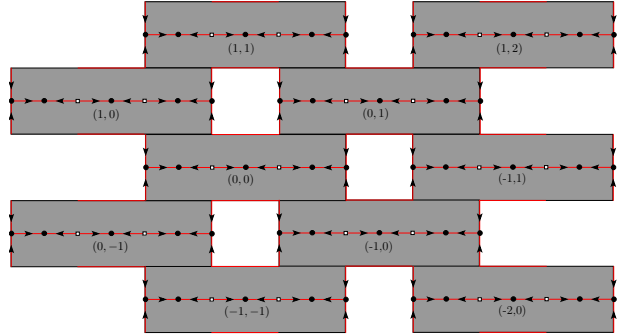}
 \caption{The universal homology cover $\widetilde{X}_6(3,1)$}
 \label{deg_6_decks}
\end{figure}
After the half-turn the absolute homology generators
are arranged as shown in Figure \ref{homology} for $X_6(3,1)$.
The arrangement for the absolute homology of  $X_3(2,1)$ looks similar after the half-turn:
The two homology generators overlap in the middle third of the rectangle representing the surface.
Then consider the universal cover $\widetilde{X}_6(3,1)\rightarrow X_6(3,1)$ determined by the pair
of homology generators.
Let us label the deck shifts as in Figure \ref{homology},
and the decks by $\Z^2$.
Then, starting at deck $(0,0)$ we reach deck $(1,0)$, once crossing the left third of the
rectangles upper edge and we reach deck $(0,1)$ when crossing the right third. We
enter deck $(1,1)$ when crossing the middle third of the upper edge and so forth.
The labeled tiles of Figure \ref{deg_6_decks}
show the cover. It has rectangular holes causing jumps of the directional dynamics
in the plane.  In particular the skeleton describing the quadratic differential contains
boundaries of the spared rectangles besides the slit-folds, see Figure \ref{deg_6_skeleton}.
\begin{figure}[!htb]
 \centering
 \resizebox{8cm}{!}{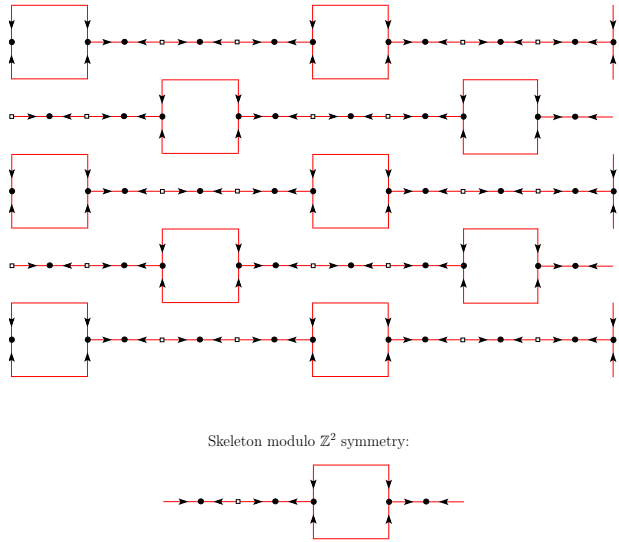}
 \caption{Skeleton representation of $\C_6(3,1)$}
 \label{deg_6_skeleton}
\end{figure}
Let us now forget the covering and just consider the skeleton on the plane.
While only the dynamics outside the spared rectangles was previously defined
we now extend the definition to the inside as follows: The folded parts of any rectangle
are genuine slit-folds and the translation identified edges are translation identified from the inside,
too. That way we obtain a quadratic differential on the whole plane that we denote by $\C_6(3,1)$.
The notation, a combination of the standard complex plane notation together with the weight notation
of the pillow case cover, will be used for other surfaces below.
The ``inside'' of  each rectangle is a pillow-case carrying invariant foliations.
The natural extension promotes an easy geometric definition of the foliation:
Given a direction $\theta \in \R/ \pi \Z$ consider the unoriented lines parallel to $\pm e^{i \theta}$
in $\C$. Then put a skeleton in the plane and identify the intersection points of the leaves with the skeleton
according to the respective rules, i.e.\ translation or central rotation.

The extension of the quadratic differential, i.e.\  $\widetilde{X}_6(3,1)$,  to the whole complex plane,
i.e.\ $\C_6(3,1)$, is a step towards realizing the skeleton by admissible Eaton lens configurations
on the plane. In fact this first step allows us to converts the ``outside'' quadratic differential
into a pre-Eaton differential, as shown in Figure \ref{mutation}.
Seen from the view point of Eaton lens dynamics the conversion removes
the jump of leaves over the rectangular gaps and replaces it by an
equivalent jumpfree dynamics.
The process shown in Figure \ref{mutation} performed backwards is a
railed deformation moving slit-folds through other slit-folds that changes their character:
A pair of slit-folds becomes a pair of translation identified lines.
At the end we merge singular points which is not
a railed deformation in the strict sense of the definition.

\begin{figure}[!htb]
 \centering
 \includegraphics[width=1\textwidth]{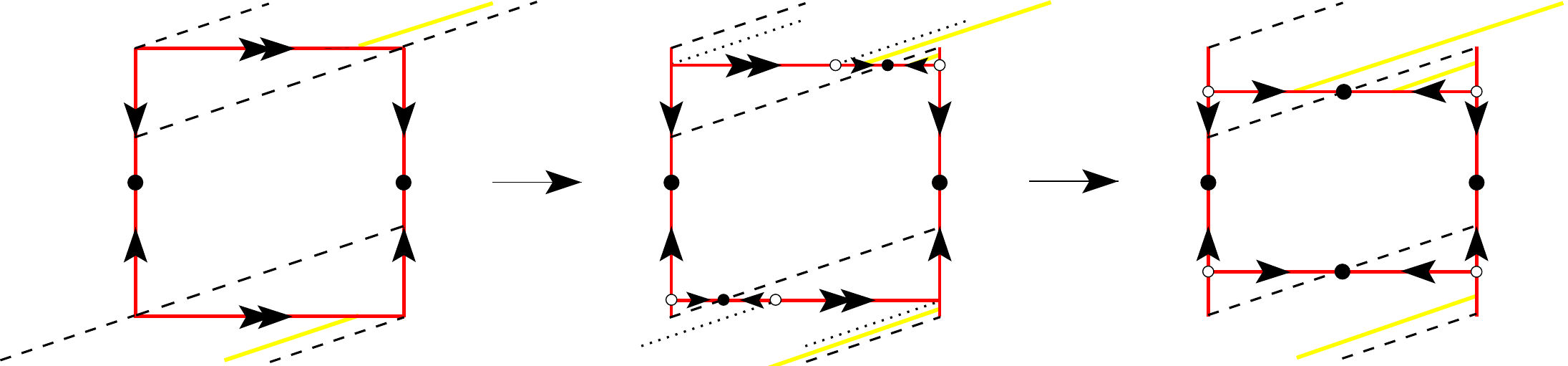}
 \caption{A railed deformation applied to both horizontal segments of a pillow-fold.
 The dotted lines indicate some leaves of the direction foliation.
 The endpoints of the line segments, marked by a black dot, move along a leaf.}
 \label{mutation}
\end{figure}

\subsection*{Pillow-folds and chip-folds}
All that can be done for rectangular folds can be done for folds built from a parallelogram.
So our definitions include parallelograms.

Let $[a,b]$ and $[c,d]$ are two non-parallel line segments in $\C$, so that $c \in [a,(b+a)/2]$.
By $[a,b] \boxtimes [c,d]$ we denote the union of segments $[a,b], [a+(d-c), b+(d-c)]$, $[c,d]$ and  $[b-(c-a), b-(c-a)+(d-c)]$. Replacing all line segments in $[a,b] \boxtimes [c,d]$ by slit-folds, we get a
what we call a {\em chip-fold} denoted by $\rangle a,b \langle \boxtimes \rangle c,d \langle$.
A {\em pillow fold} on the other hand is obtained by identifying two segments
parallel to, say $[c,d]$, with a translation parallel to $[a,b]$ and the two other segments
with slit-folds. Let us denote this fold by $\rangle a,b \langle \boxtimes | c, d |$.
Chip-folds are parts of the skeletons in Figures~\ref{Eaton_C_3(2,1)} and \ref{Eaton_C_6(3,1)}.
Chip-folds and their generalization are necessary to
replace the jumps, created by the translation identification in pillow-folds.

If $[a,b]$ and $[c,d]$ are line segments, so that $c \in [a,(b+a)/2]$,
then for $n\geq 2$ define  $\rangle a,b \langle \boxtimes \ n \cdot \rangle c,d \langle $
to be
\[ \left( \bigcup^{n-1}_{k=0}
\big(\rangle a,b \langle\ \cup \ \rangle c,d \langle\ \cup \ \rangle b+a-c,b+a+d-2c \langle\big)+k(d-c)\right)
\cup \ \big(\rangle a,b \langle + n(d-c)\big),
\]
this is the fold configuration with $n+1$ slit folds parallel to $[a,b]$.
 We call this object \emph{$n$-chip-fold}. In particular, a $1$-chip-fold is a chip-fold.
Analogously  $\rangle a,b \langle \boxtimes \ n\cdot | c,d |$ denotes the {\em $n$-pillow-fold} obtained
by replacing all slit-folds of an $n$-chip-fold that are parallel to $[c,d]$ by line segments. These line segments are identified by a translation in the direction of the vector $\overrightarrow{ab}$.
\begin{proposition}\label{transform}
Take a plane equipped with a single pillow-fold and consider
 a fixed direction foliation on the outside of a pillow-fold.
Then there is an $n$-chip-fold or a pair of parallel slit-folds which has an
outer measured foliation Whitehead equivalent (up to a finite number of leaves) to the given measured foliation.
\end{proposition}
\begin{proof}
Since the problem is invariant under affine transformations we can consider a pillow-fold
$  \rangle 0,ib \langle  \ \boxtimes \ |0,a|$ in the complex plane where
$a, b \in \R_+$, so the segment $[0, a]$ is horizontal, and $[0, ib]$ is vertical.
For fixed $\theta \in \R/\pi \Z$ consider the outer foliation for $  \rangle 0,ib \langle  \ \boxtimes \ |0,a|$. If $\theta = \pi/2$, we translate the two horizontal sides together at the center of the rectangle.
The resulting skeleton consists of two vertical slit-folds.
That is a  Whitehead move and so the outer foliations are equivalent.

Suppose $|\tan \theta| \leq \frac{b}{a}$. That is, the slope of the foliation is bounded by
the slope of the diagonal $[0, a+ib]$ of the rectangle $[0,ib] \times [0,a]$.
In this case translate the two horizontal edges of the rectangle parallel to the foliation towards its inside through the vertical slit-folds, so that every point on the edges remains on the same line
(including the slit-fold identification) of slope $\tan \theta$, as shown in Figure \ref{mutation}.
Note, that the two horizontal edges form a loop.
Two slit-folds appear unless $|\tan \theta| = \frac{b}{a}$. In that case, both slit-folds fall together and we regard it as a single slit-fold centered at the center of the rectangle. By construction both outer measured foliations differ by a Whitehead move that breaks up the singular point at the vertex of the pillow-fold, so they are equivalent.

For larger angles we need to use an intermediate step. In fact, if $\theta$ is
not covered by the previous case(s), then there is a minimal $n\geq 2$ such that $|\tan \theta| \leq n \frac{b}{a}$. Then change the given pillow-fold  $ \rangle 0,ib \langle \  \boxtimes \ |0,a|$ into an $n$-pillow-fold $\rangle 0,ib\langle \ \boxtimes \ n\cdot |0,a/n|$ by putting $n-1$ successive $a/n$ translates of the left vertical slit-fold into the rectangle. Then for each of the $n$ (translation equivalent) pillow-folds the previous conversion into a union of slit-folds applies. Here we may change finitely many leaves, the ones hitting the endpoints of the $n-1$ new slit-folds that are put into the
pillow case. Again we find a measurably equivalent outer foliation. Note, that the
inner foliation is changed by this procedure, but this is irrelevant for our claim.
\end{proof}
Let us call a skeleton in the plane {\em standard skeleton}, if it is a countable union of pillow-folds and
slit-folds, so that no pillow-fold contains other folds. For those we can use Proposition \ref{transform}
inductively to obtain:
\begin{corollary}
For any quadratic differential defined by a standard skeleton and any direction
$\theta \in \R/\pi\Z$ the outer measured foliation tangential to $\theta$
in the plane is up to countably many leaves Whitehead equivalent to the
direction foliation of a pre-Eaton-differential in the plane.
\end{corollary}
The skeletons we consider are special, they have exactly
one unbounded component. With the boundary identifications given
by the skeleton the unbounded component is homeomorphic to a plane.

\begin{figure}[!htb]
 \centering
 \includegraphics[scale=0.40]{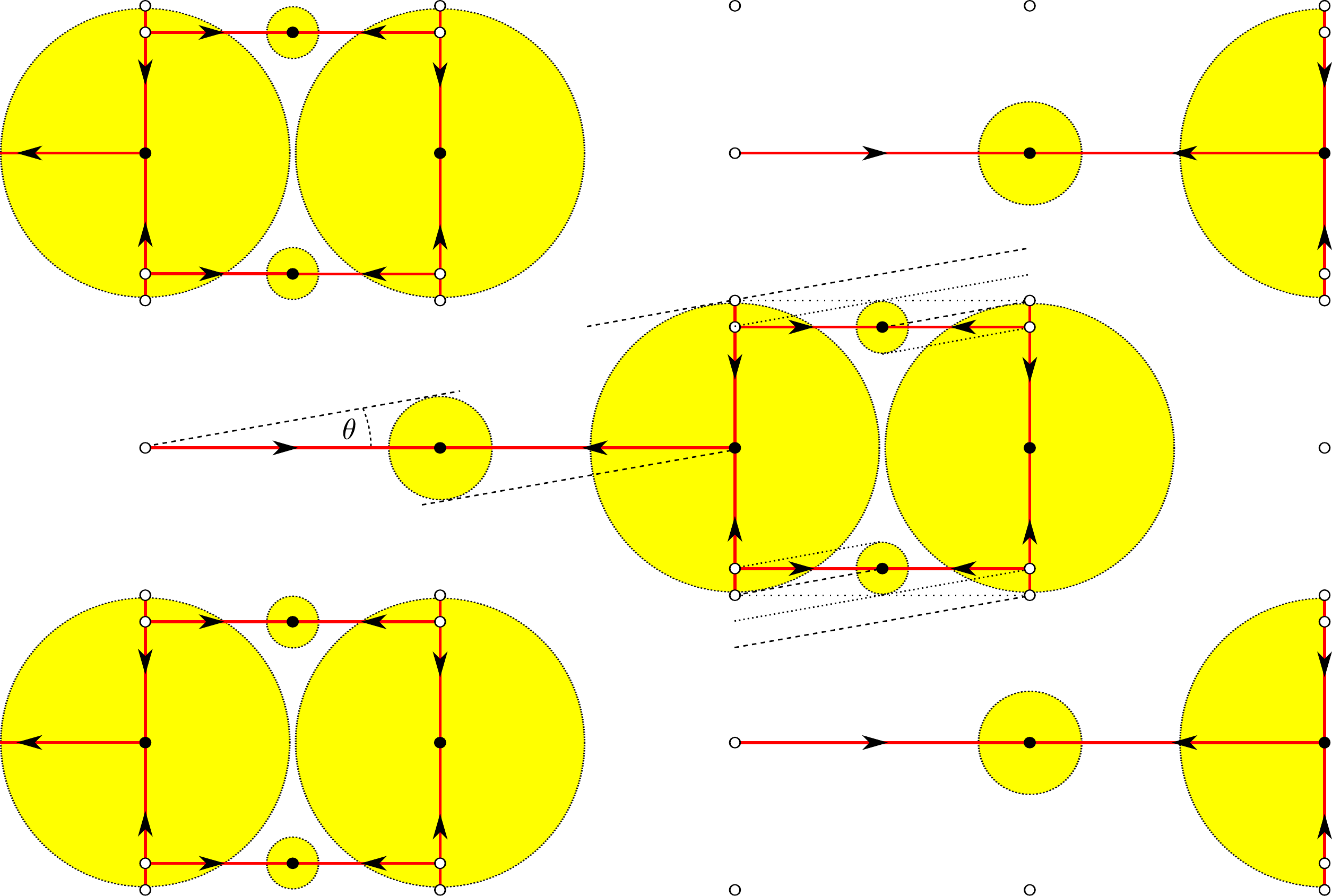}
 \caption{Eaton lens configuration for $\C_3(2,1)$}
 \label{Eaton_C_3(2,1)}
\end{figure}

\subsection{Other ergodic Eaton curves}
Using the $ X_3(2,1)$,  $ X_6(3,1)$ and $ X_6(3,2)$ torus differentials, we present more examples of admissible ergodic Eaton lens curves.
Skeletons of the torus differentials allow us to write down differentials on the plane
and represent them geometrically by arrow diagrams as in Figures \ref{Eaton_C_6(3,1)}, \ref{Eaton_C_6(3,2)}
and \ref{Eaton_C_3(2,1)}. In those particular cases all folds will be horizontal and vertical in cartesian coordinates.
Because the skeleton depends on the angle, see Proposition \ref{transform}, we only present the ergodic curve for small angles.
We do not give a formal proof of admissibility for those Eaton lens distributions,
it would go along the same lines as done in Proposition \ref{prop:adm} for the Wollmilchsau differential.
The figures give some clues how to work out the details, such as dividing tangent lines between some lenses.
\begin{figure}[!htb]
 \centering
 \includegraphics[scale=0.40]{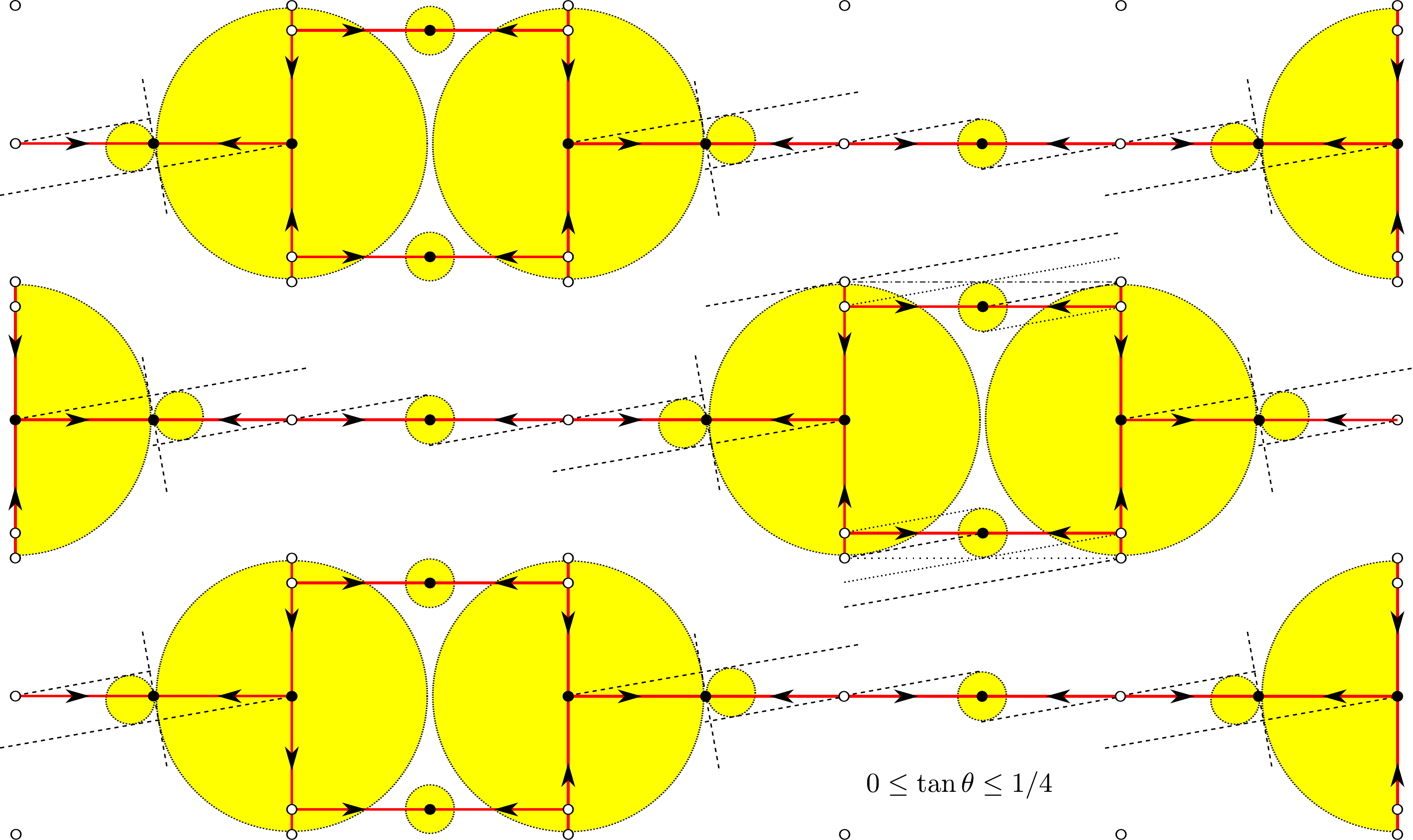}
 \caption{Eaton lens configuration for $\C_6(3,1)$}
\label{Eaton_C_6(3,1)}
\end{figure}
\begin{figure}[!htb]
 \centering
 \includegraphics[scale=0.22]{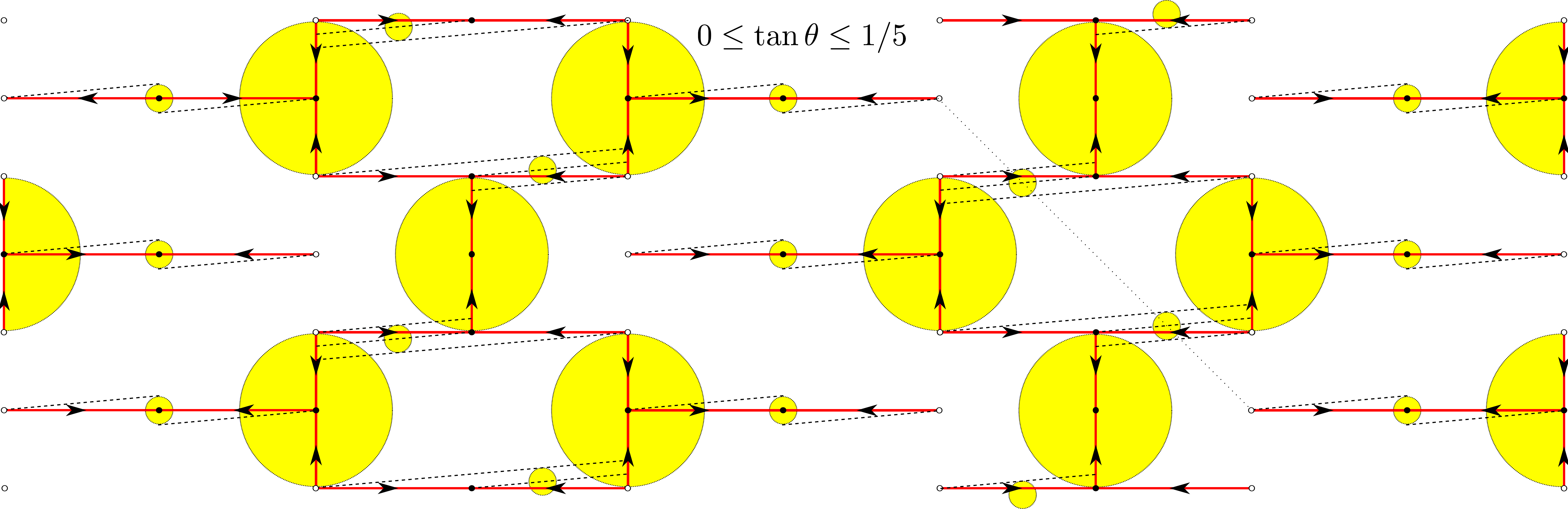}
 \caption{Eaton lens configuration for $\C_6(3,2)$}
 \label{Eaton_C_6(3,2)}
 \end{figure}

\end{document}